\documentclass[11pt]{article}

\usepackage[usenames]{color}

\usepackage[square,numbers]{natbib}
\usepackage{undertilde}  
\usepackage{latexsym,amsmath,amssymb,amsbsy, amsthm}

\usepackage[left=1.2truein, right=1.2truein, top = 1.2truein, bottom = 1.2truein]{geometry}

\usepackage[CJKbookmarks=true,
            bookmarksnumbered=true,
			bookmarksopen=true,
			colorlinks=true,
			citecolor=red,
			linkcolor=blue,
			anchorcolor=red,
			urlcolor=blue]{hyperref}
\usepackage{setspace}
\usepackage{lscape,fancyhdr,fancybox}
\usepackage{graphicx,epsfig, xy}
\usepackage{float}
\usepackage{graphicx}
\usepackage{caption}
\usepackage{subcaption}

\newcommand{\comment}[1]{}

\numberwithin{equation}{section}

\newtheorem{theorem}{Theorem}[section]
\newtheorem{lemma}{Lemma}[section]
\newtheorem{proposition}{Proposition}[section]

\newtheorem{prop}[theorem]{Proposition}

\numberwithin{equation}{section}

\DeclareMathOperator{\Var}{Var}

\theoremstyle{definition}
\newtheorem{definition}{Definition}[section]
\newtheorem{remark}{Remark}[section]

\usepackage{amsfonts}

\newcommand{\Cov}{\text{Cov}}

\usepackage{enumerate}

\DeclareMathOperator{\E}{E}

\DeclareMathOperator{\Tr}{Tr}

\newcommand{\nb}[1]{\textcolor{red}{\texttt{[#1]}}}
\newcommand{\ER}{{Erd\H{o}s--R\'{e}nyi}}
\newcommand{\FK}{{F{\"u}redi--Koml{\'o}s}}

\newcommand{\Acenone}{A_{\mathrm{cen1}}}
\newcommand{\Acentwo}{A_{\mathrm{cen2}}}
\newcommand{\pnav}{p_{n,\mathrm{av}}}
\newcommand{\hpnav}{\widehat{p}_{n,\mathrm{av}}}

\title{Optimal hypothesis testing for stochastic block models with growing degrees\footnote{We thank Edgar Dobriban for careful reading of an earlier version of the manuscript and for constructive comments.}}
\author{Debapratim Banerjee\footnote{Email: \texttt{dban@wharton.upenn.edu}.}~
and~
Zongming Ma\footnote{Email: \texttt{zongming@wharton.upenn.edu}.}\\
~\\
\textit{University of Pennsylvania}}
\date{August 10, 2017}
 
\begin{document}
\maketitle
\begin{abstract}
	 
The present paper considers testing an \ER~random graph model against a stochastic block model in the asymptotic regime where the average degree of the graph grows to infinity with the graph size $n$.
Our primary interest lies in those cases in which the signal-to-noise ratio is at a constant level.
Focusing on symmetric two block alternatives, we first derive joint central limit theorems for linear spectral statistics of power functions for properly rescaled graph adjacency matrices under both the null and local alternative hypotheses.
The powers in the linear spectral statistics can also grow to infinity together with the graph size.
In addition, we show that linear spectral statistics of Chebyshev polynomials are closely connected to signed cycles of growing lengths that determine the asymptotic likelihood ratio test for the hypothesis testing problem.
This enables us to construct a sequence of test statistics that achieves the exact optimal asymptotic power within $O(n^3\log n)$ time complexity in the contiguous regime when $n^2 \pnav^3 \to\infty$ where $\pnav$ is the average connection probability.
We further propose a class of adaptive tests that are computationally tractable and completely data-driven.
They achieve nontrivial powers in the contiguous regime and consistency in the singular regime whenever $n\pnav\to\infty$.
These tests remain powerful when the alternative becomes a stochastic block model with more than two blocks.
 
\medskip   

 
\end{abstract}


\section{Introduction}
Stochastic block model (SBM) \cite{holland83} is an active domain of modern research in statistics, computer science and many other related fields. 
A stochastic block model for random graphs encodes a community structure where a pair of nodes from the same community are expected to be connected in a different manner from those from different communities. 
This model, together with the related community detection problem, has drawn substantial attentions in statistics and machine learning.
Throughout the paper, let $\mathcal{G}_1(n,p_n)$ denote the \ER~graph with $n$ nodes in which the edges are i.i.d.~Bernoulli random variables with success probability $p_n$. 
For any integer $\kappa \geq 2$, let $\mathcal{G}_\kappa(n,p_n,q_n)$ denote the symmetric stochastic block model with $\kappa$ different blocks where the label $\sigma_u$ of any node $u$ is assigned independently and uniformly at random from the set $\{1,2,\dots, \kappa\}$. 
The edges are independent Bernoulli random variables, and 
two nodes are connected with probability $p_n$ if they share the same label and $q_n$ otherwise.

A fundamental question related to stochastic block models is community detection where one aims to recover the partition of nodes into communities based on one instance of the random graph.
Depending on the signal-to-noise ratio, there are three different regimes for recovery, namely partial recovery, almost exact recovery and exact recovery.
In the asymptotic regime of bounded degrees (i.e.~$np_n$ and $nq_n$ remain constants as $n\to\infty$), the seminal papers by Mossel et al.~\cite{MNS12, MNS13} 
and \citet{Mas14} established sharp threshold for $\mathcal{G}_2(n,p_n, q_n)$ on when it is possible and impossible to achieve a partial recovery of community labels that is strictly better than random guessing, which confirmed the conjecture in \citet{DKMZ11}.
See \cite{AS16} for an extension to multiple blocks and \cite{Ban16} for an extension to the regime of growing degrees (i.e.~$n p_n, nq_n\to\infty$ as $n\to\infty$). 
In the regime of growing degrees,
\citet{MNS15} established the necessary and sufficient condition for achieving almost exact recovery in $\mathcal{G}_2(n, p_n ,q_n)$, i.e.~when only a vanishing proportion of node labels are not recovered correctly.
See also \cite{abbe2015community,gao2015achieving,yezhang15,yun2015optimal,
gao2016community} for results on more general SBMs.
Furthermore, \citet{ABH14} and \citet{MNS15} established the necessary and sufficient condition for achieving exact recovery of labels for $\mathcal{G}_2(n,p_n,q_n)$ which was later extended by \cite{hajek2014achieving, hajek2015achieving,abbe2015community,jog2015information,yun2015optimal,gao2016community} to more general cases.
See \cite{abbe2017community} for a survey of some recent results.

In addition to the literature on information-theoretic limits, many community detection algorithms have been proposed, including but not limited to spectral clustering and likelihood based clustering.
An almost universal assumption of these algorithms is the knowledge of the number of blocks $\kappa$, which usually is unknown in practice. 
For data-driven choice of $\kappa$, researchers have proposed different methods.
One popular way is information criterion based model selection. See, e.g., \cite{daudin2008mixture,latouche2012variational,peixoto2013parsimonious,wang2015likelihood,saldana2017many}.
In addition, several block-wise cross-validation methods have been proposed and studied. See, e.g., \cite{chen2014network,dabbs2016}.
Furthermore, \citet{bickel2016} proposed to recursively apply the largest eigenvalue test for partitioning the nodes and for determining $\kappa$. 
The proposal was based on the GOE Tracy--Widom limit \cite{tracy1996orthogonal} of the largest eigenvalue distribution for adjacency matrices of \ER~graphs when the average degree grows linearly with $n$.
\citet{lei2016} extended it to a procedure based on sequential largest eigenvalue tests in the regime where exact recovery can be achieved.
See also \cite{le2015estimating} for another spectral method for choosing $\kappa$.

Let the observed adjacency matrix be $A\in \{0,1\}^{n\times n}$.
The major focus of the present paper is to test the following hypotheses:
\begin{equation}
	\label{test:sbm}
H_0 : A \sim \mathcal{G}_1\left(n,\,\frac{p_n+q_n}{2}\right)
\quad\mbox{vs.}\quad
H_1 : A \sim \mathcal{G}_2(n,p_n,q_n)
\end{equation}
when the average degree of the random graph grows to infinity with the graph size.
The parameters in the hypotheses are so chosen that the expected numbers of edges match under null and alternative.
Let $a_n = np_n$ and $b_n = nq_n$.
Our primary interest lies in the cases where the signal-to-noise ratio
\begin{equation}
	\label{eq:snr} 
	c := \frac{(a_n - b_n)^2}{a_n+b_n}  
\end{equation}
is a constant, and we call any such alternative a \emph{local} one.
For such cases, one has growing average degree if and only if $np_n\to\infty$.
In what follows, we denote the null and alternative hypotheses in \eqref{test:sbm} by $\mathbb{P}_{0,n}$ and $\mathbb{P}_{1,n}$ respectively.
This testing problem is not only fundamental to inference for SBMs but is also the foundation of any test based method for choosing $\kappa$.

For \eqref{test:sbm},
\citet{MNS12} (resp.~\citet{Ban16}) proved that when $a_n\equiv a$ and $b_n\equiv b$ are fixed constants (resp.~when $a_n\to\infty$ and ${a_n}/{n} = p_n \to p \in[0,1)$), 
if $c < 2$ (resp.~$c<2(1-p)$), then
the measures $\mathbb{P}_{0,n}$ and $\mathbb{P}_{1,n}$ are \emph{mutually contiguous} \cite{LeCam}, i.e.~for a sequence of events $E_n$, $\mathbb{P}_{0,n}(E_n) \to 0$ if and only if $\mathbb{P}_{1,n}(E_n) \to 0$. 
On the other hand, if $c>2$ (resp.~$c > 2(1-p)$), then $\mathbb{P}_{0,n}$ and $\mathbb{P}_{1,n}$ are \emph{asymptotically singular}.
These results imply that whenever $c<2$ (resp.~$c<2(1-p)$), it is impossible to find a consistent test for \eqref{test:sbm}.

In the respective asymptotic regimes,
\citet{MNS12} and \citet{Ban16} further obtained 
explicit descriptions of the asymptotic log-likelihood ratio 
within the contiguous regime.
Let $L_n := \frac{\mathrm{d}\mathbb{P}_{1,n}}{\mathrm{d}\mathbb{P}_{0,n}}$ be the likelihood ratio.
In the growing degree asymptotic regime, 
\citet[Proposition 3.4]{Ban16} showed that if $c < 2(1-p)$ where $p=\lim_{n\to\infty} p_n$, then
\begin{equation}
	\label{exp:likelihood}
\begin{aligned}
\log(L_n)  \,|\, \mathbb{P}_{0,n} & \stackrel{d}{\to} \sum_{i=3}^{\infty}\frac{2t^{i}Z_{i} - t^{2i}}{4i},
\quad
\log(L_n) \,|\, \mathbb{P}_{1,n} 
\stackrel{d}{\to} \sum_{i=3}^{\infty}\frac{2t^{i}Z_{i}+t^{2i}}{4i},
\end{aligned}
\end{equation}  
where 
\begin{equation}
	\label{eq:t}
t= \sqrt{\frac{c}{2(1-p)}} \quad \text{and} \quad
Z_{i} \stackrel{\text{ind}}{\sim} N(0,2i).	
\end{equation}
Each random variable $Z_i$ comes from the weak limit of the signed cycle of length $i$.
See \cite[Definition 4.1]{Ban16} and Eq.~\eqref{def:signed} below for the exact definition. 
Asymptotically the log-likelihood is a measurable function of the signed cycles. 
As a consequence, in the contiguous regime, knowing the signed cycles is enough for obtaining the asymptotically optimal test for (\ref{test:sbm}).
On the other hand, in the singular regime, one has a consistent test for \eqref{test:sbm} by using the signed cycle statistic the length $k_n$ of which tends to infinity with $n$ at a rate of $o(\min\{\log(np_n), \sqrt{\log n}\})$
\cite[Proposition 4.1]{Ban16}.
Here and after, for any two sequences of positive numbers $x_n$ and $y_n$, we write $x_n = O(y_n)$ if $x_n/y_n$ is uniformly bounded by a numeric constant and $x_n=o(y_n)$ or $x_n\ll y_n$ if $x_n/y_n$ converges to zero as $n\to\infty$.

These results are satisfying from a statistical optimality viewpoint because the Neyman--Pearson lemma dictates that the likelihood ratio test is optimal for the simple vs.~simple testing problem \eqref{test:sbm}. 
However, there are two major drawbacks.
First, neither the likelihood ratio test nor any test involving signed cycles of diverging lengths is computationally tractable.
In particular, evaluation of the likelihood function of the alternative is of exponential time complexity,
and calculating the signed cycle of length $k$ directly requires enumeration of all node subsets of size $k$ which is of $O({n\choose k})$ time complexity. 
It grows faster than any polynomial of $n$ as long as $k$ diverges with $n$.
In addition, to decide on which test statistic to use, one needs to know the null and alterative (or at least the value of $t$ in \eqref{eq:t}), and so one does not yet have a test procedure that is adaptive to different null and alternative hypotheses.

In view of the foregoing shortcomings, we pursue answers to the following two questions in the present paper:
\begin{itemize}
\item Can one achieve the sharp asymptotic optimal power of the likelihood ratio test in the contiguous regime with a test of polynomial time complexity?

\item Can one design an adaptive test which achieves nontrivial power in the contiguous regime and consistency in the singular regime?
\end{itemize}
 
\subsection{Main contributions}
In this paper, we provide affirmative answers to both of the foregoing questions under appropriate conditions, which are summarized as the following main contributions:
\begin{enumerate}
\item For appropriately rescaled graph adjacency matrices,
we derive joint central limit theorems for their linear spectral statistics (LSSs) of power functions under both the null and local alternative hypotheses in the growing degree asymptotic regime. 
An important feature of the central limit theorems is that we allow the powers in LSSs to grow to infinity with the graph size.
The proof of these CLTs based on the ideas of F{\"u}redi--Koml{\'o}s enumeration and unicyclic graphs further reveals a deep connection between LSSs of Chebyshev polynomials and signed cycles.

\item Based on the connection between the spectrum of an adjacency matrix and signed cycles, given the knowledge of both hypotheses in \eqref{test:sbm} (or the quantity $t$ in \eqref{eq:t}), we propose a test based on a special linear spectral statistic.
The test statistic can be evaluated within $O(n^3 \log{n})$ time complexity and achieves the sharp optimal asymptotic power as the likelihood ratio test in the contiguous regime under the additional condition that $n^2 p_n^3 \to\infty$. 
If only $np_n\to\infty$ holds, we have a slightly different test with $O(n^3\log n)$ time complexity that achieves a nontrivial fraction of the optimal power in the contiguous regime.
It is worth noting that
regardless of the rate at which $p_n$ scales with $n$, no community detection method can perform better than random guessing within the contiguous regime. In other words, we can only tell with nontrivial probability that the random graph comes from $\mathcal{G}_2(n,p_n,q_n)$ while having little idea about how the nodes are partitioned into communities.
Based on our limited knowledge, the present paper is one of the first to achieve the exact asymptotic power of the likelihood ratio test on a non-Gaussian model in high dimensions.

\item Further exploiting the connection between LSSs and signed cycles, we propose several adaptive tests for \eqref{test:sbm}.
These tests are data-driven, of $O(n^3\log n)$ time complexity and do not require knowledge of $p_n$ or $q_n$.
Moreover, they achieve nontrivial power in the contiguous regime and consistency in the singular regime. 
We also show that they remain consistent when the alternative is some symmetric SBM with $\kappa > 2$ blocks and the separation is above the Kesten--Stigum threshold \cite{DKMZ11}.

\end{enumerate}

\subsection{Related works}
The present work is closely related to a large body of work on the edge scaling limits of \ER~graphs, and more generally of Wigner matrices. 
Consider first the case where $p_n \to p >0$ and so the average degree of the graph grows linearly with the graph size $n$.
In this case, the rescaled graph adjacency matrix (s.t.~the entries are i.i.d.~random variables with mean zero and variance $1/n$) is a Wigner matrix under the null hypothesis in \eqref{test:sbm}, and can essentially be viewed as a rank-one perturbation of a Wigner matrix under the alternative with operator norm of the perturbation given by $t$ in \eqref{eq:t}.
It is well known that the largest eigenvalue of a Wigner matrix converges to $2$ almost surely and have the GOE Tracy--Widom scaling limit under a fourth moment condition. 
Assuming the perturbation is positive semi-definite (corresponding to $p_n>q_n$ in \eqref{test:sbm}), the largest eigenvalue of a rank-one deformed Wigner matrix undergoes a phase transition. 
In particular, it converges to $2$ or $t+{1}/{t}$ depending on whether $t<1$ or $t>1$. 
In addition, it has a GOE Tracy--Widom limit when $t<1$ and a non-universal scaling limit when $t>1$. 
See for instance \cite{PF07,CDF09,PRS11} for more details. 
Note that the threshold $t=1$ is exactly the threshold between the contiguous and the singular regimes for the null and alternative hypotheses in \eqref{test:sbm}. 
The phase transition thus suggests that any test based solely on the largest eigenvalue has trivial power within the contiguous regime for having the same scaling limit under null and alternative. 
A result of this flavor was first discovered by \citet{baik2005phase} for complex sample covariance matrices. 

When $p_n$ and $q_n\to 0$, the average degree of the graph grows sub-linearly with the graph size. 
In this regime, results about the edge scaling limits under either the null or the alternative are less complete compared with the linear degree growth regime.
Under the null, the convergence to the GOE Tracy--Widom limit was established in \cite{EKYY12} under the assumption that $np_{n}^{3} \to \infty$. \citet{LS16} weakened the condition to $n^2p_{n}^{3} \to \infty$. 
Turning to the alternative. 
Suppose $p_n>q_n$.
\citet{EKYY13} showed that the largest eigenvalue of the rescaled graph adjacency matrix 
converges in probability to $t+1/t$ whenever $np_n \gg (\log n)^{6\xi}$
for some $1 < \xi = O(\log \log n)$ and $t>1$. Further, it was proved in \cite{EKYY13} when $t> C_{0} (\log n)^{2\xi}$ for some large constant $C_{0}$,
the largest eigenvalue has a $\sqrt{2/n}$ fluctuation and a normal scaling limit.
To the best of our knowledge, little is known about the asymptotic null distribution of the largest eigenvalue when $n^2 p_n^3$ is bounded or about its distribution under any local alternative when $t$ is a constant.

As discussed earlier, one of the main contributions of the present paper is to link signed cycles and linear spectral statistics. 
In that sense analyzing linear spectral statistics of the rescaled adjacency matrices of \ER~graphs lies at the heart of our technical analysis. 
There are a series of papers on linear spectral statistics of Wigner and Wishart matrices relying on the methods introduced by \citet{bai2004clt}. 
See, in particular, \cite{bai2009clt} for CLTs of linear spectral statistics of Wigner matrices. These techniques are however specific to the asymptotic regime where the average degree grows linearly with graph size. 
In this paper we adopt the combinatorial methods developed by \citet{AZ05}
which we modify and use for all growing degree cases, regardless of the growth rate.  

In addition, our results are connected with the literature on optimal hypothesis testing in high dimensions.
Onatski et al.~\cite{onatski13, onatski14} studied the optimal tests for an identical covariance (or correlation) matrix against a spiked local alternative for Gaussian data when the sample size and the ambient dimension grow proportionally to infinity. 
Remarkably, they further studied the asymptotic powers of the Gaussian likelihood ratio tests for non-Gaussian data.
\citet{johnstone14} and \citet{johnstone15} studied analogous questions for an exhaustive collection of testing problems in various ``double Wishart'' scenarios where the sufficient statistics of the observations are essentially two independent Wishart matrices.
See also \citet{dobriban2016} for an important extension of \cite{onatski13, onatski14} where one is allowed to have general covariance matrices as the null model.
From a slightly different viewpoint, \citet{cai2013optimal} and \citet{cai2015optimal} studied minimax optimal hypothesis testing for an identity covariance matrix.
The concurrent work by \citet{gao2017testing} studied minimax rates for testing \ER~graph against SBMs and more general alternatives.
Interestingly, one of their proposed test statistics is asymptotically equivalent to the signed cycle of length three, also known as the signed triangle \cite{BDER14}.


\subsection{Organization}
The rest of the paper is organized as the following.
After a brief introduction of definitions and notation in Section \ref{sec:notation}, we establish in Section \ref{sec:lss} joint central limit theorems under both the null and local alternatives for linear spectral statistics of rescaled graph adjacency matrices and their connection to signed cycles.
In addition, we propose a computationally tractable testing procedure based on these findings that achieves the same optimal asymptotic power as the likelihood ratio test.
Section \ref{sec:adaptive} investigates adaptive testing procedures for \eqref{test:sbm} and we also study the powers of the proposed tests under symmetric multi-block alternatives.
We give an outline of the proofs in Section \ref{sec:overview} and the detailed proofs are presented in the appendix.
Finally, we conclude in Section \ref{sec:conclusion}.


\section{Definitions and notation}
\label{sec:notation}
We first introduce some preliminary definitions and notation to be used throughout the paper.
We let $\E_{i,n}$ and $\Var_{i,n}$ denote expectation and variance under $\mathbb{P}_{i,n}$ for $i=0$ and $1$.
For any random graph $G$, its adjacency matrix will be denoted by $A$ and $x_{i,j}$ (instead of $a_{i,j}$) will be used to denote the indicator random variable corresponding to an edge between the nodes $i$ and $j$. 
We denote the expected average connection probability and its sample counterpart by
\begin{equation}\label{def:hatp}
\pnav =\frac{1}{n(n-1)} \sum_{i \neq j} \E_{0,n}[x_{i,j}],
\quad \text{and} \quad
\hpnav = \frac{1}{n(n-1)} \sum_{i \neq j} x_{i,j}.
\end{equation}
Under out settings, $\pnav$ remains unchanged if we replace $\E_{0,n}$ with $\E_{1,n}$ in its definition.
The signed cycle of length $k$ of the graph $G$ is defined to be 
\begin{equation}\label{def:signed}
C_{n,k}(G)= \left(\frac{1}{\sqrt{np_{n,\mathrm{av}}(1-p_{n,\mathrm{av}})}}\right)^{k} \sum_{i_0,i_1,\ldots,i_{k-1}} (x_{i_0,i_1}-p_{n,\mathrm{av}})\ldots(x_{i_{k-1}i_0}-p_{n,\mathrm{av}})
\end{equation}
where $i_0,i_1,\dots, i_{k-1}$ are all distinct.
We define the following centered and scaled versions of the adjacency matrix $A$. 
For any $n\in \mathbb{N}$, let $1_n = (1,\dots,1)'\in \mathbb{R}^n$
and $I_n$ be the $n \times n$ identity matrix. Then
\begin{equation}\label{cen1}
\Acenone := \frac{A- p_{n,\mathrm{av}}(1_n 1_n'- I_{n})}{\sqrt{np_{n,\mathrm{av}}(1-p_{n,\mathrm{av}})}}\, ,
\end{equation}
and
\begin{equation}\label{cen2}
\Acentwo := \frac{A- \hpnav(1_n 1_n'- I_{n})}{\sqrt{n\hpnav(1-\hpnav)}}\,.
\end{equation}
Note that $\Acentwo$ is completely data-driven.
If $A$ is a random instance of the \ER~graph $\mathcal{G}_1(n,\pnav)$, then $\Acenone$ has zeros on the diagonal and the sub-diagonal entries (subject to symmetry) are i.i.d.~with mean zero and variance $1/n$.

We now introduce an important generating function. 
Given any $r \in \mathbb{N}$, let
\begin{equation}\label{gen:f}
\left( \frac{1-\sqrt{1-4z^2}}{2z} \right)^r= \sum_{m=r}^{\infty} f(m,r) z^{m}.
\end{equation}
The coefficients $f(m,r)$'s are key quantities for defining the variances and covariances of linear spectral statistics constructed from different power functions. 
For any $k \in \mathbb{N}$ denote 
\begin{equation}\label{def:catalan}
\psi_{k}=\left\{  
\begin{array}{ll}
0 & \text{if $k$ is odd}\\
\frac{1}{\frac{k}{2}+1} \binom{k}{\frac{k}{2}} & \text{if $k$ is even}.
\end{array}
\right.
\end{equation}
So $\psi_k$ is the $\frac{k}{2}$-th Catalan number for every even $k$.
Finally, we define a set of rescaled Chebyshev polynomials. 
These polynomials are important for drawing the connection between signed cycles and the spectrum of adjacency matrix. 
The standard Chebyshev polynomial of degree $m$ is denoted by $S_{m}(x)$ and can be defined by the identity
\begin{equation}\label{def:ChebyshevI}
S_{m}\left(\cos(\theta)\right) = \cos(m\theta).
\end{equation}
In this paper we use a slight variant of $S_m$, denoted by $P_{m}$ and defined as
\begin{equation}\label{def:ChebyshevII}
P_{m}(x)= 2S_{m}\left(\frac{x}{2}\right).
\end{equation}
In particular, 
$P_{m}(2\cos(\theta)) = 2 \cos(m\theta)$.
It is easy to note that 
$P_{m}\left( z+ z^{-1}\right)= z^{m} + z^{-m}$ for all $z \in \mathbb{C}$.
One also notes that $P_{m}(\cdot)$ is even and odd whenever $m$ is even or odd respectively.

Throughout the paper, we use $C, C_1, C_2, \dots$ 
to denote positive numeric constants and their values may value from occurrence to occurrence.
For any matrix (and vector) $U$, $U'$ stands for its transpose.


\section{Linear spectral statistics and likelihood ratio tests}
\label{sec:lss}

\subsection{Joint CLTs for LSSs of power functions}

We first characterize the asymptotic joint normality of linear spectral statistics of the form
$\sum_{i=1}^n g(\lambda_i)$
where $\lambda_1\geq \cdots \geq \lambda_n$ are ordered eigenvalues of either $\Acenone$ in \eqref{cen1} or $\Acentwo$ in \eqref{cen2} and $g(\lambda) = \lambda^k$ for some integer $k\geq 2$.
For convenience, we often write the statistic as $\Tr(A_{\mathrm{cen}i}^k)$ for $i=1,2$.
In what follows, we separate the discussion under the null and the alternative hypotheses in \eqref{test:sbm}.

\subsubsection{Results under the null}
Recall the definition of $f(m,r)$ in \eqref{gen:f}.
For any $1 \leq k_1< \cdots < k_l$, define an $l\times l$ symmetric matrix $\Sigma_{2k_1+1,\dots, 2k_l + 1}$ by setting its $(i,j)$-th entry as
\begin{equation}\label{cov:null:odd}
\Sigma_{2k_1+1,\ldots, 2k_l+1}(i,j)= \sum_{r=3: \text{$r$ odd}}^{\min(2k_i+1,2k_j+1)} 2f(2k_i+1,r)f(2k_j+1,r)\frac{(2k_i+1)(2k_j+1)}{r}.
\end{equation}
In addition, define a second $l\times l$ symmetric matrix $\widetilde{\Sigma}_{2k_1,\dots, 2k_l}$ by setting its $(i,j)$-th entry as
\begin{equation}\label{cov:null:even}
\begin{split}
\widetilde{\Sigma}(2k_1,\ldots,2k_l)(i,j) & = \sum_{r=4: \text{$r$ even}}^{\min(2k_i,2k_j)} 2f(2k_i,r)f(2k_j,r)\frac{(2k_i)(2k_j)}{r}
\\ & ~~~
+ 2(k_ik_j\psi_{2k_i}\psi_{2k_j})\lim_{n \to \infty}\frac{\Var_{0,n}\left[(x_{1,2}-\E_{0,n}[x_{1,2}])^{2}\right]}{\pnav^2(1-\pnav)^2}.
\end{split}
\end{equation}
Here and after, we may omit the subscripts in variance and expectation when there is no ambiguity.
With the foregoing definitions, we have the following results under $\mathbb{P}_{0,n}$.

\begin{theorem}
	\label{thm:null}
Suppose $A \sim \mathbb{P}_{0,n}$ and $n \pnav \to \infty$. 
For any fixed $l\geq 1$, we have:
\begin{enumerate}[(i)]
\item If $1\le k_1 < \ldots< k_l=o(\log(np_{n,\mathrm{av}}))$, 
then for $\Sigma = \Sigma_{2k_1+1,\ldots, 2k_l+1}$ defined in \eqref{cov:null:odd}
\begin{equation}\label{dist:null:oddpower}
\Sigma^{-\frac{1}{2}}\left(\Tr(A_{\mathrm{cen1}}^{2k_1+1}),\ldots, \Tr(A_{\mathrm{cen1}}^{2k_l+1}) \right)^{'} \stackrel{d}{\to} N_{l}(0,I_l).
\end{equation} 

\item Suppose $1\le k_1 < \ldots< k_l=o(\log(np_{n,\mathrm{av}}))$.
If $\pnav \to 0$,
\begin{equation}\label{dist:null:evenpowerptozero}
\frac{\sqrt{\pnav}\left(\beta_{2k_i}\right)}{\sqrt{2}k_i\psi_{2k_i}} \stackrel{d}{\to} N(0,1)
\end{equation}
where 
\begin{equation*}
\beta_{2k_i}= 
\Tr(A_{\mathrm{cen1}}^{2k_i})- \E_{0,n} (\Tr(A_{\mathrm{cen1}}^{2k_i}) ) .
\end{equation*}
Further,
$\mathrm{Cov}\left(\beta_{2k_i},\beta_{2k_j}\right) \to 2(k_ik_j\psi_{2k_i}\psi_{2k_j})$.
In other words, when $\pnav \to 0$, asymptotically the even moments are constant multiples of each other after rescaling.

If $\pnav \to p \in (0,1)$, 
then for $\widetilde{\Sigma} = \widetilde{\Sigma}_{2k_1,\ldots, 2k_l}$ defined in \eqref{cov:null:even}
\begin{equation}\label{dist:null:evenpowerptop}
(\widetilde{\Sigma})^{-\frac{1}{2}}\left(\beta_{2k_1},\ldots, \beta_{2k_l} \right)^{'} \stackrel{d}{\to} N_{l}(0,I_l).
\end{equation}

\item For any $k_i= o(\log(np_{n,\mathrm{av}}))$,
\begin{equation}\label{formula1}
{\Tr(A_{\mathrm{cen1}}^{2k_i+1})- \sum_{r=3: \text{$r$ odd}}^{2k_i+1} f(2k_i+1,r)\frac{2k_i+1}{r} C_{n,r}(G)} \stackrel{p}{\to} 0.
\end{equation}

\item If $1\le k_1 < \ldots< k_l=o\left(\min\left(\log(np_{n,\mathrm{av}}), \sqrt{\log{n}}\right)\right)$, results in (i),(ii) and (iii) continue to hold 
when $A_{\mathrm{cen1}}$ is replaced with $A_{\mathrm{cen2}}$.
\end{enumerate}
\end{theorem}

For finite $k$, parts (i) and (ii) in Theorem \ref{thm:null} are known in the literature.
See, for instance, \citet{AZ05} or \citet[pp.30-35]{AGZ} for reference. 
In particular, the variance expression given in \cite[Equation (2.1.44)]{AGZ}  matches with \eqref{cov:null:even} and differs from \eqref{cov:null:odd} at only the term involving $\E[Y_1^2]$. 
This term comes from the diagonal entries which are zeros in our case.
The significance of (i) and (ii) in Theorem \ref{thm:null} (and Theorem \ref{thm:alt} below) lies in the fact that the CLTs continue to hold when the powers grow to infinity with the graph size $n$. 

\subsubsection{Results under local alternatives}

Recall the definition of $c$ and $t$ in \eqref{eq:snr} and \eqref{eq:t}.
Our next result gives the counterpart of Theorem \ref{thm:null} under any local alternative where $c$ and hence $t$ are finite.

\begin{theorem}
	\label{thm:alt}
Suppose that $A \sim \mathbb{P}_{1,n}$ and that as $n\to\infty$, $n \pnav \to \infty$ while $t$ in \eqref{eq:t} remains a constant. 
For any fixed $l\geq 1$, we have:
\begin{enumerate}[(i)]
\item If $1\le k_1 < \ldots< k_l=o (\min(\log(np_{n,\mathrm{av}}), \sqrt{\log{n}} ))$, 
then for $\Sigma = \Sigma_{2k_1+1,\ldots, 2k_l+1}$ in \eqref{cov:null:odd}
\begin{equation}\label{dist:null:alt}
\Sigma^{-\frac{1}{2}}\big(\Tr(A_{\mathrm{cen1}}^{2k_1+1})-\nu_{2k_1+1},\ldots, \Tr(A_{\mathrm{cen1}}^{2k_l+1})-\nu_{2k_l+1} \big)^{'} \stackrel{d}{\to} N_{l}(0,I_l) 
\end{equation} 
where if $p_n > q_n$, for $\mu_r = t^r$, $r=1,2,\dots$,
\[
\nu_{2k_i+1}= \sum_{r=3: \text{$r$ odd}}^{2k_i+1} f(2k_i+1,r)\frac{2k_i+1}{r} \mu_{r}.
\]
If $p_n<q_n$, we set $\mu_r = (-t)^r$ for all $r \geq 1$.

\item 
If $1\le k_1 < \ldots< k_l=o\left(\min\left(\log(np_{n,\mathrm{av}}), \sqrt{\log{n}}\right)\right)$, then \eqref{dist:null:evenpowerptozero} (resp.~\eqref{dist:null:evenpowerptop}) continues to hold when $p_{n,\mathrm{av}}\to 0$ (resp.~when $p_{n,\mathrm{av}}\to p$) where the expectation in the definition of $\beta_{2k_i}$ is now taken under $\mathbb{P}_{1,n}$ while the definition of $\widetilde{\Sigma}$ remains unchanged.

\item For any $k_i= o\left(\min\left(\log(np_{n,\mathrm{av}}), \sqrt{\log{n}}\right)\right)$, \eqref{formula1} continues to hold.

\item If $1\le k_1 < \ldots< k_l=o\left(\min\left(\log(np_{n,\mathrm{av}}), \sqrt{\log{n}}\right)\right)$, results in (i),(ii) and (iii) continue to hold 
when $\Acenone$ is replaced with $\Acentwo$.
\end{enumerate}
\end{theorem}

Theorem \ref{thm:null}, Proposition 3.4 in \cite{Ban16} and Le Cam's Third Lemma \cite{LeCam} jointly imply claims (i) and (iv) in Theorem \ref{thm:alt} within the contiguous regime, i.e., when $c < 2(1-p)$ or equivalently $t<1$.
The significance of Theorem \ref{thm:alt} is that the CLTs continue to hold in the singular regime as long as $t$ is finite. It is not implied by Le Cam's Third Lemma and requires a dedicated proof.

When $\pnav \to 0$, 
for traces of even powers of $\Acenone$ and $\Acentwo$,
the second term on the right side of (\ref{cov:null:even}) dominates. This explains the result in (\ref{dist:null:evenpowerptozero}).
Indeed, one can further show that  $\frac{\sqrt{\pnav}\beta_{2k_i}}{\sqrt{2}k_i\psi_{2k_i}}$ is asymptotically the same as a rescaled version of the average degree of the graph when $c$ (and hence $t$) is finite. 
On the other hand, it is more complicated to state the counterpart of claim (iii) in both Theorem \ref{thm:null} and Theorem \ref{thm:alt} for traces of even powers and signed cycles of even lengths, which is our next focus.

\subsubsection{Connection between traces of even powers and even signed cycles}

Fix any integer $k\geq 2$.
We first decompose 
$\Tr(A_{\mathrm{cen1}}^{2k})$ as
\begin{equation}\label{eq:T2k}
\Tr(A_{\mathrm{cen1}}^{2k})= n\psi_{2k} - \binom{k+1}{2}\psi_{2k} + R_{1,2k} + R_{2,2k} + T_{2k},
\end{equation} 
where
\begin{equation}
	\label{eq:R}
\begin{aligned}
R_{1,2k} & := k\psi_{2k}\left[ \Tr\left( A_{\mathrm{cen1}}^{2}\right)- n + 1\right],
\,\,
R_{2,2k} 
:= \sum_{r=4: \text{$r$ even}}^{2k} f(2k,r)\frac{2k}{r} C_{n,r}(G),	
\end{aligned}
\end{equation}
and $T_{2k}$ is the remainder term. 
Observe that under both $\mathbb{P}_{0,n}$ and local $\mathbb{P}_{1,n}$, $$
\sqrt{\pnav}\left[\Tr\left( A_{\mathrm{cen1}}^{2}\right)- n + 1\right] \stackrel{d}{\to} N(0,\sigma^2).
$$ 
Here $\sigma^2= 2\lim_{n \to \infty}\frac{\Var_{0,n}\left[(x_{1,2}-\E_{0,n}[x_{1,2}])^{2}\right]}{\pnav(1-\pnav)^2}$.
When $\pnav \to 0$, the scaling $\sqrt{\pnav}$ kills the mean shift in $R_{2,2k}$ which leads to the degeneracy in the asymptotics. 
However, this can be circumvented by working with the difference
\[
\Tr(A_{\mathrm{cen1}}^{2k})- k\psi_{2k} \Tr\left( A_{\mathrm{cen1}}^{2}\right).
\]
One can show that under appropriate conditions, $T_{2k}$ has negligible fluctuation around its mean. Therefore, this difference has a nontrivial asymptotic normal distribution without any further scaling,
which is a direct consequence of the joint asymptotic normality of the signed cycles \cite{Ban16}. 
This is described in more details in Theorem \ref{thm:even} below.

The following theorem characterizes the first and second moments of $T_{2k}$.
\begin{theorem}
\label{thm:even}
Suppose $n\pnav\to\infty$ and $k = o( \min( \log(n\pnav), \sqrt{\log(n)} ) )$ as $n\to\infty$.
Under both $\mathbb{P}_{0,n}$ and any local $\mathbb{P}_{1,n}$, if $\pnav\to 0$, one has for $i\in \{0,1\}$
\begin{align*}
\E_{i,n}[T_{2k}] & = (1+o(1))\left[\alpha_{1,2k}+ \frac{\alpha_{2,2k}}{\pnav} + \frac{\alpha_{3,2k}}{n\pnav^2} + \varepsilon^{(1)}_{i,2k}\right],\\
\Var_{i,n}[T_{2k}] & = (1+o(1)) \left[\frac{v_{2k}}{n^2p^3_{n,av}} + \varepsilon^{(2)}_{i,2k} \right],
\end{align*}
where 
\begin{align*}
\alpha_{1,2k} &= 2^{2k-1} - \binom{2k}{k}\frac{5k+1}{2(k+1)} + \binom{k+1}{2} \psi_{2k} - 3\binom{2k}{k+2},
\,\,
\alpha_{2,2k} = \binom{2k}{k+2},
\end{align*}
$\alpha_{3,2k} \le 2^{2k} (2k)^{12}$ and $v_{2k} \le 2^{4k} (C_{1}k)^{C_2}$ for some numeric constants $C_1, C_2 > 0$.
Moreover, $\varepsilon^{(2)}_{i,2k}\to 0$ for $i=0,1$.
If further $n^2\pnav^3\to \infty$ and $k = o(\min( \log(n^2\pnav^3), \sqrt{\log(n)} )$, we also have
$\varepsilon^{(1)}_{i,2k}\to 0$ for $i=0,1$.

When $\pnav \to p > 0$, we replace the multiplier $\frac{1}{\pnav}$ in the second term of $\E_{i,n}[T_{2k}]$ with $\lim_{n\to\infty} \frac{\Var_{0,n}\left[(x_{1,2}-\E_{0,n}[x_{1,2}])^{2}\right]}{\pnav^2(1-\pnav)^2}$,
while all the other conclusions remain the same.
The results continue to hold if we replace $\Acenone$ with $\Acentwo$ in the definition of $T_{2k}$.
\end{theorem} 


\begin{remark}\label{rem:alpha}
If $n^2 \pnav^3\to\infty$ and $k= o(\min( \log(n^2\pnav^3), \sqrt{\log(n)} ) )$, then Theorem \ref{thm:even} implies that 
$\Var_{i,n}\left[ T_{2k}\right] \to 0$ for $i=0,1$, and so $T_{2k}$ is essentially a constant though we do not have an explicit formula for its mean.
Moreover, the constant remains the same whether it is under null or local alternative since $\E_{0,n}(T_{2k}) = \E_{1,n}(T_{2k})+o(1)$ under the above growth condition.
Furthermore, if $n \pnav^2 \to\infty$ and 
$k= o( \min ( \log(n\pnav^2), \sqrt{\log(n)} ) )$, one gets 
\begin{equation*}
T_{2k} = \binom{2k}{k+2}\left[\frac{1}{\pnav}-3\right]+ 2^{2k-1} - \binom{2k}{k}\frac{5k+1}{2(k+1)} + \binom{k+1}{2} \psi_{2k} +o_p(1)
\end{equation*}
under both null and local alternative when $\pnav\to 0$,
where $o_p(1)$ stands for a term that goes to zero in probability.
When $\pnav\to p > 0$, we replace $1/\pnav$ in the last display as in Theorem \ref{thm:even}.
\end{remark}

A careful examination of the sources of randomness in $T_{2k}$ suggests that one could offset the randomness by computing specific statistics of the graph adjacency matrix. 
However, these correction statistics are specific to the value of $k$ and hence are hard to track when $k$ is large.
As an example, the following proposition spells out the corrections for $T_4$ and $T_6$.

\begin{proposition}\label{prop:correction}
Suppose $n\pnav\to\infty$ as $n\to\infty$.
Under both $\mathbb{P}_{0,n}$ and $\mathbb{P}_{1,n}$ with finite $c$, 
\begin{align*}
& T_{4} - \left(\frac{1}{n\pnav(1-\pnav)}\right)^2\sum_{i,j}(x_{i,j}- \pnav)^{4} \stackrel{p}{\to} 0, \quad \mbox{and}\\
& T_{6} -  6\left(\frac{1}{n\pnav(1-\pnav)}\right)^3 \sum_{i_1,i_2,i_3} (x_{i_1,i_2}- \pnav)^{4}(x_{i_2,i_3}- \pnav)^{2} \\
& ~~~~~~~~~~~~~~~~~~~~~~~ - \left(\frac{1}{n\pnav(1-\pnav)}\right)^3 \sum_{i,j}(x_{i,j}- \pnav)^{6} -4  \stackrel{p}{\to} 0,
\end{align*}
where all the summations are over distinct indices. 
The results remain valid if $\pnav$ is replaced by $\hpnav$ in $T_{4}, T_6$ and the correction terms.
\end{proposition}

%

\subsection{Approximation of signed cycles by LSSs}

Theorems \ref{thm:null}--\ref{thm:even} suggest that signed cycles and linear spectral statistics of properly rescaled adjacency matrices are closely connected.
In what follows, we further formalize this idea and demonstrate how one could approximate signed cycles of growing lengths with carefully chosen linear spectral statistics.

As an illustration, let 
\begin{align*}
\overrightarrow{C}_{n,2k+1} & := (C_{n,3}(G),C_{n,5}(G),\ldots, C_{n,2k+1}(G) )',~~~\text{and}\\
\overrightarrow{\mathrm{Tr}}_{n,2k+1} & := \big(\Tr(A_{\mathrm{cen2}}^{3}),\Tr(A_{\mathrm{cen2}}^{5}),\ldots, \Tr( A_{\mathrm{cen2}}^{2k+1})\big)'.
\end{align*}
We proved in Theorems \ref{thm:null} and \ref{thm:alt} that under both null and local alternatives,
whenever $k= o\left(\min\left(\log(np_{n,\mathrm{av}}), \sqrt{\log{n}}\right)\right)$ and $n\pnav\to\infty$, we have elementwise,
\[
\mathbb{D}_{2k+1}\overrightarrow{C}_{n,2k+1} - 
\overrightarrow{\mathrm{Tr}}_{n,2k+1} \stackrel{p}{\to} 0.
\]
Here $\mathbb{D}_{2k+1}$ 
is the $k \times k$ lower triangular matrix
given by
\begin{equation}\label{def:D}
\left(
\begin{array}{lllll}
1& 0 & 0 & \ldots & 0\\
\frac{5f(5,3)}{3}& 1 & 0 &\ldots & 0\\
\frac{7f(7,3)}{3}& \frac{7f(7,5)}{5} & 1 & \ldots & 0\\
\vdots           & \vdots            & \vdots & \ddots &0\\
\frac{(2k+1) f(2k+1,3)}{3}& \frac{(2k+1)f(2k+1,5)}{5} &  \frac{(2k+1)f(2k+1,7)}{7} & \ldots & 1              
\end{array}
\right).
\end{equation}
Using the fact \cite[Lemma 2]{lang2000polynomials} that 
\begin{equation}
	\label{eq:fmr}
f(m,r)\frac{m}{r} = \binom{m}{\frac{m+r}{2}},
\end{equation}
one can prove that 
\begin{equation*}
(\mathbb{D}_{2k+1}^{-1})_{k,j} = P_{2k+1}[2j+1]
\end{equation*}
where $P_{j}[i]$ is the coefficient of $z^{i}$ in the polynomial $P_{j}(z)$ defined in (\ref{def:ChebyshevII}). 
See, for instance, \cite[Equation (37)]{lang2000polynomials}.
An analogous result holds for signed cycles of even lengths, in which case one needs to take in account the random variables $T_{2k}$ to offset the mean values of the even powers. 
Formally, we have the following result.
\begin{theorem}
	\label{thm:cyclestotrace}
Suppose $n \pnav \to\infty$ and $k = o( \min( \log(n\pnav), \sqrt{\log(n)} ) )$ as $n\to\infty$.
The following results hold under both $\mathbb{P}_{0,n}$ and local $\mathbb{P}_{1,n}$:
\begin{enumerate}[(i)]
\item (Construction of odd signed cycles from LSS) We have
\begin{equation}\label{formulaodd}
C_{n,2k+1}(G)- \Tr\left( P_{2k+1}(\Acentwo) \right)
\stackrel{p}{\to} 0. 
\end{equation}
\item (Construction of even signed cycles from LSS)
Let $T_0 = T_2 = 0$. Then
\begin{equation}\label{formulaeven}
C_{n,2k}(G)- 
\Tr\left( P_{2k}(\Acentwo) \right)
- \sum_{r=0:r~\text{even}}^{2k} P_{2k}[r] \left[T_{r}-{\frac{r}{2}+1 \choose 2}\psi_r \right] \stackrel{p}{\to} 0.
\end{equation}
If further $n^2\pnav^3\to\infty$ and $k = o( \min( \log(n^2\pnav^3), \sqrt{\log(n) } ) )$, then we may replace $T_r$ in \eqref{formulaeven} with $\E_{0,n}[T_r]$.
\end{enumerate}
\end{theorem} 
   
For the third term on the left side of \eqref{formulaeven}, we do not have other deterministic terms involved in \eqref{eq:T2k} because of the following cancellation (see appendix for proofs) 
\begin{align}
	\label{eq:cheby-facts}
\sum_{r=0}^{k} P_{2k}[2r]\psi_{2r} = 0, \quad \mbox{and}\quad
\sum_{r=1}^{k} P_{2k}[2r] r \psi_{2r} = 0, \quad \mbox{for all $k\geq 2$.}
\end{align}

\begin{remark}
A careful examination of the proofs in the appendix shows that all the conclusions under local alternatives in Theorems \ref{thm:alt} -- \ref{thm:cyclestotrace} actually hold conditioning on the group assignments $\sigma_i$, $1\leq i\leq n$.
Thus, the approximation of signed cycles by LSSs of Chebyshev polynomials works for any group assignment configurations as long as $t$ in \eqref{eq:t} is finite.
\end{remark}

\subsection{Likelihood ratio tests}

Recall the null and alternative hypotheses in \eqref{test:sbm} with the key index $t$ defined as in \eqref{eq:t}.
\citet{Ban16} showed that if $n\pnav\to\infty$ as $n\to\infty$ and $p_n > q_n$, in the contiguous regime, i.e.~$0<t<1$, the likelihood ratio test is asymptotically the same as the test that rejects for large values of 
\begin{equation}
\label{eq:Lc}
L_c := \sum_{r=3}^\infty \frac{t^r C_{n,r}(G)}{2r} 
\end{equation}
which has the following asymptotic distributions under the null and alternative:
\begin{equation}
	\label{eq:Lcdist}
L_c | \mathbb{P}_{0,n}  \stackrel{d}{\to} N\left(0, \sigma(t)^2  \right)
\quad \mbox{and} \quad
L_c | \mathbb{P}_{1,n}  \stackrel{d}{\to} N\left( \sigma(t)^2 , \sigma(t)^2 \right),
\end{equation}  
where 
\begin{equation}
	\sigma(t)^2 = \frac{1}{2}\left[ -\log(1-t^2) - t^2 - \frac{t^4}{2} \right].
\end{equation}
If $p_n < q_n$, we replace every $t^r$ in \eqref{eq:Lc} with $(-t)^r$ while everything else remains the same. 
Hence, at any given $t\in (0,1)$, the largest asymptotic power achievable by any level $\alpha$ test is 
\begin{equation}
\label{eq:optpower}
\Phi\left( - z_{\alpha} + \sigma(t) \right),
\end{equation}
where $\Phi$ is the CDF of the standard normal distribution and $z_\alpha = \Phi^{-1}(1-\alpha)$.
However, neither the exact likelihood ratio test nor the test in \eqref{eq:Lc} based on sign cycles is computationally tractable. 

Given Theorems \ref{thm:null}--\ref{thm:cyclestotrace}, one of the key findings of the present paper is that when $n^2\pnav^3\to\infty$, we can achieve the exact asymptotic optimal power \eqref{eq:optpower} by a test based on some linear spectral statistic, which is of $O(n^3\log{n})$ time complexity.
If we only have $n\pnav \to\infty$, we propose a slightly different test based on another linear spectral statistic that has a smaller but nontrivial asymptotic power in the contiguous regime.
In particular, we have the following theorem.

\begin{theorem}
\label{thm:llslike}
Suppose that as $n\to\infty$, $t$ defined in \eqref{eq:t} satisfies $t\in (0,1)$. 
Then the following results hold if $p_n > q_n$:
\begin{enumerate}[(i)]
\item When $n^2\pnav^3\to\infty$ and $k_n = o( \min( \log(n^2\pnav^3), \sqrt{\log(n) } ) ) \to\infty$, then the test statistic
\begin{equation}
	\label{eq:optlssall}
L_a = \sum_{r=3}^{k_{n}}\frac{t^{r}\Tr\left(P_{r}\left(\Acentwo\right)\right)}{2r}
\end{equation}
satisfies 
\begin{equation}
	\label{eq:Ladist}
\begin{aligned}
L_a - \mu_{n,\pnav}(t) | \mathbb{P}_{0,n}  & \stackrel{d}{\to} N\left(0, \sigma(t)^2  \right),\\
L_a - \mu_{n,\pnav}(t) | \mathbb{P}_{1,n} & \stackrel{d}{\to} N\left( \sigma(t)^2 , \sigma(t)^2 \right),
\end{aligned}
\end{equation}  
where $\mu_{n,\pnav}(t)$ is a deterministic quantity depending only on $n,\pnav$ and $t$.
Therefore, a level $\alpha$ test that rejects for large values of $L_a$ achieves the exact asymptotic optimal power \eqref{eq:optpower}.

\item When $n\pnav\to\infty$ and $k_n = o( \min( \log(n \pnav), \sqrt{\log(n) } ) ) \to\infty$, then the test statistic
\begin{equation}
\label{eq:optlssodd}
L_o = \sum_{r=1}^{k_{n}}\frac{t^{2r+1}\Tr (P_{2r+1}(\Acentwo ) )}{2(2r+1)}
\end{equation}
satisfies 
\begin{equation}
\label{eq:Lodist}
L_o  | \mathbb{P}_{0,n}  \stackrel{d}{\to} N\left(0, \sigma_1(t)^2  \right)
\quad \mbox{and} \quad
L_o  | \mathbb{P}_{1,n}  \stackrel{d}{\to} N\left( \sigma_1(t)^2 , \sigma_1(t)^2 \right) 
\end{equation} 
where 
\begin{equation}
\sigma_1(t)^2 = \frac{1}{4}\left[ -\log\left(\frac{1-t^2}{1+t^2}\right) -2t^2
 \right].
\end{equation}
Therefore, a level $\alpha$ test that rejects for large values of $L_o$ achieves an asymptotic power of $\Phi(-z_\alpha + \sigma_1(t))$.
\end{enumerate}
If $p_n < q_n$, we replace every $t$ in the definitions of \eqref{eq:optlssall} and \eqref{eq:optlssodd} with $-t$, and the same conclusions hold.
\end{theorem}

If only $n\pnav\to\infty$, it is possible to construct tests that are more powerful than that based on \eqref{eq:optlssodd}.
We leave the pursuit of the optimal asymptotic power under this condition for future research.

We conclude the section with a discussion on the quantity $\mu_{n,\pnav}(t)$.
First suppose $\pnav\to 0$.
When $n^2\pnav^3\to\infty$ and $k_n = o( \min( \log(n^2\pnav^3), \sqrt{\log(n) } ) ) \to\infty$, we have from Theorems \ref{thm:even} and \ref{thm:cyclestotrace} that
\begin{equation}
	\label{eq:munpt}
\begin{aligned}
\mu_{n,\pnav}(t) = \frac{1}{2}\sum_{i=1}^{\infty} \frac{1}{2i}\left(\frac{t}{1+t^2}\right)^{2i} \left[\alpha_{1,2i}+\frac{\alpha_{2,2i}}{p_{n,\mathrm{av}}} + \frac{\alpha_{3,2i}}{np_{n,\mathrm{av}}^2}- \binom{i+1}{2}\psi_{2i}\right] 
-\frac{1}{2}\log(1+t^2). 
\end{aligned} 
\end{equation}
Here $\alpha_{1,2i},\alpha_{2,2i}$ and $\alpha_{3,2i}$ have been defined in Theorem \ref{thm:even}.
Although we do not have explicit formula for $\alpha_{3,2i}$'s, we may estimate them by simulation under $\mathbb{P}_{0,n}$ with an estimated $\hpnav$.
To obtain \eqref{eq:munpt}, we have used the following generating function of Chebyshev polynomials 
\[
\sum_{i=1}^{\infty} \frac{t^i P_{i}(x)}{i} = \log\left(\frac{1}{1-tx + t^2} \right)\,.
\] 
If further $n\pnav^2\to\infty$ and $k_n = o( \min( \log(n\pnav^2), \sqrt{\log(n) } ) ) \to\infty$, then we may replace the right side in \eqref{eq:munpt} with the following explicit expression
\begin{equation}
\label{eq:munpt2}
\begin{split}
& \frac{1}{2}\sum_{i=1}^{\infty} \frac{1}{2i}\left(\frac{t}{1+t^2}\right)^{2i} \left[\alpha_{1,2i}+\frac{\alpha_{2,2i}}{p_{n,\mathrm{av}}} - \binom{i+1}{2}\psi_{2i}\right]-\frac{1}{2}\log(1+t^2)  \\
  &             = \frac{1}{2}\sum_{i=1}^{\infty} \frac{1}{2i}\left(\frac{t}{1+t^2}\right)^{2i}\left[ 2^{2i-1}- \binom{2i}{i}\left[ \frac{5i+1}{2i+2}\right]+ \binom{2i}{i+2}\left[ \frac{1}{\pnav} -3 \right] \right] 
				 -\frac{1}{2}\log(1+t^2) . 
\end{split}
\end{equation}
When $\pnav \to p > 0$, we replace $1/\pnav$ in the term involving $\alpha_{2,2i}$ with $\lim_{n\to\infty} \frac{\Var_{0,n}\left[(x_{1,2}-\E_{0,n}[x_{1,2}])^{2}\right]}{\pnav^2(1-\pnav)^2}$ while the others are the same.

\begin{remark}
Suppose for simplicity $p_n > q_n$.
One might observe that when $t<1$, the analytic functions used in the LSSs in $L_a$ and $L_o$ in Theorem \ref{thm:llslike} have limits
\begin{equation}\label{functionall}
f_a(x) = \sum_{i=3}^{\infty} \frac{t^{i}P_{i}(x)}{2i}= \frac{1}{2} \log\left( \frac{1}{1-tx +t^2} \right) - \frac{tx}{2} - \frac{t^2(x^{2}-2)}{4}
\end{equation}
and 
\begin{equation}\label{functionodd}
f_o(x) = \sum_{i=1}^{\infty} \frac{t^{2i+1}P_{2i+1}(x)}{4i+2} = \frac{1}{4} \log\left( \frac{1+tx +t^2}{1-tx + t^2} \right) - \frac{tx}{2},
\end{equation}
respectively.
So it might be tempting to directly use LSSs of the foregoing limits directly as the test statistics in (\ref{eq:optlssall}) and (\ref{eq:optlssodd}) respectively. 
However, this is not preferable for the following two reasons.

First, observe that given any $t<1$, both $f_a$ and $f_o$ take finite values only in the open interval $\left(-\left(t+ \frac{1}{t} \right),  t+ \frac{1}{t} \right)$. On the other hand, it is known that the spectral norm of $\Acenone$ converges to $2$ only under the condition $\pnav \gg \frac{\log(n)^{4}}{n}$. See Vu (2007) \cite{vu2005spectral} for a reference and using Weyl's interlacing inequality it is easy to see that the same holds for the spectral norm of $\Acentwo$. 
However, the result in (\ref{eq:optlssodd}) holds as long as $n\pnav \to \infty$. So in this case the test statistic $\Tr f_{o}( \Acentwo)$ will be undefined when $\pnav \ll \frac{\log(n)^{4}}{n}$ with a nontrivial  probability. 
In an unreported simulation study, we find both $\Tr f_{a}( \Acentwo)$ and $\Tr f_{o}( \Acentwo)$ highly unstable for small values of $\pnav$. 

Another technical difficulty is the evaluation of $\E_{0,n}\left[ \Tr f(\Acentwo)\right]- n\E_{\mathrm{sc}}\left[  f(X)\right]$ for general analytic functions. 
Here $\E_{\mathrm{sc}}$ denote the expectation with respect to the semi-circle law on $[-2,2]$. 
However, the CLT of $\Tr\left(f(\Acentwo)\right)- \E_{0,n}\left[ \Tr f(\Acentwo)\right]$ with the correct asymptotic variance can be obtained from Theorem 3.5 of Anderson and Zeitouni \cite{AZ05}. 
The evaluation of $\E_{0,n}\left[ \Tr f(\Acentwo)\right]- n \E_{\mathrm{sc}}\left[  f(X)\right]$ can be shown from the convergence of $M_{n}^{2}(z)$ in Bai and Silverstein \cite{bai2004clt} (pp.585--593) when $\pnav\to p>0$. 
However, a possibly different argument is required when $\pnav \to 0$ at some rate. Informed readers might note that $M_{n}^{2}(z)$ is expected to blow up due to the $\frac{1}{\pnav}$ factor in (\ref{eq:munpt}) whenever $\pnav \to 0$. 
We anticipate it would be difficult to find the asymptotics of $M_{n}^{2}(z)$ when $n\pnav^2$ does not diverge to infinity. 
We leave it as a technical question for future research. 
\end{remark}

\section{Adaptive tests}
\label{sec:adaptive}

\subsection{Adaptive test statistics for \eqref{test:sbm}}

The following result presents two classes of adaptive test statistics which do not require knowledge of any model parameters and are completely data-driven. 
Remarkably, under appropriate growth conditions on the connection probabilities, for testing \eqref{test:sbm}, tests that reject for large values of these test statistics achieve nontrivial powers in the contiguous regime and consistency in the singular regime.
All these tests are of $O(n^3\log n)$ time complexity.

\begin{proposition}\label{prop:sing}
Suppose that as $n\to\infty$, $t$ defined in \eqref{eq:t} remains bounded.
In addition, let $\varepsilon$ be a constant in $(0, \frac{1}{2}]$.
The following results hold when $p_n > q_n$,
\begin{enumerate}[(i)]
\item If $n\pnav \to \infty$ and $k_n= o( \min( \log(n\pnav), \sqrt{\log(n) } ) ) \to\infty$, then the test statistic
\begin{equation}
	\label{eq:adapt-odd}
\widetilde{L}_o = \sum_{r=1}^{k_{n}}\frac{\Tr\left(P_{2r+1}\left(\Acentwo\right)\right)}{(4r+2)(\log (2r+1))^{\frac{1}{2}+ \varepsilon}}
\end{equation}
satisfies 
\[
\widetilde{L}_o | \mathbb{P}_{0,n} \stackrel{d}{\to} N\left(0, \widetilde{\sigma}_{o}(t)^2  \right) 
\quad \mbox{and} \quad
\widetilde{L}_o - \widetilde{\mu}_{n,o}(t) | \mathbb{P}_{1,n} \stackrel{d}{\to} N\left(0, \widetilde{\sigma}_{o}(t)^2  \right).
\]
Here 
\[
\widetilde{\sigma}_{o}(t)^{2}= \sum_{r=1}^{\infty} \frac{1}{(4r+2)(\log (2r+1))^{1+2\varepsilon}}< \infty
\]
and 
\[
\widetilde{\mu}_{n,o}(t) \to \left\{
\begin{array}{ll}
\sum_{r=1}^{\infty} \frac{t^{2r+1}}{(4r+2)(\log (2r+1))^{\frac{1}{2}+ \varepsilon}} & \text{when} ~ t<1,\\
\infty & \text{when} ~ t\geq 1.
\end{array}
\right.
\] 
\item If $n^2\pnav^3 \to \infty$ and $k_n= o( \min( \log(n^2\pnav^3), \sqrt{\log(n) } ) ) \to\infty$, then the test statistic
\begin{equation}
\label{eq:adapt-all}
\widetilde{L}_a = \sum_{r=3}^{k_{n}}\frac{\Tr\left(P_{r}\left(\Acentwo\right)\right)}{2r(\log r)^{\frac{1}{2}+ \varepsilon}}- \bar\mu_{n,a}
\end{equation}
satisfies 
\[
\widetilde{L}_a | \mathbb{P}_{0,n} \stackrel{d}{\to} N\left(0, \widetilde{\sigma}_{a}(t)^2  \right)
\quad \mbox{and} \quad
\widetilde{L}_a - \widetilde{\mu}_{n,a}(t) | \mathbb{P}_{1,n} \stackrel{d}{\to} N\left(0, \widetilde{\sigma}_{a}(t)^2  \right).
\]
Here 
\[
\widetilde{\sigma}_{a}(t)^{2}= \sum_{r=3}^{\infty} \frac{1}{2r(\log r)^{1+2\varepsilon}}< \infty
\]
and  
\[
\widetilde{\mu}_{n,a}(t) \to \left\{
\begin{array}{ll} 
\sum_{r=3}^{\infty} \frac{t^{r}}{2r(\log r)^{\frac{1}{2}+ \varepsilon}} & \text{when} ~ t<1,\\
\infty & \text{when} ~ t\geq 1.
\end{array}
\right.
\] 
In addition, if $\pnav \to 0$,
\[
\bar{\mu}_{n,a}= \sum_{r=2}^{[\frac{k_n}{2}]} \sum_{j=1}^{r} \frac{1}{2r(\log r)^{\frac{1}{2}+ \varepsilon}}P_{2r}(2j)\left[  \alpha_{1,2j}+\frac{\alpha_{2,2j}}{p_{n,\mathrm{av}}} + \frac{\alpha_{3,2j}}{np_{n,\mathrm{av}}^2}- \binom{j+1}{2}\psi_{2j}\right].
\]
When $\pnav \to p \in(0,1)$ or $n\pnav^2 \to \infty$, the change in the form of $\bar{\mu}_{n,a}$ is analogous to the discussion below Theorem \ref{thm:llslike}.
\end{enumerate}
If $p_n < q_n$, we define $\widetilde{L}_o$ by multiplying \eqref{eq:adapt-odd} with $-1$, and $\widetilde{L}_a$ by multiplying the $r$th term in the series in \eqref{eq:adapt-all} with $(-1)^r$ for all $r\geq 3$, then the same conclusions hold. 
\end{proposition}

\begin{remark}
In practice, if only $n^2 \pnav^3\to\infty$ while $n\pnav^2$ does not diverge, then $\bar{\mu}_{n,a}$ cannot be evaluated directly in closed form. However, one may obtain an approximation to sufficiently high precision by simulation.
\end{remark}


\subsection{Powers under multi-block alternatives}

We now investigate the powers of the adaptive tests against multi-block alternatives. 
In particular, for any fixed $\kappa > 2$, suppose the testing problem is now
\begin{equation}
\label{eq:alt-kappa}
H_0: A \sim \mathcal{G}_1\left(n, \frac{p_n+(\kappa-1)q_n}{\kappa}\right), \quad \mbox{vs.}\quad H_{1,\kappa}: A \sim \mathcal{G}_\kappa(n,p_n,q_n).
\end{equation}
We are interested in the cases where for $a_n = np_n$ and $b_n = nq_n$, as $n\to\infty$, $a_n$ and $b_n\to\infty$ and 
\begin{equation}
\label{eq:c-kappa}
c_\kappa = \frac{(a_{n}-b_{n})^2}{a_n+ (\kappa-1)b_n}
\end{equation}
remains a constant.

One can prove that the results in (\ref{formula1}), Theorem \ref{thm:even} and Theorem \ref{thm:cyclestotrace} remain valid for multiple block case by arguments similar to the proofs of (\ref{formula1}), Theorem \ref{thm:even} and Theorem \ref{thm:cyclestotrace}. However, in the multiple block case the signed cycles $C_{n,k}(G)$'s will be asymptotically independent normal with a  different asymptotic mean and asymptotic variance $2k$. 
In particular, if one can prove that the mean of $C_{n,k}(G)$ grow faster than exponential in $k$, then the tests given in Proposition \ref{prop:sing} will be consistent. 
The following Proposition gives a sufficient condition \emph{(not necessary}) for the consistency of the tests in Proposition \ref{prop:sing}. 
\begin{proposition}
	\label{prop:multi}
Suppose $c_\kappa > \kappa$, then both tests in Proposition \ref{prop:sing} are asymptotically consistent under the respectively growth conditions.
\end{proposition}

The threshold $\frac{(a_{n}-b_{n})^2}{a_n+ (\kappa-1)b_n} > \kappa$ is known as Kesten--Stigum threshold. 
See \cite{AS16} and \cite{bank16} for further details. 
A systematic investigation of powers against possibly asymmetric multi-block alternatives is beyond the scope of the present paper.


\section{Outline of proofs}
\label{sec:overview}
In this section we give a brief outline of the proofs for Theorems \ref{thm:null}--\ref{thm:even}. 
The other theorems and propositions are essentially corollaries of these core results. 
For conciseness, throughout this section, we focus on the assortative case of $p_n > q_n$ when discussing results under local alternatives.

\subsection{Outline of proof for Theorem \ref{thm:null} (i)--(iii)}
The fundamental idea here is to prove that $\Tr(\Acenone^{k})-\E_{0,n}(\Tr(\Acenone^{k}))$ converges in distribution by using the method of moments and the limiting random variables satisfy Wick's formula and hence are Gaussian. 
We first state the method of moments.
\begin{lemma}\label{lem:mom}
Let $Y_{n,1},\ldots, Y_{n,l}$ be a random vector of dimension $l$. 
Then $(Y_{n,1},\ldots, Y_{n,l}) \stackrel{d}{\to} (Z_1,\ldots,Z_{l})$ if both of the following conditions are satisfied:
\begin{enumerate}[(i)]
\item 
$\lim_{n \to \infty}\E[X_{n,1}\ldots X_{n,m}] $
exists for any fixed $m$ and $X_{n,i} \in \{ Y_{n,1},\ldots,Y_{n,l} \}$ for $1\le i \le m$.
\item(Carleman's Condition)\cite{Carl26}
\[
\sum_{h=1}^{\infty} \left(\lim_{n \to \infty}\E[X_{n,i}^{2h}]\right)^{-\frac{1}{2h}} =\infty ~~ \forall ~ 1\le i \le l.
\]
\end{enumerate} 
Further,  
$\lim_{n \to \infty}\E[X_{n,1}\ldots X_{n,m}]= \E[X_{1}\ldots X_{m}]$.
Here $X_{n,i} \in \{ Y_{n,1},\ldots,Y_{n,l} \}$ for $1\le i \le m$ and $X_{i}$ is the in distribution limit of $X_{n,i}$.
\end{lemma} 

Next we state Wick's formula for Gaussian random variables which was first proved by \citet{I18} and later on introduced by \citet{W50} in the physics literature. 
\begin{lemma}[Wick's formula \cite{W50}]
	\label{lem:wick}
Let $(Y_1,\ldots, Y_{l})$ be a multivariate mean $0$ random vector of dimension $l$ with covariance matrix $\Sigma$ (possibly singular). 
Then $(Y_1,\ldots, Y_{l})$ is jointly Gaussian if and only if for any positive integer $m$ and $X_{i} \in \{ Y_1,\ldots,Y_{l} \}$ for $1\le i \le m$
\begin{equation}\label{eqn:wick}
\E[X_1\ldots X_{m}]=\left\{
\begin{array}{ll}
  \sum_{\eta} \prod_{i=1}^{\frac{m}{2}} \E[X_{\eta(i,1)}X_{\eta(i,2)}] & ~ \text{for $m$ even}\\
  0 & \text{for $m$ odd.}
\end{array}
\right.
\end{equation}
Here $\eta$ is a partition of $\{1,\ldots,m \}$ into $\frac{m}{2}$ blocks such that each block contains exactly $2$ elements and $\eta(i,j)$ denotes the $j$th element of the $i$th block of $\eta$ for $j=1,2$.
\end{lemma}

It is worth noting that the random variables $Y_1,\ldots, Y_{l}$ need not be distinct. 
When $Y_1=\cdots = Y_{l}$, Lemma \ref{lem:wick} provides a description of the moments of Gaussian random variables.

In what follows, we focus on odd powers to illustrate the main ideas. 
Detailed arguments for even powers can be found in the actual proof.
We start with the following identity. 
\begin{equation}\label{overview:trexp}
\Tr(\Acenone^{k})= \left(\frac{1}{n\pnav(1-\pnav)}\right)^{\frac{k}{2}}\sum_{w} \left[X_{w}\right].
\end{equation}
Here any $w$ is an ordered tuple of indices (not necessarily distinct) $(i_0,\ldots,i_{k})$ with $i_0=i_{k}$ where $i_{j}\in \{  1,2,\ldots, n\}$ for $0\le j \le k$ and we define 
\begin{equation*}
X_{w}:= \prod_{j=0}^{k-1}\left(x_{i_j,i_{j+1}}- \pnav\right). 	
\end{equation*}
As we shall formally define in Section \ref{subsec:word} below, 
these $w$'s are called closed word of length $k+1$.

We at first prove when $k$ is odd, most of the random variables $X_{w}$ have mean $0$.
As a consequence, one doesn't need a centering for $\Tr(\Acenone^{k})$ when $k$ is odd. 
This is not the case for even $k$ though. 
The next step 
is to prove $\lim_{n \to \infty}\E[R_{n,1}\ldots R_{n,m}]$ exists for any fixed $m$ where $R_{n,i} \in \{ \Tr(\Acenone^{2k_1+1}),\ldots, \Tr(\Acenone^{2k_{l}+1}) \}$ and to prove the limit $\lim_{n \to \infty}\E\left[ R_{n,1}\ldots R_{n,m} \right]$ satisfy the Wick's formula (\ref{eqn:wick}).
To this end, observe that 
\begin{equation*}
\begin{split}
\E\left[\Tr(\Acenone^{l_1})\ldots \Tr(\Acenone^{l_m})\right]
&=  \left(\frac{1}{n \pnav(1-\pnav)}\right)^{\frac{\sum l_{i}}{2}} \sum_{w_1,\ldots, w_{i}} \E\left[X_{w_1}\ldots X_{w_m}\right]\\
&= \left(\frac{1}{n \pnav(1- \pnav)}\right)^{\frac{\sum l_{i}}{2}} \sum_{a}\E\left[X_{w_1}\ldots X_{w_m}\right].
\end{split}
\end{equation*}
Here $w_i$ is a closed word of length $l_{i}+1$ and any sentence $a$ is an ordered collection of words $[w_{i}]_{i=1}^{m}$.
Then we verify that $\E\left[X_{w_1}\ldots X_{w_m}\right]=0$ unless the corresponding sentence $a=[w_{i}]_{i=1}^{m}$ is a weak CLT sentence \cite{AZ05} (see also Def.~\ref{def:weakCLT} in appendix). 
We then show that among weak CLT sentences, 
\[
\left(\frac{1}{n\pnav(1-\pnav)}\right)^{\frac{\sum l_{i}}{2}} \sum_{a}\E\left[\left|X_{w_1}\ldots X_{w_m}\right|\right] \le \frac{\Psi(l_1,\ldots,l_m)}{n\pnav} \to 0
\]
as $n\pnav\to\infty$
unless $a$ is a CLT sentence in which all the involved random variables are naturally paired \cite{AZ05} (cf.~Def.~\ref{def:weakCLT} and Prop.~\ref{prop:cltrep} in appendix). 
This natural pairing is closely related to the partition $\eta$ introduced in Lemma \ref{lem:wick} which essentially proves the CLT in part (i) of Theorem \ref{thm:null}. 
Here $\Psi(\cdot)$ is an implicit function depending on the values $l_{i}$,
and we develop a careful upper bound on the number of weak CLT sentences (Lemma \ref{lem:appendix}) to ensure that the convergence to zero in the last display happens whenever $\max_{i}(l_i)=o(\log(np_n))$.
This completes the proof of asymptotic normality under the condition of of the theorem.


The variance formula (\ref{cov:null:odd}) is derived using the concepts of {\FK} sentences and unicyclic graphs introduced in Sections \ref{subsec:FK} and \ref{subsec:unicyclic}. 
Here the basic idea is similar to that in \cite{AZ05}. 
The major difference is that we shall also develop Proposition \ref{prop:uniword} and Lemma \ref{lem:cltwordpaircnt} which enable us to calculate the covariance between the traces and the signed cycles in addition to calculating $\Sigma$ in (\ref{cov:null:odd}).

The proof of part (iii) is completed by calculating the co-variance between the signed cycles and traces and hence showing the variance of the random variable in (\ref{formula1}) goes to $0$.

\subsection{Outline of proof Theorem \ref{thm:alt} (i)--(iii)}
This proof is based on the second moment argument. 
All expectation and variance calculation is under $\mathbb{P}_{1,n}$ conditioning on group assignments $\sigma_i$ for $1\leq i\leq n$. 
The subscript is thus omitted.
As before, we focus on odd powers to illustrate the main idea.
Observe that when the data is generated from $\mathcal{G}_2(n,p_n,q_n)$, 
the matrix $A_{\mathrm{cen1}}$ is \emph{not} properly centered, i.e., $\E\left[A_{\mathrm{cen1}} \right] \neq 0$. Here we write 
\begin{equation*}
\begin{split}
\Tr(\Acenone^{k})& = \left(\frac{1}{n\pnav(1-\pnav)}\right)^{\frac{k}{2}}\sum_{w} \left[X_{w}\right]\\
&= \left(\frac{1}{n\pnav(1-\pnav)}\right)^{\frac{k}{2}}\sum_{w} \left[\prod_{j=0}^{k-1}\left(x_{i_j,i_{j+1}}- p_{i_j,i_{j+1}}\right)+ V'_{n,w} \right]
\end{split}
\end{equation*}
where $w=(i_0,\ldots,i_{k}) $ is a generic closed word of length $k+1$ and $p_{i_j,i_{j+1}}= \E[x_{i_j,i_{j+1}}|\sigma_{i_j},\sigma_{i_{j+1}}]$. Here $V'_{n,w}$ is obtained by expanding $X_{w}$ for any $w$ and considering all the remaining terms apart from $\prod_{j=0}^{k-1}\left(x_{i_j,i_{j+1}}- p_{i_j,i_{j+1}}\right)$. 
Using arguments similar to those in the proof of part (i) in Theorem \ref{thm:null} one can prove that the random variable 
\[
\left(\frac{1}{n\pnav(1-\pnav)}\right)^{\frac{k}{2}}\sum_{w} \prod_{j=0}^{k-1}\left(x_{i_j,i_{j+1}}- p_{i_j,i_{j+1}}\right)
\]
converges to a Gaussian random variable with mean $0$ and variance same as the null case irrespective of the group assignments $\sigma_{i}$, $i=1,\dots,n$. 

The main task of the proof is then to prove that 
\[
\E\left[\left(\frac{1}{n\pnav(1-\pnav)}\right)^{\frac{k}{2}} \sum_{w} V'_{n,w} \right] \to \sum_{r=3: r ~ \text{odd}}^{k} f(k,r)\frac{k}{r}\, t^r
\]
and 
\[
\Var \left[ \left(\frac{1}{n\pnav(1-\pnav)}\right)^{\frac{k}{2}} \sum_{w} V'_{n,w}  \right] \to 0
\]
under suitable growth condition of $k$. 
The level of technicality here is increased due to the complicated form of the $V_{n,w}'$s, while the key
ideas underlying the proof are still {\FK} sentences and unicyclic graphs. 

We mention that the arguments in this particular proof are new and cannot be obtained by modifying the arguments in \cite{AZ05}. 
In particular,
we will be able to show that 
\begin{align*}
& \E\left[\left(\frac{1}{n\pnav(1-\pnav)}\right)^{\frac{k}{2}} \sum_{w} V'_{n,w} \right] \\
& ~~~ = (1+o(1))\sum_{r=3: r ~ \text{odd}}^{k} f(k,r)\frac{k}{r}\,t^{r} + O\left(\frac{2^{C_1k}\mathrm{poly}(k)}{n\pnav}+\frac{(C_2k)^{C_3k}}{n}\right)
\end{align*}
and 
\[
\Var \left[ \left(\frac{1}{n\pnav(1-\pnav)}\right)^{\frac{k}{2}} \sum_{w} V'_{n,w}  \right]= O \left(\frac{(C_4 k)^{C_5 k}}{n}\right).
\]
Here $C_i$'s are positive numeric constants and $\mathrm{poly}(k)$ is a known polynomial of $k$. 
Note that when $k= o(\min(\log(n\pnav),\sqrt{\log n}))$, $\frac{2^{C k}\mathrm{poly}(k)}{n\pnav}\to 0$ for any fixed $C$ and $\frac{(Ck)^{Dk}}{n} \to 0$ for any fixed $C$ and $D$. 
This gives part (i) of Theorem \ref{thm:alt}.
The proofs of parts (ii) and (iii) here rely on similar ideas to those used in the proofs of their counterparts in Theorem \ref{thm:null}.

\subsection{Outline of proof for Theorem \ref{thm:null} (iv) and Theorem \ref{thm:alt} (iv)} 

We focus on the null case and the proof for the alternative is similar.
All the expectation and variance taken below are with respect to $\mathbb{P}_{0,n}$.
Recall that when $A_{\mathrm{cen2}}$ is considered, the matrix is centered by the sample estimate $\hpnav$ instead of the actual parameter $\pnav$. 
Here we write 
\begin{equation*}
\begin{split}
\Tr(\Acentwo^{k})& = \left(\frac{1}{n\hpnav(1-\hpnav)}\right)^{\frac{k}{2}}\sum_{w} \left[X_{w}\right]\\
&= \left(\frac{1}{n\hpnav(1-\hpnav)}\right)^{\frac{k}{2}}\sum_{w} \left[\prod_{j=0}^{k-1}\left(x_{i_j,i_{j+1}}- \pnav\right)+ E_{n,w} \right].
\end{split}
\end{equation*}
Observe that $\Var[\hpnav]= O({\pnav}/{n^2})$. As a consequence, one might expect that $|\hpnav - \pnav|\le {\sqrt{\pnav}}/{n^{\delta}}$ for some $\delta \in (\frac{1}{2},1)$ with very high probability.  
We call the indicator random variable corresponding to this high probability event $\mathrm{Ev}$. 
When $\mathrm{Ev}$, $|\pnav -\hpnav| \ll (p_n-q_n)=O\left(\frac{\sqrt{\pnav}}{\sqrt{n}}\right)$. We do the analysis of 
\[
\E\left[\mathrm{Ev}\left(\left(\frac{1}{n\hpnav(1-\hpnav)}\right)^{\frac{k}{2}} \sum_{w} E_{n,w} \right)^2\right].
\] 
Note that $\mathrm{Ev}(p_n-\hpnav)$ is a random variable, not a constant like  $(p_n-q_n)$. 
So, many random variables that had mean $0$ in the proof of part (i) of Theorem \ref{thm:alt} no longer have mean $0$ due to dependence.
To tackle the additional dependence, a more careful combinatorial analysis will be carried out and we obtain
\begin{equation*}
\begin{split}
&\E\left[\mathrm{Ev} \left(\left(\frac{1}{n\hpnav(1-\hpnav)}\right)^{\frac{k}{2}} \sum_{w} E_{n,w} \right)^2\right]\\
&~~~~~ \le 
(C_1 k)^{C_2 k} \frac{1}{\sqrt{n}}+ n^{C_{3}k}\exp(-n^{C_4})+\frac{(C_{\mathrm{5}}k)^{C_{\mathrm{6}}k}}{n^{\delta-\frac{1}{2}}} \to 0.
\end{split}
\end{equation*}
Here the $C_i$'s are positive numeric constants.
This completes the proof.

\subsection{Outline of proof for Theorem \ref{thm:even}}
We start with the random variable 
\begin{equation}\label{outlineevenexp}
\begin{split}
& \Tr ( A_{\mathrm{cen1}}^{2k} )- k\psi_{2k} \Tr\left( A_{\mathrm{cen1}}^{2}\right)
\\
& ~~~
= \left(\frac{1}{n\pnav(1-\pnav)}\right)^{k} \sum_{w} X_{w} - k\psi_{2k} \Tr\left( A_{\mathrm{cen1}}^{2}\right) .
\end{split}
\end{equation}
In this case, we break the collection of words in \eqref{outlineevenexp} into four subgroups as follows 
\begin{equation*}
\begin{split}
&\Tr ( A_{\mathrm{cen1}}^{2k} )
= \left(\frac{1}{n\pnav(1-\pnav)}\right)^{k} \sum_{w} X_{w} \\
&~~~ = \left(\frac{1}{n\pnav(1-\pnav)}\right)^{k} \left[\sum_{w\in \mathcal{W}_1} X_{w} + \sum_{w\in \mathcal{W}_2} X_{w} + \sum_{w \in \mathcal{W}_3 } X_{w}+ \sum_{w \in \mathcal{W}_4} X_{w} \right]. 
\end{split}
\end{equation*}
Here $\mathcal{W}_1$ corresponds to the set of Wigner words, $\mathcal{W}_2$ stands for the set of all weak Wigner words (cf.~Def.~\ref{def:weakwigner}), $\mathcal{W}_3$ is $\cup_{r} \mathfrak{W}_{2k+1,r,k+r/2}$ which collects all unicyclic graphs (cf.~Prop.~\ref{prop:uniword}) and $\mathcal{W}_{4}$ is the complement of $\mathcal{W}_1\cup \mathcal{W}_{2}\cup\mathcal{W}_{3}$. 
Using  Lemma \ref{lem:expcltpair} in the appendix, one can ignore the class $\mathcal{W}_{4}$. 

We first show that under both $\mathbb{P}_{0,n}$ and local $\mathbb{P}_{1,n}$ (conditioning on group assignment $\sigma_i$ for $1\leq i\leq n$),
\[
\E \left[ \left(\frac{1}{n\pnav(1-\pnav)}\right)^{k} \sum_{w\in \mathcal{W}_1} X_{w} \right]= n\psi_{2k}- \binom{k+1}{2}\psi_{2k} + o(1)
\]
and 
\[
\Cov\bigg(\left(\frac{1}{n\pnav(1-\pnav)}\right)^{k} \sum_{w\in \mathcal{W}_1} X_{w}, \Tr\left( A_{\mathrm{cen1}}^2\right)  \bigg) 
= 
\frac{\sigma^2k\psi_{2k}}{\pnav} 
+ o(1).	
\]
Hence, $\Var((\frac{1}{n\pnav(1-\pnav)})^{k} \sum_{w\in \mathcal{W}_1} X_{w}- k\psi_{2k}\Tr( A_{\mathrm{cen1}}^2))\to 0$.
Next, arguments similar to the proof of Theorem \ref{thm:alt} will show that 
\[
\left(\frac{1}{n\pnav(1-\pnav)}\right)^{k} \sum_{w\in \mathcal{W}_3} X_{w}- \sum_{r=4:r ~ \text{even}} f(2k,r) \frac{2k}{r} C_{n,r}(G) \stackrel{p}{\to} 0.
\]
So our final focus is on the random variable 
\[
\left(\frac{1}{n\pnav(1-\pnav)}\right)^{k}\sum_{w\in \mathcal{W}_2} X_{w} = T_{2k} + o_p(1).
\]
We again break $\mathcal{W}_{2}$ into two further groups depending on whether the graph $G_{w}$ corresponding to a word $w$ is a tree or not. 
It can be proved that under both $\mathbb{P}_{0,n}$ and local $\mathbb{P}_{1,n}$
\[
\E\left[\left(\frac{1}{n\pnav(1-\pnav)}\right)^{k}\sum_{w \in \mathcal{W}_{2}:G_{w}\neq\text{tree}} X_{w}\right] = (1+o(1))\alpha_{1,2k} + o(1),
\] 
and 
\[
\Var\left[\left(\frac{1}{n\pnav(1-\pnav)}\right)^{k}\sum_{w \in \mathcal{W}_{2}:G_{w} \neq \text{tree}} X_{w} \right] \to 0.
\]
Finally, to get the leading terms in the expectation and variance expressions, among the words $w \in \mathcal{W}_{2}$ we only need to focus on those where $G_{w}$ is a tree with $k$ or $k-1$ nodes. 
Call the collection of these words $ \mathcal{W}_{2,1}$ and $ \mathcal{W}_{2,2}$, and the corresponding sums $W_{2,1}$ and $W_{2,2}$. 
It can be shown then under  $\mathbb{P}_{0,n}$,
\begin{align*}
\E[W_{2,1}] = (1+o(1))\frac{\alpha_{2,2k}}{\pnav} + o(1),
\quad
& \Var[W_{2,1}] = (1+o(1))\frac{v_{2k}}{n^2\pnav^3} + \tilde\varepsilon_{2,1},\\
\E[W_{2,2}] = (1+o(1))\frac{\alpha_{3,2k}}{n\pnav^2},
\quad
& \Var[W_{2,2}] = \tilde\varepsilon_{2,2},
\end{align*} 
and $\tilde\varepsilon_{2,i}\to 0$ for $i=1,2$ when $n^2\pnav^3\to \infty$ and $k=o(\min(\log(n^2\pnav^3), \sqrt{\log n}))$.
On the other hand, under $\mathbb{P}_{1,n}$ (conditioning on group assignment) by arguments similar to the proof of part (i) of Theorem \ref{thm:alt} that for $i=1,2$, 
\[
W_{2,i} \stackrel{d}{=} \tau_{i} + \Xi_{i}
\]
where $\tau_{i}$ has the same asymptotic distribution as 
$W_{2,i}|\mathbb{P}_{0,n}$
and  
\[
\E[\Xi_{i}] = O\left(  \frac{t^{2k}}{n}\right) \quad \mbox{and} \quad
\Var [\Xi_{i} ]= O\left( \frac{\left(C_1k\right)^{C_2k}}{n} \right).
\]
for some universal constants $C_{1}$ and $C_{2}$. 
The proof of this step is very similar to that of part (i) of Theorem \ref{thm:alt}. 
This completes the proof of Theorem \ref{thm:even}.



\section{Conclusion}
\label{sec:conclusion}

The present paper studies computationally feasible and adaptive procedures for testing an \ER~model against a symmetric stochastic block model with two blocks.
By a careful examination of joint asymptotic normality of linear spectral statistics of power functions and their connections to signed cycles, we find an intimate connection between the spectrum of the graph adjacency matrix and the signed cycle statistics through Chebyshev polynomials, which hold under both the null and any local alternative.
Together with the understanding that signed cycles determine the asymptotic likelihood ratio for the testing problem of interest \cite{Ban16}, this connection enables us to obtain sharp asymptotic optimal power of the likelihood ratio test via computationally tractable tests based on linear spectral statistics under an appropriate growth condition on the average degree, namely $n^2\pnav^3\to\infty$. 
In addition, we have also proposed tests which achieve nontrivial power in the contiguous regime and consistency in the singular regime as long as $n\pnav\to\infty$.
Furthermore, we have considered and proposed adaptive tests which are computationally tractable, completely data-driven, and only suffer relatively small loss in power.
A graphical illustration of our findings can be found in Fig.~\ref{fig:power}.
An important question for future research is whether it is possible to achieve sharp asymptotic optimal power of the likelihood ratio test using a procedure of polynomial time complexity when $n^{-1}\ll \pnav \leq Cn^{-2/3}$.

\begin{figure}[th]
\centering
\includegraphics[width = 0.7\textwidth]
{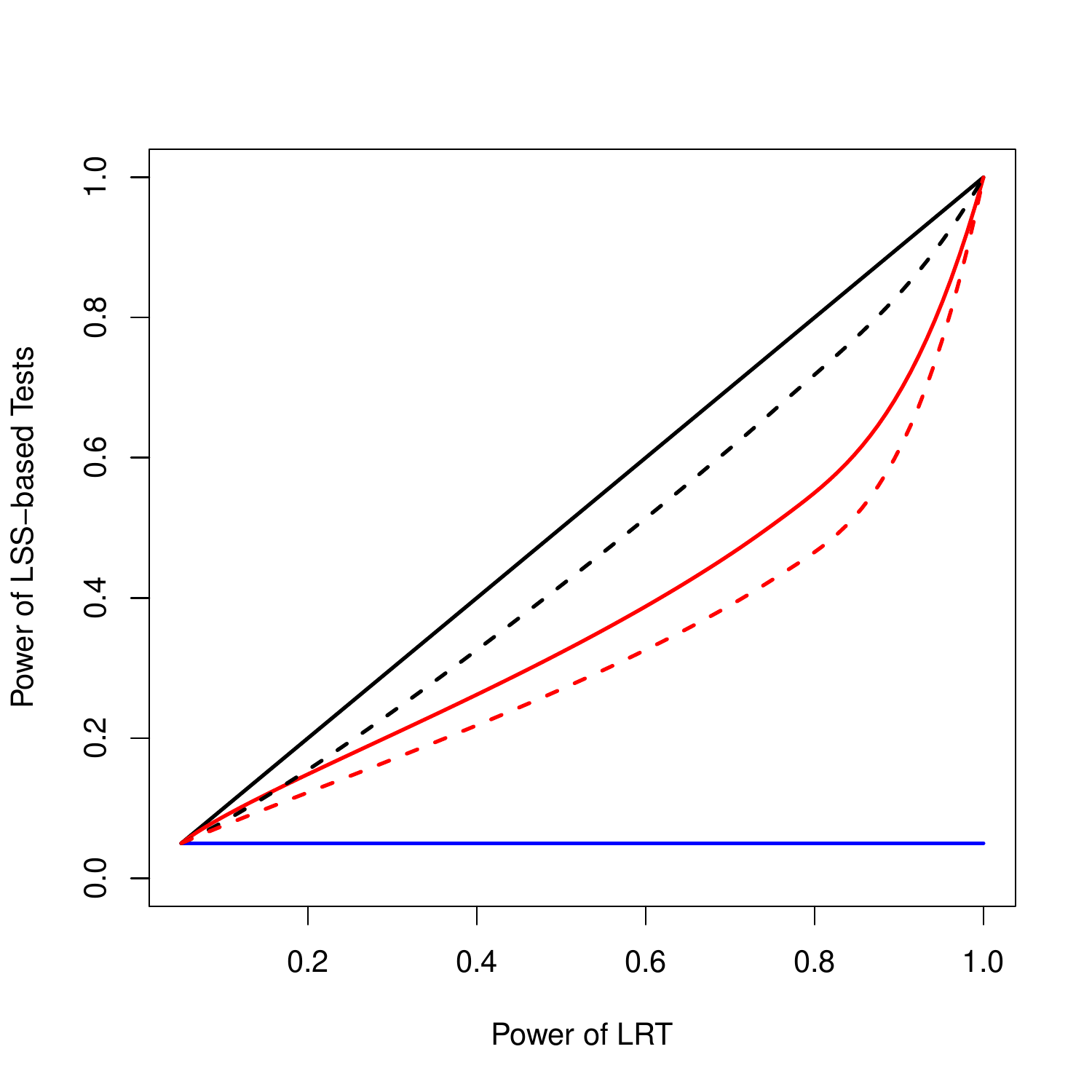}
\caption{Asymptotic power achievable by tests based on linear spectral statistics vs.~the asymptotic power of the likelihood ratio test for \eqref{test:sbm}. All tests are of level $\alpha =  0.05$.
Black curves (solid and dashed) are under growth condition $n^2\pnav^3 \to \infty$. Red curves (solid and dashed) are under growth condition $n\pnav\to\infty$. The blue line corresponds to trivial power.
The dashed curves are achievable by adaptive tests which do not require any knowledge about the hypotheses in \eqref{test:sbm}. 
They correspond to \eqref{eq:adapt-odd} (red) and \eqref{eq:adapt-all} (black) with $\varepsilon=0.15$.\label{fig:power}}
\end{figure}

\appendix 

\begin{center}
{\Large \textbf{Appendix}}
\end{center}

The following appendix presents detailed proofs of the main results.
As we have mentioned earlier, they are highly influenced by the techniques introduced in \citet{AZ05} and \citet{AGZ}. 
In the rest of this section, we first introduce important combinatorics background in Section \ref{subsec:prelim}.
Next, Section \ref{subsec:proof} provides the proofs of all the major results.

 

\section{Preliminary combinatorics results}
\label{subsec:prelim}

Section \ref{subsec:word} is dedicated to building up the preliminary ideas about words, sentences, CLT sentences and state a few important lemmas required in the proofs. 
In Section \ref{subsec:FK} we present the ideas about {\FK} sentences and related topics. These ideas will be in the center of our proofs. 
In Section \ref{subsec:unicyclic} we study unicyclic graphs. 
These results play a fundamental role in finding out the exact formula for the covariances of the linear spectral statistics and signed cycles. 
Most of the definitions and preliminary results in this section can be found in \cite{AZ05} and \cite{AGZ}, which we include here mainly for the proofs to be self-contained.
The new results here are Lemma \ref{lem:appendix} (which first appeared in \cite{Ban16}), Proposition \ref{prop:uniword} and Lemma \ref{lem:cltwordpaircnt}.
Informed readers may focus on these new results only while skipping the rest of this section.

\subsection{Words, sentences and their equivalence classes}
\label{subsec:word}

In this part we give a very brief introduction to words, sentences and their equivalence classes essential for the combinatorial analysis of random matrices. 
The definitions are taken from \citet{AGZ} and \citet{AZ05}.
For more general information, see \cite[Chapter 1]{AGZ} and \cite{AZ05}. 
\begin{definition}[$\mathcal{S}$ words]
Given a set $\mathcal{S}$, an $\mathcal{S}$ letter $s$ is simply an element of
$\mathcal{S}$. An $\mathcal{S}$ word $w$ is a finite sequence of letters $s_1
\ldots s_k$, at least one letter long.
An $\mathcal{S}$ word $w$ is \emph{closed} if its first and last letters are the same. In this paper $\mathcal{S}=\{1,\ldots,n\}$ where $n$ is the number of nodes in the graph.
\end{definition}
Two $\mathcal{S}$ words
$w_1,w_2$ are called \emph{equivalent}, denoted $w_1\sim w_2$, if there is a bijection on $\mathcal{S}$ that
maps one into the other.
For any word $w = s_1 \ldots s_k$, we use $l(w) = k$ to denote the \emph{length} of $w$, define
the \emph{weight} $wt(w)$ as the number of distinct elements of the set ${s_1,\ldots , s_k
}$ and the
\emph{support} of $w$, denoted by $\mathrm{supp}(w)$, as the set of letters appearing in $w$. 
With any word $w$ we may associate an undirected graph, with $wt(w)$ vertices and at most $l(w)-1$ edges,
as follows.
\begin{definition}[Graph associated with a word] 
	\label{def:graphword}
Given a word $w = s_1 \ldots s_k$,
we let $G_w = (V_w,E_w)$ be the graph with set of vertices $V_w = \mathrm{supp}(w)$ and (undirected)
edges $E_w = \{\{s_i, s_{i+1}
\}, i = 1,\ldots ,k - 1
\}.$
\end{definition}
The graph $G_w$ is connected since the word $w$ defines a path connecting all the
vertices of $G_w$, which further starts and terminates at the same vertex if the word
is closed. 
We note that equivalent words generate the same
graphs $G_w$ (up to graph isomorphism) and the same passage-counts of the edges. 
Given an equivalence class $\mathbf{w}$, we shall sometimes denote $\#E_{\mathbf{w}}$ and $\#V_{\mathbf{w}}$ to be the common number of edges and vertices for graphs associated with all the words in this equivalence class $\mathbf{w}$. 

\begin{definition}[Weak Wigner words]
	\label{def:weakwigner}
Any word $w$ will be called a \emph{weak Wigner word} if the following conditions are satisfied:
\begin{enumerate}
\item $w$ is closed.
\item $w$ visits every edge in $G_{w}$ at least twice. 
\end{enumerate}
\end{definition}
Suppose now that $w$ is a weak Wigner word. If $wt(w) = (
l(w) + 1)/2$, then we
drop the modifier ``weak" and call $w$ a \emph{Wigner word}. (Every single letter word is
automatically a Wigner word.) 
Except for single letter words, each edge in a Wigner word is traversed exactly twice.
If $wt(w) = (l(w)-1)/2$, then we call $w$ a \emph{critical
weak Wigner word}.

We now move to definitions related to sentences.

\begin{definition}[Sentences and corresponding graphs]
	\label{def:sentence}
A sentence $a=[w_i]_{i=1}^{m}=[[\alpha_{i,j}]_{j=1}^{l(w_i)}]_{i=1}^{m}$ is an ordered collection of $m$ words of length $(l(w_1),\ldots,l(w_m))$ respectively. We define the graph $G_a=(V_a,E_a)$ to be the graph with 
\[
V_a= \mathrm{supp}(a),\quad 
E_a= \left\{ \{ \alpha_{i,j},\alpha_{i,j+1}\}| i=1,\ldots,m ; j=1,\ldots, l(w_i)-1 \}  \right\}. 
\] 
\end{definition} 
 
\begin{definition}[Weak CLT sentences]
	\label{def:weakCLT}
A sentence $a=[w_i]_{i=1}^{m}$ is called a \emph{weak CLT sentence}, if the following conditions are satisfied:
\begin{enumerate}
\item All the words $w_i$'s  are closed.
\item Jointly the words $w_i$ visit each edge of $G_a$ at least twice.
\item For each $i\in \{1,\ldots,m \}$, there is another $j\neq i \in \{ 1,\ldots,m\}$ such that $G_{w_i}$ and $G_{w_j}$ have at least one edge in common. 
\end{enumerate}
\end{definition}
Suppose now that $a$ is a weak CLT sentence. If $wt(a)=\sum_{i=1}^{m} \frac{l(w_{i})-1}{2}$, then we $a$ a \emph{CLT sentence}. If $m = 2$ and $a$ is a CLT sentence,
then we call $a$ a \emph{CLT word pair}.

We now introduce an additional notion regarding permutation which will be important in our computations. 
\begin{definition}\label{def:cyclic}
Suppose we have a word $w=(\alpha_1,\ldots, \alpha_{k})$ of length $k$ and a permutation $\sigma$ of the set $\{ 1,\ldots,k\}$, 
we define $w^{\sigma}$ to be the word $(\alpha_{\sigma(1)},\ldots, \alpha_{\sigma(k)})$. 
If $\sigma$ is a power of the cycle $(123\ldots k)$, we call $\sigma$ a cyclic permutation and the corresponding word $w^{\sigma}$ to be a cyclic permutation of $w$. 
\end{definition}

We now state a few propositions and lemmas which will be used in our proof. These results, except for that of Lemma \ref{lem:appendix}, can be found in \cite{AZ05}.
We at first state an elementary yet general lemma about a forest $G$ and a word $w$ admitting the interpretation of a walk on $G$.

\begin{lemma}[The parity principle. Lemma 4.4 in \cite{AZ05}]
	\label{lem:parity}
Let $G$ be a forest and $e$ be an edge of $G$.
Let $w$ be a word admitting the interpretation as a walk on $G$. 
Let $w_*$ be the unique
path in $G$ with initial and terminal vertices coinciding with those of $w$. Then
the word/walk $w$ visits the edge $e$ an odd number of times if and only if  $w_*$ visits $e$.
\end{lemma}

The following facts about critical weak Wigner words are important for analyzing trace of even powers and for proving Theorem \ref{thm:even}.
\begin{proposition}
	[Proposition 4.8 in \cite{AZ05}]
	\label{prop:weakwigrep}
Let $w$ be a critical weak Wigner word and let $G=(V,E)= G_{w}=(V_{w},E_{w})$. Then the following hold:
\begin{enumerate}
\item $G$ is connected.
\item Either $\#V-1= \# E$ or $\#V=\#E$.
\item If $\#V-1= \# E$, then:
\begin{enumerate}[a)]
\item $G$ is a tree.
\item With exactly one exception $w$ visits each edge of $G$ exactly twice.
\item But $w$ visits the exceptional edge exactly four times.
\end{enumerate}
\item If $\#V = \#E$, then:
\begin{enumerate}[a)]
\item $G$ is not a tree. 
\item $w$ visits each edge of $G$ exactly twice.
\end{enumerate}
\end{enumerate}
\end{proposition}
 
The following proposition is crucial for verifying Wick's formula for traces of powers which will in turn prove the CLTs in Theorem \ref{thm:null}.

\begin{proposition}[Proposition 4.9 in \cite{AZ05}]
	\label{prop:cltrep}
Let $a=[w_i]_{i=1}^{m}$ be a weak CLT sentence containing $m$ words. Then we have the following:
\begin{enumerate}
\item 
\[
wt(a)\le \sum_{i=1}^{m} \frac{l(w_i)-1}{2}.
\]
\item Suppose the equality holds i.e. $a$ is a CLT sentence. Then the words of the sentence a are perfectly matched
in the sense that for all $i$ there exists unique $j$ distinct from $i$ such that $w_i$ and $w_j$
have at least one letter in common. In particular, $m$ is even.
\end{enumerate}
\end{proposition}

The proof of Proposition \ref{prop:cltrep} in \cite{AZ05} is based on the next important Lemma. 
\begin{lemma}[Lemma 4.10 in \cite{AZ05}]
	\label{lem:cltrep}
Let $a=[w_i]_{i=1}^{m}$ be weak CLT sentence containing $m$ words. Put $G=G_{a}$. Let $k$ be the number of connected components of $G$. Then 
\begin{enumerate}
\item $k\le [\frac{m}{2}]$. 
\item $wt(a) \le k-m+ \left[ \frac{\sum_{i=1}^{n}l(w_i)}{2}\right].$
\end{enumerate}
Here $[y]$ denotes the largest integer less than or equal to $y$.
\end{lemma}

Lemma \ref{lem:expcltpair} below gives an explicit description of the structure of the CLT word pairs.
\begin{lemma}[Proposition 4.12 in \cite{AZ05}]
	\label{lem:expcltpair}
Let $a = [w, x]$ be a CLT word pair and put $G=(V,E)= G_{a}=(V_{a},E_{a})$. For any $e \in E$ let $\nu(e,w)$ (respectively $\nu(e,x)$) denote the number of time the edge $e$ is visited by the word $w$ (respectively $x$). Then the following hold:
\begin{enumerate}
\item $G$ is connected.
\item $\#V-1=\#E$ or $\#V=\#E$. 
\item If $\#V-1=\#E$, then:
\begin{enumerate}[a)]
\item $G$ is a tree.
\item For all $e\in E$, $\nu(e,w)$ and $\nu(e,x)$ are even.
\item There is an unique $e_{0} \in E$ such that $\nu(e_0,w)=\nu(e_0,x)=2$. 
\item For all $e \in E\backslash\{e_0\}$, $\nu(e,w)+\nu(e,x)=2$.
\item Both $w$ and $x$ are Wigner words.
\end{enumerate}
\item If $\#V=\#E$, then:
\begin{enumerate}[a)]
\item $G$ is not a tree.
\item For all $e\in E$ we have $\nu(e,w)+\nu(e,x)=2$.
\item There is at least one edge $e \in E$ such that $\nu(e,w)=\nu(e,x)=1$.
\end{enumerate}
\end{enumerate}
\end{lemma}

In the next part we shall be able to enumerate all the CLT word pairs explicitly. 
We end the discussion of this part by Lemma \ref{lem:appendix} which will crucial for proving the CLTs when the power $k$ slowly diverges to infinity. 
Its first proof can be found in \cite{Ban16}. 
However, the embedding algorithm used in the proof will also be useful in the proof of Theorem \ref{thm:alt}. 
So we spell out the proof of Lemma \ref{lem:appendix} here to make the present paper self-contained. 
To begin with, we give an upper bound on the number of equivalence classes corresponding to weak Wigner words.

\begin{lemma}[Lemma 2.1.23 in \cite{AGZ}]
	\label{lem:bounded}
Let $\mathcal{W}_{k,t}$ collect the equivalence classes corresponding to all weak Wigner words $w$ of length $k+1$ with ${wt}(w)=t$. Then for $k\geq \min(2,2t-2)$,
\[ 
\# \mathcal{W}_{k,t} \le 2^k k^{3(k-2t+2)}.
\] 
\end{lemma} 

We now state Lemma \ref{lem:appendix} and its proof.

\begin{lemma}\label{lem:appendix}
Let $\mathcal{A} = \mathcal{A}^{n}_{m,t}(l_1,\dots,l_m)$ be the set of weak CLT sentences $a=[w_i]_{i=1}^{m}$ such that the letter set is $\{1,\dots,n\}$, $\#V_{a}=t$ and $l(w_i) = l_i$ for $i=1,\dots,m$. 
If $l_i\geq 3$ for $i=1,\dots,m$, then 
\begin{equation}\label{eqn:bddweakclt}
\# \mathcal{A} \le n^{t}2^{l}\left(C_1 l\right)^{C_2 m} l^{3(l-2t)}
\end{equation}
where $l = \sum_{i=1}^m l_i$ and $C_1,C_2>0$ are numeric constants.
\end{lemma}
\begin{proof}	
Let $a=[w_{i}]_{i=1}^{m}$ be a weak CLT sentence in $\mathcal{A}$.
Suppose $G_a$ have $\mathcal{C}(a)$ many connected components. 
Then the sentence $a$ induces a partition $\eta(a)$ of the set $\{1,\dots, m\}$: we put $i$ and $j$ in same block of $\eta(a)$ if and only if $G_{w_i}$ and $G_{w_j}$ share an edge. 
We first fix such a partition $\eta$ and bound the number of sentences $a$ such that $\eta(a)=\eta$.  
Let $\mathcal{C}(\eta)$ be the number of blocks in $\eta$, then for any $a$ with $\eta(a)=\eta$, we have $\mathcal{C}(\eta)=\mathcal{C}(a)$. From now on we denote $\mathcal{C}(\eta)$ by $\mathcal{C}$ for convenience.

Let $a$ be any weak CLT sentence such that $\eta(a)=\eta$. We now propose an algorithm to embed $a$ into $\mathcal{C}$ ordered closed words $(W_1,\ldots,W_{\mathcal{C}})$ such that the equivalence class of each $W_j$ belongs to $\mathcal{W}_{L_j,t_j}$ for some numbers $L_j$ and $t_j$.
This would then allow us to invoke Lemma \ref{lem:bounded} to upper bound the number of such sentences.
Let $B_1,\ldots,B_\mathcal{C}$ be the blocks of the partition $\eta$ ordered in the following way. 
Let $i_j(1) = \min\{i: i \in B_j\}$ and we order the blocks $B_j$ such that $i_1(1)< i_2(1) < \ldots < i_{\mathcal{C}}(1)$. 
Note that this ordering is unique for any given partition $\eta$.
For $j=1,\dots, \mathcal{C}$, let 
\[
B_j=\{ i_j(1)<i_j(2)<\ldots<i_j(m_j) \}.
\] 
Here $m_j$ denotes the number of elements in $B_j$ and so $\sum_{j=1}^\mathcal{C}m_j = m$. 

\paragraph{An embedding algorithm:}
For each $B_j$ we embed the sentence $a_j := [w_{i_j(k)}]_{1\leq k\leq m_j}$ into a weak Wigner word $W_j$ sequentially in the following manner.

\textbf{Step 1.} 
Let $S_1=\{i_j(1)\}$ and $\mathfrak{w}_1=w_{i_j(1)}$.

\textbf{Step 2.}
For $k=1,\dots, m_j-1$, perform the following:
\begin{enumerate}
\item Write $\mathfrak{w}_k = \alpha^k_1 \dots \alpha^k_{l(\mathfrak{w}_k)}$. 
Find an index $n_k \in B_i\backslash S_k$ such that the following two conditions hold:
\begin{enumerate}[(a)]
	\item $G_{\mathfrak{w}_k}$ and $G_{w_{n_k}}$ shares at least one edge $e = \{\alpha^k_{\nu_1}, \alpha^k_{\nu_1+1}\}$.
	\item $\nu_1$ is the minimum among all such choices.
\end{enumerate}

\item Write $w_{n_k} = \beta^k_1\dots \beta^k_{l(w_{n_k})}$ and let $\{\beta_{\nu_2}^k, \beta^k_{\nu_2+1}\}$ be the first time the edge $e$ appears in $w_{n_k}$. Note that $\alpha^k_{\nu_1} \in \{\beta_{\nu_2}^k, \beta^k_{\nu_2+1}\}$. 
Let ${\nu_1'} \in \{\nu_2, \nu_2 + 1\}$ such that $\alpha^k_{\nu_1} = \beta^k_{{\nu_1'}}$.
If $\beta^k_{\nu_2} = \beta^k_{\nu_2+1}$, set $\nu_1' = \nu_2$.

\item Define $S_{k+1} := S_k \cup \{n_k\}$ and
\begin{equation*}
\mathfrak{w}_{k+1} := \alpha^k_1\, \ldots \, \alpha^k_{\nu_1}\,
\beta^k_{\nu_1'+1}\, \ldots\, \beta^k_{l(w_{n_k})}
\,\beta^k_{2}\,\ldots\, \beta^k_{\nu_1'}\,\alpha^k_{\nu_1+1}\,\ldots\, \alpha^k_{l(\mathfrak{w}_{k})}.
\end{equation*}
Let $\tilde{a}_k = [\mathfrak{w}_k, w_{n_k}]$ be the sentence consisting of the two words $\mathfrak{w}_k$ and $w_{n_k}$.
Then the edges in $G_{\tilde{a}_k}$ are preserved along with their passage counts in $G_{\mathfrak{w}_{k+1}}$.
By induction and the fact that $w_{n_k}$ is closed, $\mathfrak{w}_{k+1}$ is a closed word. 
\end{enumerate}

\textbf{Step 3.} Return $W_j = \mathfrak{w}_{m_j}$.

\bigskip

By induction $W_j$ is a weak Wigner word. 
Its length $L_j = \sum_{k=1}^{m_j} l(w_{i_j(k)}) = \sum_{i\in B_j} l(w_i)$, 
and its weight $t_j \le\frac{L_j+1}{2}$.

Note that the embedding algorithm has actually defined a function $f$ which maps any weak CLT sentence $a$ into $\mathcal{C} = \mathcal{C}(a)$ ordered weak Wigner words $(W_1,\ldots,W_{\mathcal{C}})$.
Unfortunately $f$ is not injective. 
So given $(W_1,\ldots,W_{\mathcal{C}})$ we are to find an upper bound on the cardinality of the following set 
\[
f^{-1}(W_1,\ldots,W_{\mathcal{C}}):=\{a| f(a)=(W_1,\ldots,W_{\mathcal{C}})\}.
\]

Note that $\mathcal{C}$ is the number of blocks in $\eta = \eta(a)$.
However, in general $(W_1,\ldots,W_{\mathcal{C}})$ specifies neither the partition $\eta$ nor the order in which the words are concatenated within each block $B_j$ of $\eta$. 
So as an intermediate step we fix a partition $\eta = \{B_1,\dots, B_\mathcal{C}\}$ with $\mathcal{C}$ many blocks and an order of concatenation $\mathcal{O}$. 
Observe that any order of concatenation with a given $\eta$ can now always be written as
\[
\mathcal{O}= (\sigma_1(B_1),\ldots,\sigma_{\mathcal{C}}(B_\mathcal{C}))
\]
where for each $j$, $\sigma_{j}(B_j)$ is a permutation of the elements in $B_j$.
Now we give a uniform upper bound to the cardinality of the following set 
\[
f^{-1}_{\eta,\mathcal{O}}(W_1,\ldots,W_{\mathcal{C}})
:=\left\{a \,|\, \eta(a)=\eta,\, \mathcal{O}(a)= \mathcal{O} \,\,\text{and}\,\,f(a)=(W_1,\ldots,W_{\mathcal{C}}) \right\}.
\]
According to the embedding algorithm, 
any word $W_j$ is formed by recursively applying Step 2 to $\mathfrak{w}_{k}$  and $w_{n_k}$ for $1 \le k \le m_j-1$. 
Given a word 
\[
\mathfrak{w}_{k+1} = \alpha^{k+1}_1 \dots \alpha^{k+1}_{l(\mathfrak{w}_{k+1})},
\]
we want to find out the number of two words combinations $\mathfrak{w}$ and $\mathfrak{w}'$ such that applying Step 2 to $\mathfrak{w}$ and $\mathfrak{w}'$ gives $\mathfrak{w}_{k+1}$ as the output.
This is equivalent to choosing three positions $a<b<c$ in the set $\{1,\dots, l(\mathfrak{w}_{k+1})\}$ such that $\alpha^{k+1}_a = \alpha^{k+1}_c$.
Once these three positions are chosen, we can set
\begin{align*}
\mathfrak{w} & = \alpha^{k+1}_1\,\dots\, \alpha^{k+1}_a\, \alpha^{k+1}_{c+1}\,\dots \, \alpha^{k+1}_{l(\mathfrak{w}_{k+1})}
\quad \mbox{and} \quad
\mathfrak{w}' = \alpha^{k+1}_b\,\dots \alpha^{k+1}_c\,\alpha^{k+1}_{a+1}\,\dots\,\alpha^{k+1}_{b}.
\end{align*}
Note that the length of each word of interest is at least three, and so the above construction is always feasible.
The total number of distinct $(a,b,c)$ triplets is upper bounded by $l(\mathfrak{w}_{k+1})^3 \leq (\sum_{i=1}^m l(w_i))^3 = l^3$.
In addition, for each block $B_j$, Step 2 of the embedding algorithm is run $m_j-1$ times.
Together with $\mathcal{C}$ different blocks, the foregoing arguments lead to
\[
\# f^{-1}_{\eta,\mathcal{O}}(W_1,\ldots,W_{\mathcal{C}}) \le \prod_{j=1}^\mathcal{C} l^{3(m_j-1) } \leq l^{3m}.
\]
Furthermore, observe that there are at most $m^{m}$ different partitions and for each partition there are at most $\prod_{i=1}^{\mathcal{C}}m_j!\le m^{m}$ choices of $\mathcal{O}$. So
\begin{equation}
	\label{bdd1}
\# f^{-1}(W_1,\ldots,W_{\mathcal{C}})\le m^{2m}l^{3m}
\le \left(D_1 l\right)^{D_2m}
\end{equation}
for some numeric constants $D_1$ and $D_2$.

\medskip

Now we fix the sequence $(L_j,t_j)$ for $j=1,\dots, \mathcal{C}$ and find an upper bound to the number of the tuples $(W_1,\ldots, W_{\mathcal{C}})$. From Lemma \ref{lem:bounded} we know the number of choices of $W_j$ is bounded by $2^{L_j-1}(L_j-1)^{L_j-2t_j+1}n^{t_j}$. So the total number of choices for the tuples $(W_1,\ldots, W_{\mathcal{C}})$ is bounded by 
\begin{equation}
	\label{bdd2}
2^{\sum_{j=1}^{m}l(w_j)}\prod_{j=1}^{\mathcal{C}}(L_j-1)^{3(L_j-2t_j+1)}n^{t_j}
\le 2^l n^{t} l^{3(l-2t + m)}.
\end{equation}
Now the number of choices $(L_j,t_j)_{j=1}^\mathcal{C}$ such that $\sum_{j=1}^{\mathcal{C}}L_{j}= l$ and $\sum_{j=1}^{\mathcal{C}}t_j = t$ are bounded by 
\begin{equation}\label{bdd3}
\binom{l-1}{\mathcal{C}-1} \binom{t-1}{\mathcal{C}-1}\le  l^{2m}.
\end{equation}
Here the inequality follows since $\mathcal{C}\le m$ and $t\le l$. 
Finally we use the fact that $1\le \mathcal{C}\le m$ and (\ref{bdd1}), (\ref{bdd2}) and (\ref{bdd3}) to conclude
\begin{equation}
\begin{split}
&\# \mathcal{A} \le \left(D_1l\right)^{D_2m} \times  2^{l}n^{t}l^{3(l-2t+m)} \times l^{2m} 
\le n^{t}2^{l}\left(C_1 l\right)^{C_2 m} l^{3(l-2t)}
\end{split}
\end{equation}
as claimed.
\end{proof}

\subsection{{\FK} enumeration}
\label{subsec:FK}

We now introduce the notion of {\FK} sentences. 
It was the key idea underlying the proof of Lemma \ref{lem:bounded} (Lemma 2.1.23 in \cite{AGZ}), which in turn was crucial for the proof of Lemma \ref{lem:appendix}. 
In addition, as we shall show below, it plays an important role in getting the exact enumeration of CLT word pairs and provides some general insight about the covariance structure between different linear spectral statistics. 
Most of the materials in this Subsection are borrowed from Section 7 of \cite{AZ05} and Chapter 1 of \cite{AGZ}. 
The original idea of {\FK} sentences dates back to \citet{furedi1981eigenvalues}.

\begin{definition}[FK sentences]
Let $a= [w_{i}]_{i=1}^{m}$ be a sentence consisting of $m$ words. We say that $a$ is an \emph{FK sentence} under
the following conditions:
\begin{enumerate}
\item $G_a$ is a tree.
\item Jointly the words/walks $w_i$, $i=1,\dots,m$, visit no edge of $G_a$ more than twice.
\item For $i = 1,\ldots, m-1$, the first letter of $w_{i+1}$ belongs to $\cup_{j=1}^{i} \mathrm{supp}(w_j)$.
\end{enumerate}
We say that $a$ is an \emph{FK word} if $m=1$.
\end{definition}

By definition, any word admitting interpretation as a walk on a forest visiting no edge of the forest more than twice is automatically an FK word. 
The constituent words of an FK sentence are FK words. 
If an FK sentence is at least two words long, then the result of dropping the last word is again an FK sentence. 
If the last word of an FK sentence is at least two letters long, then the
result of dropping the last letter of the last word is again an FK sentence.

\begin{definition}[The stem of an FK sentence]
Given an FK sentence $a=[w_{i}]_{i=1}^{m}$, we define $G_{a}^{1}=(V_{a}^{1},E_{a}^{1})$ to be the
subgraph of $G_a=(V_a,E_a)$ with $V^{1}_{a}= V_a$ and $E_{a}^{1}$ equal to the set of edges $e \in E_{a}$
such that the words/walks $w_i$, $i=1,\dots,m$, jointly visit $e$ exactly once.
\end{definition}

The following lemma characterizes the exact structure of an FK word.

\begin{lemma}[Lemma 2.1.24 in \cite{AGZ}]
	\label{lem:FKembed}
Suppose $w$ is an FK word. Then there is exactly one way to write $w=w_1,\ldots,w_r$ where each $w_i$ is a closed Wigner word and they are pairwise disjoint. 
\end{lemma} 

Let $\alpha_i$ be the first letter of the word $w_i$, we declare the word $\alpha_{1},\ldots,\alpha_r$ to be the \emph{acronym} of the word $w$ in Lemma \ref{lem:FKembed}.  
Since the counts of Wigner words are well known, one can explicitly enumerate the equivalence classes of all FK words as follows.

\begin{proposition}
	\label{prop:FKcount}
Let $F(m,r)$ be the set of equivalence classes of all FK words of length $m$ with acronym of size $r\le m.$ Then 
\[
\#F(m,r)= f(m,r)
\]
and $\sum_{r=1}^m f(m,r)\le 2^{m-1}$.
Here $f(m,r)$ is as defined in (\ref{gen:f}).
\end{proposition}
\begin{proof}
The proof follows immediately from the proof of Lemma \ref{lem:FKembed} (See the proof of Lemma 2.1.24 in \cite{AGZ}).
\end{proof}

\paragraph{FK syllabification}
Our interest in FK sentences is mainly due to the fact that any word $w$ can be parsed into an FK sentence sequentially.
In particular, one declares a new word at each time when not doing so would prevent the sentence formed up to that point from being an FK sentence. 
Formally, we define the FK sentence corresponding to any given word $w$ in the following way. Suppose the word $w$ is formed by $m$ letters. We declare any edge $e \in E_{w}$ to be new if $e=\{\alpha_{i},\alpha_{i+1}\}$ and $\alpha_{i+1} \notin \{\alpha_1,\ldots,\alpha_{i} \}$ otherwise we declare $e$ to be old. 
We now construct the FK sentence $w'$ corresponding to the word $w$ by breaking the word at each position of an old edge and the third and all subsequent positions of a new edge. 
Observe that any old edge gives rise to a cycle in $G_{w}$. As a consequence, by breaking the word at the old edge we remove all the cycles in $G_w$. 
On the other hand, all new edges are traversed at most twice as we break at their third and all subsequent occurrences. 
It is easy to see that the graph $G_{w'}$ remains connected since we are not deleting the first occurrence of a new edge. 
As a consequence, the graph $G_{w'}$ is a tree where every edge is traversed at most twice. 
Furthermore, by the definition of old and new edges, 
the first letter in the second and any subsequent word in $w'$ belongs to the support of all the previous ones. 
Therefore, the resulting sentence $w'$ is an FK sentence. 
Note that the FK syllabification preserves equivalence, i.e., if $w \sim x$ then the corresponding FK sentences $w'\sim x'$. 

\medskip

The discussion about FK syllabification shows that all words can be uniquely parsed into an FK sentence. 
Hence we can use the enumeration of FK sentences to enumerate words of specific structures of interest.
The following lemma gives an upper bound to the number of ways one FK sentence $b$ and one FK word $c$ can be concatenated so that the sentence $[b,c]$ is again an FK sentence.
\begin{lemma}[Lemma 7.6 in \cite{AZ05}]
	\label{lem:FKcon}
Let $b=[w_i]_{i=1}^{m}$ be an FK sentence and $c$ be an FK word such that the first letter in $c$ is in $\mathrm{supp}(b)$. Let $\gamma_1,\ldots,\gamma_r$ be the acronym of $c$ where $\gamma_1 \in \mathrm{supp}(b)$. Let $l$ be the largest index such that $\gamma_{l} \in \mathrm{supp}(b)$ and write $d=\gamma_1,\ldots,\gamma_{l}$. Then the sentence $[b,c]$ is an FK sentence if and only if the following conditions are satisfied:
\begin{enumerate}
\item $d$ is a geodesic in the forest $G^1_
b$.
\item $\mathrm{supp}(b) \cap \mathrm{supp}(c)=\mathrm{supp}(d)$.
\end{enumerate}
Here a geodesic connecting $x,y \in G^1_
b$
is a path of minimal length starting at $x$ and
terminating at $y$. 
Further, there are at most $(wt(b))^2$ equivalence classes $[x_{i}]_{i=1}^{m+1}$ such that $b\sim [x_i]_{i=1}^{m}$ and $c\sim x_{m+1}$.
\end{lemma}

The following two lemmas together give an upper bound on the number of equivalence classes corresponding to closed words via the corresponding FK sentences. 
\begin{lemma}[Lemma 7.7 in \cite{AZ05}]
	\label{lem:wordbyFK}
Let $\Gamma(k,l,m)$ denote the set of equivalence classes of FK sentences $a=[w_i]_{i=1}^{m}$ consisting of $m$ words such that $\sum_{i=1}^{m}l(w_i)=l$ and $wt(a)=k$. Then we have 
\begin{equation}\label{wordbyFK}
\#\Gamma(k,l,m) \le 2^{l-m} \binom{l-1}{m-1}k^{2(m-1)}.
\end{equation}
\end{lemma}

\begin{lemma}[Lemma 7.8 in \cite{AZ05}]
	\label{lem:mbdd}
For any FK sentence $a=[w_i]_{i=1}^{m}$, we have 
\begin{equation}\label{mval}
m =\#E_{a}^{1} - 2wt(a) + 2 + \sum_{i=1}^{m} l(w_i).
\end{equation}
\end{lemma}

\subsection{Unicyclic graphs, CLT word pairs and their enumeration}
\label{subsec:unicyclic}

As discussed earlier, enumeration of the CLT word pairs gives us the exact variance expression in Theorem \ref{thm:null}. 
In this part we shall present a few more concepts and finally give an explicit enumeration of the CLT word pairs based on {\FK} enumeration. 
Most of the concepts in this part can be found in Section 8 in \cite{AZ05}.  

\begin{definition}[Bracelets]
We say a graph $G=(V,E)$ is a \emph{bracelet} if there is an enumeration $\alpha_{1},\ldots,\alpha_r$ of $V$ such that 
\begin{equation}
E=\left\{  
\begin{array}{ll}
\{\{\alpha_1,\alpha_1\}\} & \text{if $r=1$}\\
\{\{\alpha_1,\alpha_2 \}  \} & \text{if $r=2$}\\
\{\{\alpha_1,\alpha_2\},\{\alpha_2,\alpha_3 \},\ldots, \{\alpha_{r-1},\alpha_{r}\},\{\alpha_{r},\alpha_1\}  \} & \text{if $r\ge 3$}.
\end{array}
\right.
\end{equation}
We call $r$ the \emph{circuit length} of the bracelet $G$.
\end{definition}
In this paper we shall ignore the bracelets corresponding $r=1$ since they correspond to the diagonal elements of $A_{\mathrm{cen}i}$, $i=1,2$, which are all zeros. 

\begin{definition}[Unicyclic graphs]
A graph $G=(V,E)$ is called \emph{unicyclic} if $\#V=\#E$. 
\end{definition}

In particular any bracelet of length $\neq 2$ is a unicyclic graph, while a bracelet of length $2$ is a tree.
The following proposition describes the structure of unicyclic graphs. 

\begin{proposition}[Proposition 8.2 in \cite{AZ05}]
	\label{prop:unicyclicgraph}
Let $G=(V,E)$ be a unicyclic graph. For each edge $e \in E$ put $G\backslash e = (V,E\backslash \{ e \})$. Let $Z$ be the subgraph of $G$ consisting of all $e \in E$ such that $G\backslash e$ is connected, along with all attached vertices. Let $r$ be the number of edges
of $Z$. Let $F$ be the graph obtained from $G$ by deleting all edges of $Z$. The following
statements hold:
\begin{enumerate}
\item $F$ is a forest with exactly $r$ connected components.
\item If $G$ has a degenerate edge, then $r = 1$.
\item If $G$ has no degenerate edge, then $r \ge 3$.
\item $Z$ meets each connected component of $F$ in exactly one vertex.
\item $Z$ is a bracelet of circuit length $r$.
\item For all $e \in E$ the following conditions are equivalent:
\begin{enumerate}
\item $G\backslash e$ is connected.
\item $G \backslash e$ is a tree.
\item $G \backslash e$ is a forest.
\end{enumerate}
\end{enumerate}
We call $Z$ the bracelet of $G$. We call $r$ the circuit length of $G,$ and each of the connected components of $F$ a pendant tree.
\end{proposition}

We shall see in the proof of Theorem \ref{thm:null} that we need to consider closed words $w$ such that $G_{w}$ is unicyclic. 
We call such words \emph{uniwords}. 
We now provide an explicit way to enumerate uniwords. 
This part of the proof deviates from \cite{AZ05}. 
The argument presented here gives us a unified way to calculate the covariances in \eqref{cov:null:odd} and \eqref{cov:null:even}, and to calculate the covariance between the signed cycles and the LSSs.

\begin{proposition}[Enumeration of uniwords]
	\label{prop:uniword}
Let $\mathfrak{W}_{m+1,r,t}$ denote the set of all closed words with letters taken from $\{1,\dots,n\}$, such that for any $w \in \mathfrak{W}_{m+1,r,t}$, $l(w)=m+1$, $wt(w)=t$ so that $2t-r=m$ and $G_w$ is a unicyclic graph with circuit length $r\ge 3$. 
Then for any $m = o(\sqrt{n})$, 
\begin{equation}
\frac{\#\mathfrak{W}_{m+1,r,t}}{n^{t}}= (1+o(1)) \frac{mf(m,r)}{r}.
\end{equation} 
Here $f(m,r)$ is as defined in (\ref{gen:f}) which by Proposition \ref{prop:FKcount} is the number of equivalence classes of FK words of length $m$ with acronyms of length $r$.
\end{proposition}

\begin{proof}
The proof will be done by creating a multivalued map $\chi$ from the set of FK words of length $m$ having acronym of length $r$ to $\mathfrak{W}_{m+1,r,t}$. We shall enumerate exactly the cardinality of the forward and inverse image of every element. In this proof every FK word will be denoted by $w_{\mathrm{FK}}$ and any closed word will be denoted by $w$ to make the distinction.

First start with any FK word $w_{\mathrm{FK}}$ of length $m$ having acronym of length $r$. We at first construct the ``base" $w_{\mathrm{FK}}^{B} \in \mathfrak{W}_{m+1,r,t} $ of $w_{\mathrm{FK}}$. Let $(\alpha_1,\ldots,\alpha_r)$ be the acronym of $w_{\mathrm{FK}}$. From the proof of Lemma \ref{lem:FKembed}, it is easy to see that the first and the last letter of $w_{\mathrm{FK}}$ is given by $\alpha_1$ and $\alpha_r$ respectively. We take 
\begin{equation}\label{eqn:base}
w_{\mathrm{FK}}^{B}= (w_{\mathrm{FK}},\alpha_1).
\end{equation}
Observe that $w_{\mathrm{FK}}^{B}$ is a closed word of length $m+1$ and it has a bracelet of length $r$ formed by the acronym and removing the bracelet we are left with a forest where each edge in the forest has been traversed exactly twice by $w_{\mathrm{FK}}^{B}$. So the graph $G_{w_{\mathrm{FK}}^{B}}$ is unicyclic hence $w_{\mathrm{FK}}^{B}$ is a uinword. It is also easy to check that $2t-r=m$. So $w_{\mathrm{FK}}^{B} \in \mathfrak{W}_{m+1,r,t}$.

Before going into the construction of the map $\chi$, we introduce a useful notation. Let $w=[\alpha_i]_{i=1}^{l(w)}$ be a closed word. We denote $\check{w}$ to be the word dropping the last letter of $w$.  
Now we construct the multivalued map $\chi$ as follows:
\begin{equation}
\chi(w_{\mathrm{FK}})= \{ w \in  \mathfrak{W}_{m+1,r,t} |~ \exists ~ \sigma ~ \text{so that} ~ \check{w}^{\sigma}\alpha_{\sigma(1)}= w_{\mathrm{FK}}^{B} \}
\end{equation}
Here $\sigma$ is a cyclic permutation of $\{1,\ldots,m\}$ and for any word $w$ we have defined $w^{\sigma}$ in Definition \ref{def:cyclic}. As there are $m$ cyclic permutations $\sigma$ of $\{1,\ldots,m\}$, $\#\chi(w_{\mathrm{FK}})= m.$

Now we shall prove that for any given $w=[\alpha_{i}]_{i=1}^{m+1} \in \mathfrak{W}_{m+1,r,t} $, $\#\chi^{-1}(w)= r$. We start with any word $w \in \mathfrak{W}_{m+1,r,t}$ and drop the last letter to get $\check{w}$. Now consider the bracelet $Z_{w}$ in $G_{w}$. There are $r$ many vertices in $Z_{w}$ from the assumption and let the corresponding letters be $\{ \beta_{1},\ldots,\beta_{r}\}$. The set $\{ \beta_{1},\ldots,\beta_{r} \}$ has cardinality $r$ from the definition of $Z_{w}$. As a consequence, the letters $\beta_i$, $1\le i \le r$ are distinct. Now consider the cyclic permutation $\sigma_{i}$, $1\le i \le r$ such that $\sigma_{i}(1)=\beta_{i}$ and $\sigma_{i}(2)$ is a vertex in the pendent tree meeting $Z_{w}$ at position $\beta_{i}$ if it is not empty. However if the pendent tree meeting $Z_{w}$ at position $\beta_{i}$ is empty we take $\sigma_i$ to be such that $\sigma_{i}(1)=\beta_{i}$ and $\sigma_{i}(2) \in \{ \beta_1,\ldots,\beta_r\}$. From the condition $2t-r=m$ it follows that every edge in $Z_{w}$ has been traversed exactly once and every edge in the forest $F_{w}$ has been traversed exactly twice by the word $w_i$. As a consequence, there are exactly $r$ such permutations $\{\sigma_{i}\}_{1\le i \le r}$.  Consider the closed word $w_i= \check{w}^{\sigma_i}\alpha_{\sigma_i(1)}.$ Let $\beta^{(i,j)}$ be the $j$ th appearing letter of $\{ \beta_1,\ldots,\beta_r\}$ in the permutation $\sigma_i$. Observe that $\beta^{(i,1)}=\beta_i$. The construction of $\sigma_i$ compels the word $w_i$ to be of the following form 
\begin{equation*}
w_{i}= w^{(i,1)}w^{(i,2)}\ldots w^{(i,r)} \beta_{i}.
\end{equation*}
Here $w^{(i,j)}$, $1\le j\le r$ is a Wigner word corresponding to the pendent tree meeting $Z_{w}$ at position $\beta^{(i,j)}$. 
Now given any such $w_{i}$ consider the unique FK word 
\begin{equation*}
W_{\mathrm{FK},i}= w^{(i,1)}w^{(i,2)}\ldots w^{(i,r)}.	
\end{equation*}
All these words $W_{\mathrm{FK},i}$'s are distinct since their starting points are distinct. As a consequence, $\#\chi^{-1}(w)= r$.

There are $f(m,r)$ many equivalence classes corresponding to FK words of length $m$ having acronym of length $r$. So the total number of words corresponding to this class is 
\[
n(n-1)\ldots(n-t+1)f(m,r).
\]
Observing $t \le m$, we get
\[
\frac{\#\mathfrak{W}_{m+1,r,t}}{n^{t}}= (1+o(1)) \frac{mf(m,r)}{r} 
\]
for all $m = o(\sqrt{n})$ as declared.
\end{proof}

Now we state one more property about the CLT word pairs.
Its proof is straightforward and follows from the discussion on p.32 of \cite{AGZ}. We omit the details.
\begin{proposition}\label{prop:finalcltwordpair}
Let $a=[w,x]$ be a CLT word pair such that $G_a$ is not a tree. Then 
\begin{enumerate}
\item The graphs $G_w$ and $G_x$ are both unicyclic.
\item They have common bracelet $Z$.
\item Every edge in the bracelet $Z$ is traversed exactly once by both $w$ and $x$.
\item Let $F_w$ and $F_x$ be the forests corresponding to $G_w$ and $G_x$. Then the common vertices between $F_w$ and $F_x$ are subset of $Z$. In other words, $F_w$ and $F_x$ can't have any common edge.
\item Each edge in $F_w$ is traversed exactly twice by the word $w$ and each edge it $F_x$ is traversed exactly twice by the word $x$.
\end{enumerate}
\end{proposition}

Our last result fixes a uniword $w$ and calculates the number of words $x$ such that $a=[w,x]$ is a CLT word pair.

\begin{lemma}[CLT word pairing]
	\label{lem:cltwordpaircnt}
Fix a uniword $w \in \mathfrak{W}_{m+1,r,t}$. Let $S_{m',t'}(w)$ be the set of words $x$ such that $x \in \mathfrak{W}_{m'+1,r,t'}$ and $a=[w,x]$ is a CLT word pair. 
Then for all $m'=o(\sqrt{n})$,
\[
\frac{\#S_{m',t'}(w)}{n^{t'-r}}=(1+o(1))\,2m'f(m',r).
\]
\end{lemma}
\begin{proof}
Before going into proof at first we observe that as $2t'-r=m'$, given $r$ and fixing $m'$, $t'$  automatically fixed.

Consider the graph $G_{w}$ with bracelet $Z_{w}$. The word $w$ admits a walk on the edges of $G_{w}$. Let $(\beta_1,\ldots,\beta_{r})$ be the vertices in $Z_w$ \emph{ordered} according to their exploration by the walk corresponding to $w$. We consider all FK words of length $m'$ and weight $t'$ $w_{\mathrm{FK}}$ such that $V_{w} \cap V_{w_{\mathrm{FK}}}= V_{Z_{w}}$. Here $V_{Z_{w}}$ is the set of vertices of the graph $Z_{w}$. We denote this set of FK words by $F_{m',t'}(w)$. There are total $f(m',r)$ \emph{equivalence classes} of such words.  
Since we have fixed the acronym and the word $w$, the number of possible choices of such $w_{\mathrm{FK}}$ is given by 
\[
\#F_{m',t'}(w)= f(m',r)(n-t)\ldots(n-t-t'+r+1)= f(m',r) (1+o(1)) n^{t'-r}.
\] 
For any FK word $w_{\mathrm{FK}}$ let us recall its base $w_{\mathrm{FK}}^{B}$ in $\mathfrak{W}_{m'+1,r,t'}$ from (\ref{eqn:base}). Let $w=[\alpha_i]_{i=1}^{l(w)}$ be a closed word. We denote $\check{w}$ to be the word dropping the last letter of $w$.
Now construct the set $S_{m',t'}(w)$ as follows
\[
S_{m',t'}(w)=\cup_{w_{\mathrm{FK} \in F_{m',t'}(w)} }\{x \in \mathfrak{W}_{m'+1,r,t'} ~|~ \exists ~ \sigma ~ \text{so that} ~ \check{w}^{\sigma}\alpha_{\sigma(1)}= w_{\mathrm{FK}}^{B} \}.
\]  
Here $\sigma$ is  either a cyclic permutation of $\{ 1,\ldots, m\}$ or its mirror image. It is easy to observe that there are $2m'$ such $\sigma$'s.

We now prove that there is no over-counting in $S_{m',t'}$. This trivially follows from the proof of Proposition \ref{prop:uniword}. Since the distinct FK words in the inverse image $\chi^{-1}(x)$ of any word $x \in S_{m',t'}(w)$ will have acronyms such that one is a non-trivial cyclic permutation of other. However we only considered FK words corresponding to fixed acronym $(\beta_1,\ldots,\beta_r)$. Now observe that these are the only possible choices of $x$ so that $[w,x]$ is a CLT word pair. Hence the result is proved.  
\end{proof}


\section{Proofs of main results}
\label{subsec:proof}

We shall first prove part (i)--(iii) of Theorem \ref{thm:null}. 
Then we shall prove part (i)--(iii) of Theorem \ref{thm:alt} and finally we shall come back to prove part (iv) of Theorem \ref{thm:null} and part (iv) of Theorem \ref{thm:alt}. The proof of all the subsequent results are given after the completion of the proofs of Theorem \ref{thm:null} and Theorem \ref{thm:alt}.
The proofs of Theorem \ref{thm:llslike} and Proposition \ref{prop:sing} are omitted as they follow directly from Theorems \ref{thm:alt}--\ref{thm:cyclestotrace}.
Before we proceed, we quote the following result on the joint asymptotic normality of signed cycles under both the null and local alternatives, which will be used repeatedly in the rest of this section.
Throughout the rest of this section, we focus on the assortative case of $p_n > q_n$ when proving results under local alternatives, and the proofs are essentially the same for the disassortative case of $p_n<q_n$ due to the second part of the following proposition. 

\begin{prop}[\cite{Ban16}]
	\label{prop:signdistr}
Suppose that as $n\to\infty$, $n\pnav\to\infty$ and $c$ and $t$ are constants.
Then the following results hold:
\begin{enumerate}[(i)]
\item Under $P_{0,n}$, for any $3\le k_1 < \ldots< k_l=o(\log(np_{n,\mathrm{av}}))$, 
\begin{equation}\label{dist:nullcycles}
\left(\frac{C_{n,k_1}(G)}{\sqrt{2k_1}},\ldots, \frac{C_{n,k_l}(G)}{\sqrt{2k_l}}\right) \stackrel{d}{\to} N_l(0,I_l).
\end{equation}

\item Under $P_{1,n}$, for any $3\le k_1 < \ldots< k_l=o (\min(\log(np_{n,\mathrm{av}}),\sqrt{\log(n)}) )$, if $p_n>q_n$,
\begin{equation}\label{dist:altcylces}
\left(\frac{C_{n,k_1}(G)-\mu_1}{\sqrt{2k_1}},\ldots, \frac{C_{n,k_l}(G)-\mu_l}{\sqrt{2k_l}}\right) \stackrel{d}{\to}
N_l(0,I_l)
\end{equation}
where $\mu_i = t^{k_i}$ for $1\le i \le l$.
If $p_n<q_n$, the conclusion holds with $\mu_i = (-t)^i$ for all $i$.
\end{enumerate}
\end{prop}

\subsection{Proof of parts (i)--(iii) of Theorem \ref{thm:null}}
Throughout this subsection, all expectation and variance are taken under $\mathbb{P}_{0,n}$.
 
\subsubsection{Proof of part (i)}
We start with a generic $k=o(\log(n\pnav))$.
Observe that 
\begin{equation}\label{traceexp}
\Tr(A_{\mathrm{cen1}}^{2k+1})= \left(\frac{1}{n\pnav(1-\pnav)}\right)^{\frac{2k+1}{2}}\sum_{w:l(w)= 2k+2 ~ \& ~ w ~\text{closed}} \left[X_{w}\right].
\end{equation}
Here any word $w$ is an ordered pair $(i_0,\ldots,i_{2k+1})$ where the numbers $i_{j}\in \{  1,2,\ldots, n\}$ for $0\le j \le 2k+1$, $i_0 = i_{2k+1}$, and we define $X_{w}:= \prod_{j=0}^{2k}\left(x_{i_j,i_{j+1}}- \pnav\right).$
The proof is divided into following two steps: (1) we figure out a subset of words in the summation which matters for the asymptotic distribution;
(2) we apply the method of moments spelled out in Section \ref{sec:overview} to summation over that subset. 

\smallskip

\noindent\textbf{Step 1:}
At first we prove that the random variable $\Tr(\Acenone^{2k+1})$ does not require any additional centering.
Observe that if $\E[X_{w}] \neq 0$, all the edges in $G_{w}=(V_{w},E_{w})$ have been traversed at least twice. Since $l(w)=2k+2$, this will imply that 
\begin{equation}
\begin{split}
\#E_{w} \le \frac{2k+1}{2}
\quad \Rightarrow \quad  \#E_{w} \le k.
\end{split}
\end{equation}
On the other hand from Lemma \ref{lem:parity},  $G_{w}$ cannot be a tree as the total number of edge traversals on $G_{w}$ by the word $w$ is odd. So $\#E_{w} \ge \#V_{w}$. This forces $\#V_w \le k$.  Denote the set of such words by $\mathrm{No}_{k}$. 
We shall prove that the contribution of these words is negligible. 
In particular,
\begin{equation}\label{NO}
\begin{split}
&\E\left[\left(\frac{1}{n\pnav(1-\pnav)}\right)^{\frac{2k+1}{2}} \sum_{w \in \mathrm{No}_{k} } X_{w}\right]^2\\
&= \left(\frac{1}{n\pnav(1-\pnav)}\right)^{2k+1} \sum_{w,w' \in \mathrm{No}_{k}} \E[X_{w}X_{w'}]\\
&\le  \left(\frac{1}{n\pnav(1-\pnav)}\right)^{2k+1} \sum_{a=[w,w']: w,w' \in \mathrm{No}_{k} } \pnav^{\# E_{a}} .
\end{split}
\end{equation}
Here $a$ is the two word sentence obtained by concatenating $w$ and $w'$. Now there can be two cases.

\underline{Case 1: The words $w_1$ and $w_2$ share an edge.} 
In this case, $a$ is a weak CLT sentence. So we can apply Lemma \ref{lem:appendix} with $m=2$ and $l_1=l_2=2k+2$
to get that the sum in this case is bounded by 
\begin{equation}\label{term1}
\begin{split}
&\left(\frac{1}{n\pnav(1-\pnav)}\right)^{2k+1}
\sum_{\zeta=1}^{2k} n^{\zeta}\pnav^{\zeta} 2^{4k+4} (C_1 k)^{2C_2}(4k+4)^{6(2k+2-\zeta)}\\
& \le \left(\frac{1}{(1-\pnav)}\right)^{2k+1}2^{4k+4}(C_1 k)^{2C_2}(4k+4)^{6}\sum_{\zeta=1}^{2k} \left(\frac{4k+4}{n\pnav}\right)^{6(2k+1-\zeta)}
\end{split}
\end{equation}
where $C_1$ and $C_2$ are known constants. 

Observe that $2k+1-\zeta>1$. 
As a consequence,
the R.S. of (\ref{term1}) is a geometric sum on $\big((\frac{4k+4}{n\pnav} )^{6} \big)^{i}$ with lowest index being $1$. 
We also have $(\frac{4k+4}{n\pnav})^{6} \to 0$ by the assumption $k=o(\log(n\pnav))$. 
As a consequence, the R.S. of (\ref{term1}) can be bounded by 
\begin{equation}\label{term1s}
\left(\frac{1}{(1-\pnav)}\right)^{2k+1}2^{4k+4} (C_1 k)^{2C_2}(4k+4)^{6} C_3 \left(\frac{4k+2}{n\pnav}\right)^{6}.
\end{equation}
Here $C_3$ is another known constant. It is easy to see (\ref{term1s}) goes to zero when $k=o(\log(n\pnav))$.

\underline{Case 2: The words $w_1$ and $w_2$ don't share an edge.}
Let $wt(w_1)=\zeta_1$ and $wt(w_2)=\zeta_2$.  
We shall apply Lemma \ref{lem:bounded} in this case. 
Since both $\zeta_1$ and $\zeta_2$ are less than or equal to $k$. 
The equation $2k+1> 2\zeta-2$ is trivially satisfied. 
Now from Lemma \ref{lem:bounded} a crude upper bound to the number of sentences $a=[w_1,w_2]$ such that $w_1$ and $w_2$ don't share an edge such that $wt(a)=\zeta$ is given by 
\begin{equation}
\begin{split}
&\sum_{\zeta_1} \sum_{\zeta_2= \zeta-\zeta_1} n^{\zeta} 2^{2k+1} (2k)^{3(2k+1-2\zeta_1+2)} \times 2^{2k+1} (2k)^{3(2k+1-2\zeta_2+2)}\\
& ~~~~~ =\sum_{\zeta_1} \sum_{\zeta_2=\zeta-\zeta_{1}} n^{\zeta}2^{4k} (2k)^{3(4k-2\zeta+4)}
 \le n^{\zeta} \zeta^2 2^{4k+2}(2k)^{6(2k-\zeta+2)}.
\end{split}
\end{equation}
Here the factor $\zeta^2$ comes due to the sum. 
Consequently, the sum in this case is bounded by 
\begin{equation}\label{term2}
\begin{split}
&\left(\frac{1}{n\pnav(1-\pnav)}\right)^{2k+1}\sum_{\zeta=1}^{2k} n^{\zeta}\pnav^{\zeta} \zeta^2 2^{4k+2}(2k)^{6(2k-\zeta+2)}\\
&\le \left( \frac{1}{1-\pnav}\right)^{2k+1}\sum_{\zeta=1}^{2k} 2^{4k+2}k^{2}(2k)^{6}\left(\frac{2k}{n\pnav}\right)^{6(2k+1-\zeta)}.
\end{split}
\end{equation}
Now (\ref{term2}) can be analyzed similarly as (\ref{term1}) to get that (\ref{term2}) goes to $0$ also.
This forces the first expression of (\ref{NO}) to go to $0$. As a consequence, we can simply neglect the words in $\mathrm{No}_{k}$.
In particular, any limiting distribution (if exists) of $\Tr(\Acenone^{2k+1})$ is same as the limiting distribution of the following random variable 
\begin{equation}\label{def:y}
Y_{n,2k+1}= \left(\frac{1}{n\pnav(1-\pnav)}\right)^{\frac{2k+1}{2}}\sum_{w \notin \mathrm{No}_{k}} X_{w}.
\end{equation}

\smallskip

\noindent 
\textbf{Step 2:}
Now we prove the joint asymptotic normality of 
$$
\Sigma^{-\frac{1}{2}}
\big(\Tr(A_{\mathrm{cen1}}^{2k_1+1}),\ldots,\Tr( A_{\mathrm{cen1}}^{2k_l+1}) \big).
$$
In particular, we are to prove the following:
There exists random variables $Z_1,\ldots,Z_l$ such that for any fixed $m$
\begin{equation}\label{eqn:lim2}
\lim_{n \to \infty}\E[R_{n,1}\ldots R_{n,m}]\to\left\{
\begin{array}{ll}
  \sum_{\eta} \prod_{i=1}^{\frac{m}{2}} \E[Z_{\eta(i,1)}Z_{\eta(i,2)}] & ~ \text{for $m$ even}\\
  0 &~ \text{for $m$ odd.}
\end{array}
\right.
\end{equation}
Here $R_{n,i} \in \{Y_{n,2k_1+1},\ldots,Y_{n,2k_l+1} \}$ and $\eta$ is a partition of $\{1,2\ldots,m \}$ into $\frac{m}{2}$ blocks such that each block contains exactly two elements.
First observe that (\ref{eqn:lim2}) will simultaneously imply part (i) and (ii) of Lemma \ref{lem:mom}. Implication of (i) is obvious. However, for (ii) one can take $R_{n,i}$'s to be all equal and from Wick's formula (Lemma \ref{lem:wick}) the limiting distribution of $R_{n,i}$'s is normal. It is well known that normal random variables satisfy Carleman's condition. 

Note that 
\begin{equation}\label{jointmom}
\E[R_{n,1}\ldots R_{n,m}]= \left(\frac{1}{n\pnav(1-\pnav)}\right)^{\sum_{i=1}^{m} \frac{l_{i}-1}{2}}\sum_{w_1\ldots w_m} \E\left[ X_{w_1}\ldots X_{w_m}\right].
\end{equation}
Here $w_{i}$ is a closed word with $l_i=l(w_i)=2k+2$, not belonging to $\mathrm{No}_{k}$ if $R_{n,i}= Y_{n,2k+1}.$
We start with any generic $X_{w_1}\ldots X_{w_m}$. 
We at first prove 
\[
\E\left[ X_{w_1}\ldots X_{w_m}\right] = 0 
\]
if the sentence $a=[w_i]_{i=1}^{m}$ is not a weak CLT sentence. If $a$ is not a weak CLT sentence, then there is at least one edge in $G_{a}$ which is traversed exactly once by the sentence $a$. This means there is at least one random variable in the product $X_{w_1}\ldots X_{w_m}$ which has appeared exactly once. Since $X_{w_1}\ldots X_{w_m}$ is product of independent mean $0$ random variable, we have 
\[
\E\left[ X_{w_1}\ldots X_{w_m}\right] = 0.
\]

Now let $\mathcal{A}_{m,\zeta}$ be the set of weak CLT sentences obtained by concatenating $m$ words such that the $i$th word has length $l_i$ and $wt(a)=\zeta$ for any $a \in \mathcal{A}_{m,\zeta}$. 
Proposition \ref{prop:cltrep} 
leads to
\[
wt(a) \le \sum_{i=1}^{m} \frac{l_{i}-1}{2}.
\]
As a consequence, we can write (\ref{jointmom}) as 
\begin{equation}\label{clt:verify}
\begin{split}
\E[R_{n,1}\ldots R_{n,m}]= 
\left(\frac{1}{n\pnav(1-\pnav)}\right)^{\sum_{i=1}^{m} \frac{l_{i}-1}{2}} 
\sum_{\zeta=1}^{\sum_{i=1}^{m} \frac{l_{i}-1}{2}} \sum_{a=[w_{i}]_{i=1}^{m} \in \mathcal{A}_{m,\zeta}} \E\left[ X_{w_1}\ldots X_{w_m}\right].
\end{split}
\end{equation}

We now show that only CLT sentences matter (those s.t.~$\zeta = \sum_{i=1}^{m} \frac{l_{i}-1}{2}$) on the right side of the last display asymptotically. 
To this end, fix any weak CLT sentences $a\in \mathcal{A}_{m,\zeta}$. 
For any edge $e=\{i,j\}$ in the graph $G_a$, we shall denote the random variable $x_{i,j}-\pnav$ by $x_{e}$. Since $|x_{i,j}-\pnav|\le 1$, we have for any power $b \ge 2$, 
\[
\E |x_{i,j}-\pnav|^{b} \le \E |x_{i,j}-\pnav|^{2} = \pnav(1-\pnav).
\] 
As a consequence, we have 
\[
\E  |X_{w_1}\ldots X_{w_m}|  \le \left(\pnav(1-\pnav)\right)^{\# E_{a}} \le \left(\pnav(1-\pnav)\right)^{\#V_{a}} .
\]
Here $V_{a}$ and $E_{a}$ denote the vertex and the edge set of the graph $G_{a}$ respectively. The second inequality follows from the fact $\#E_{a} \ge \# V_{a}$ as $l(w_i)$ is even for all $1\le i \le m $.
As a consequence, 
\begin{equation}\label{bddweakclt}
\begin{split}
&\left(\frac{1}{n\pnav(1-\pnav)}\right)^{\sum_{i=1}^{m} \frac{l_{i}-1}{2}} \sum_{1 \le \zeta <\sum_{i=1}^{m} \frac{l_{i}-1}{2}} \sum_{a \in \mathcal{A}_{m,\zeta}} \E  \left|X_{w_1}\ldots X_{w_m}\right| \\
& \le \left(\frac{1}{n\pnav(1-\pnav)}\right)^{\sum_{i=1}^{m} \frac{l_{i}-1}{2}} \sum_{1 \le \zeta <\sum_{i=1}^{m} \frac{l_{i}-1}{2}} \sum_{a \in \mathcal{A}_{m,\zeta}} \left(\pnav(1-\pnav)\right)^{\zeta}\\
&= \left(\frac{1}{n\pnav(1-\pnav)}\right)^{\sum_{i=1}^{m} \frac{l_{i}-1}{2}} \sum_{1 \le \zeta <\sum_{i=1}^{m} \frac{l_{i}-1}{2}}   \left(\pnav(1-\pnav)\right)^{\zeta} \# \mathcal{A}_{m,\zeta}.
\end{split}
\end{equation}
Now we use Lemma \ref{lem:appendix} to get that 
\[
\# \mathcal{A}_{m,\zeta} \le n^{\zeta} 2^{\sum_{i} l_i}\left(C_1\sum_{i}l_i\right)^{C_2m}
\left(\sum_{i}l_i\right)^{3(\sum_{i}l_i-2\zeta)}.
\] 
As a consequence, the first expression in (\ref{bddweakclt}) is bounded by
\begin{equation}\label{bddwakcltII}
\begin{split}
&\left(\frac{1}{n\pnav(1-\pnav)}\right)^{\sum_{i=1}^{m} \frac{l_{i}-1}{2}}
\sum_{1 \le \zeta <\sum_{i=1}^{m} \frac{l_{i}-1}{2}}   \left(n\pnav(1-\pnav)\right)^{\zeta} 
2^{\sum_{i} l_i}\left(C_1\sum_{i}l_i\right)^{C_2m}
\left(\sum_{i}l_i\right)^{3(\sum_{i}l_i-2\zeta)}\\
& \le \left(\frac{1}{n\pnav(1-\pnav)}\right)^{\sum_{i=1}^{m} \frac{l_{i}-1}{2}}
\sum_{1 \le \zeta<\sum_{i=1}^{m} \frac{l_{i}-1}{2}}   \left(n\pnav(1-\pnav)\right)^{\zeta} 2^{2mk^*}(C_12mk^*)^{C_2m}\left( 2mk^* \right)^{3(\sum_{i}l_i-2\zeta)}\\
&= \left(\frac{1}{n\pnav(1-\pnav)}\right)^{\sum_{i=1}^{m} \frac{l_{i}-1}{2}}\\
& ~~~~~~~~\times \sum_{1 \le \zeta<\sum_{i=1}^{m} \frac{l_{i}-1}{2}} \left(n\pnav(1-\pnav)\right)^{\zeta}2^{2mk^*}(C_12mk^*)^{C_2m}(2mk^*)^{3m}\left( 2mk^* \right)^{6(\sum_{i}\frac{l_i-1}{2}-\zeta)}\\
&=\sum_{1 \le \zeta<\sum_{i=1}^{m} \frac{l_{i}-1}{2}}2^{2mk^*}(C_3mk^*)^{C_4m} \left( \frac{2mk^*}{n\pnav(1-\pnav)} \right)^{3({\sum_{i}(l_i-1)}-2\zeta)}\,.
\end{split}
\end{equation}
Here $k^*=\max_{1\le i \le l}(k_i+1)$ and $C_1$, $C_2$, $C_3$ and $C_4$ are positive numeric constants. 
Now $\zeta$ and $\sum_{i}(l_i-1)$ are both integers so ${\sum_{i}(l_i-1)}-2\zeta \ge 1$. As a consequence, again by property of geometric series and the fact $k^*=o(\log(n\pnav))$, we have the last expression in (\ref{bddwakcltII}) is bounded by
\begin{equation}\label{bddwakcltIII}
C_52^{2mk^*}(C_3mk^*)^{C_4m}\left(\frac{2mk^*}{n\pnav(1-\pnav)}\right)^{3}
\end{equation}
where $C_5$ is another numeric constant. Now we consider two cases when $\pnav$ converges to $0$ and when $\pnav$ converges to some $p<1$. In both the cases $(1-\pnav)$ is asymptotically lower bounded by $\frac{1}{2}$ and $\frac{1-p}{2}$ respectively. So we shall not be concerned about the factor $\frac{1}{(1-\pnav)^3}$ in (\ref{bddwakcltIII}). Now ignoring $\frac{1}{(1-\pnav)^3}$ in (\ref{bddwakcltIII}) and taking logarithm of the rest we have 
\begin{equation}\label{bddwakcltIV}
\log(C_5)+ 2mk^*\log(2)+ (C_4m)\left(\log(k^*)+\log(m)+\log(C_3)\right)+ 3\left( \log(2mk^*)- \log(n\pnav)\right).
\end{equation}
For large value of $k^*$ the dominant term with the positive sign in (\ref{bddwakcltIV}) is $2mk^*\log(2)$. However from our assumption 
\[
2mk^*\log(2)-3\log(n\pnav) \to -\infty
\] 
for any fixed $m$. As a consequence, the first expression in (\ref{bddweakclt}) goes to $0$.

So we can only focus on the words such that $wt(a)=\frac{\sum_{i=1}^{m} l_i-1}{2}$. In this case the words $w_i$ of the sentence $a$ are perfectly matched
in the sense that for any $i$ there exists a unique $j$ distinct from $i$ such that $w_i$ and $w_j$
have at least one letter in common. 
In particular, $m$ is even. 
Now given any such sentence $a$, we introduce a partition $\eta(a)$ of $\{1,\ldots,m \}$ in the following way. 
If $i$ and $j$ are in same block of the partition $\eta(a)$, then $G_{w_i}$ and $G_{w_j}$ have at least one edge in common. Observe that any such $\eta(a)$ is a partition of $\{ 1,\ldots,m \}$ such that each block contains exactly two elements. 
As a consequence, we can write the L.S. of (\ref{clt:verify}) as 
\begin{equation}\label{wickapprox}
\begin{split}
& \E[R_{n,1}\ldots R_{n,m}]\\ &=o(1)+\left(\frac{1}{n\pnav(1-\pnav)}\right)^{\sum_{i=1}^{m} \frac{l_{i}-1}{2}}\sum_{\eta} \sum_{a: \eta(a)=\eta} \E\left[ X_{w_1}\ldots X_{w_m}\right]\\
                          &=o(1)+\left(\frac{1}{n\pnav(1-\pnav)}\right)^{\sum_{i=1}^{m} \frac{l_{i}-1}{2}}\sum_{\eta} \sum_{a: \eta(a)=\eta} \prod_{i=1}^{\frac{m}{2}} \E\left[X_{w_{\eta(i,1)}}X_{w_{\eta(i,2)}}\right]\\
                          &=o(1)+\left(\frac{1}{n\pnav(1-\pnav)}\right)^{\sum_{i=1}^{m} \frac{l_{i}-1}{2}}\sum_{\eta} \sum_{a: \eta(a)=\eta} \prod_{i=1}^{\frac{m}{2}} \left(\pnav(1-\pnav)\right)^{wt([w_{\eta(i,1)},w_{\eta(i,2)}])}\\
                          &= o(1)+\left(\frac{1}{n}\right)^{\sum_{i=1}^{m} \frac{l_{i}-1}{2}}\sum_{\eta}  \prod_{i=1}^{\frac{m}{2}}(\#[w_{\eta(i,1)},w_{\eta(i,2)}]). 
\end{split}
\end{equation}
Here $[w_{\eta(i,1)},w_{\eta(i,2)}]$ denotes a typical CLT word pair where $w_{\eta(i,1)}$ and $w_{\eta(i,2)}$ are closed words of length $l_{\eta(i,1)}$ and $l_{\eta(i,2)}$ respectively and $\#[w_{\eta(i,1)},w_{\eta(i,2)}]$ denotes the cardinality of such CLT word pairs. 
Note that second step in (\ref{wickapprox}) follows from Proposition \ref{prop:cltrep}.
The third step follows from Lemma \ref{lem:expcltpair}, and
since each $w_i$ has an odd number of total edge visits, it has to be $\#V = \#E$ by Lemma \ref{lem:expcltpair}.
Now recalling Proposition \ref{prop:finalcltwordpair} and applying Proposition \ref{prop:uniword} and Lemma \ref{lem:cltwordpaircnt} we get if the length of the common bracelet between $w_{\eta(i,1)}$ and $w_{\eta(i,2)}$ is $r$ then there are 
\[
n^{\zeta_1}(1+o(1)) \frac{ (l_{\eta(i,1)}-1)f(l_{\eta(i,1)}-1,r)}{r} 
\]
many choices of $w_{\eta(i,1)}$ and for any such $w_{\eta(i,1)}$, there are
\[
n^{\zeta_2-r}(1+o(1)) 2 (l_{\eta(i,2)}-1)f(l_{\eta(i,2)}-1,r)
\]
many choices of $w_{\eta(i,2)}$. 
Here $\zeta_1$ and $\zeta_2$ are $wt(w_{\eta(i,1)})$ and $wt(w_{\eta(i,2)})$ respectively. Finally 
$wt(w_{\eta(i,1)},w_{\eta(i,2)})= \zeta_1 + \zeta_2 - r$. 
So
\begin{equation}\label{eqn:varexpproof}
\begin{split}
\#[w_{\eta(i,1)},w_{\eta(i,2)}]=(1+o(1)) n^{wt(w_{\eta(i,1)},w_{\eta(i,2)})} V_{\eta(i,1), \eta(i,2)},
\end{split}
\end{equation}
where
\begin{equation}\label{varexp}
V_{i,j}:= \sum_{r=3~:~ r ~\text{odd}}^{\min(l_i-1,l_j-1)}f(l_i-1,r)f(l_j-1,r)\frac{2(l_{i}-1)(l_j-1)}{r}.
\end{equation}
Plugging in these values in (\ref{wickapprox}), we get
\[
\E[R_{n,1}\ldots R_{n,m}]= o(1)+ \sum_{\eta}\prod_{i=1}^{\frac{m}{2}} V_{\eta(i,1),\eta(i,2)}.
\]
Finally taking $m=2$ we get $V_{i,j}$ to be the asymptotic covariance between $Y_{n,2k_i+1}$ and $Y_{n,2k_j+1}$. This completes the proof of part (i).

\subsubsection{Proof of part (ii)}
We only prove the case when $\pnav \to p \in (0,1)$ here. The case when $\pnav \to 0$ is similar. Firstly,
\begin{equation}
\beta_{2k}=  \left(\frac{1}{n\pnav(1-\pnav)}\right)^{\frac{2k+1}{2}}\sum_{w~:~ l(w)= 2k+1 ~ \& ~ w ~\text{closed}} \big(X_{w}-\E\left[X_{w}\right]\big).
\end{equation}
For any fixed $m$, we shall again verify (\ref{clt:verify}), but with $R_{n,i} \in \{\beta_{2k_1},\ldots, \beta_{2k_{l}} \}$. 
We again have, 
\[
\E\left[R_{n,1},\ldots,R_{n,m}  \right]=\left(\frac{1}{n\pnav(1-\pnav)}\right)^{\sum_{i=1}^{m} \frac{l_{i}-1}{2}}\sum_{w_1\ldots w_m} \E\left[ \left(X_{w_1}-\E\left[  X_{w_1}\right]\right)\ldots \left(X_{w_m}-\E\left[  X_{w_m}\right]\right)\right].
\]
\noindent  
Repeating the arguments in the previous proof, it is easy to see that 
\[
\E\left[ \left(X_{w_1}-\E\left[  X_{w_1}\right]\right)\ldots \left(X_{w_m}-\E\left[  X_{w_m}\right]\right)\right]=0
\]
unless the sentence $a=[w_{i}]_{i=1}^{m}$ is a weak CLT sentence. 
So we keep our focus only on the weak CLT sentences. 
Let $\mathcal{A}_{m,\zeta}$ be the set of weak CLT sentences obtained from concatenating $m$ words such that the $i$th word has length $l_{i}$ and $wt(a)=\zeta$ for any $a \in \mathcal{A}_{m,\zeta}$. 
From Proposition \ref{prop:cltrep}, $\zeta \le \sum_{i=1}^{m} \frac{l_{i}-1}{2}$ where the equality holds only if $a$ is a CLT sentence. 
Analysis similar to (\ref{bddweakclt}) and (\ref{bddwakcltII}) shows that 
\begin{equation}\label{bddweakclteven}
\begin{split}
&\left(\frac{1}{n\pnav(1-\pnav)}\right)^{\sum_{i=1}^{m} \frac{l_{i}-1}{2}} \sum_{1 \le \zeta <\sum_{i=1}^{m} \frac{l_{i}-1}{2}} \sum_{a \in \mathcal{A}_{m,\zeta}} \E  \left|\left(X_{w_1}-\E\left[ X_{w_1}\right]\right)\ldots \left(X_{w_1}-\E\left[ X_{w_1}\right]\right) \right|\\
& \le \left(\frac{1}{\pnav(1-\pnav)}\right)^{\frac{m}{2}}\sum_{1 \le \zeta <\sum_{i=1}^{m} \frac{l_{i}-1}{2}}2^{2mk^*}(C_3mk^*)^{C_4m} \left( \frac{2mk^*}{n\pnav(1-\pnav)} \right)^{3({\sum_{i}(l_i-1)}-2\zeta)}. 
\end{split}
\end{equation}
Here $k^*=\max_{1\le i \le l}(k_i)$ and $C_1$, $C_2$, $C_3$ and $C_4$ are positive numeric constants. 
The additional factor $\left(\frac{1}{\pnav(1-\pnav)}\right)^{\frac{m}{2}}$ is due to the following fact. 
Here some of the connected components in the graph $G_{a}$ for some $a \in \mathcal{A}_{m,\zeta}$ can be trees where the number of edges is one less than the number of vertices. 
On the other hand, by Lemma \ref{lem:cltrep} the number of connected components in $G_{a} \le \frac{m}{2}$. 
This gives rise to the additional factor.
In this context we mention that this additional factor also comes when $\pnav \to 0$, 
which is compensated by the scaling $\sqrt{\pnav}$ in the CLT of $\beta_{2k}$. 
 
By the foregoing discussion, we still only need to consider the CLT sentences as in the odd power case. 
Let $a=[w_{i}]_{i=1}^{m}$ be a typical CLT sentence. Applying Proposition \ref{prop:cltrep} we again have for any word $w_{i}$ there is exactly one other word $w_{j}$ such that $G_{w_{i}}$ and $G_{w_{j}}$ share an edge. As a consequence, the partition $\eta= \eta(a)$ is again a partition of $\{ 1,\ldots, m \}$ such that each block has exactly two elements. As a consequence,
\begin{equation}
\begin{split}
&\E[R_{n,1}\ldots R_{n,m}] \\ & =o(1)+\left(\frac{1}{n\pnav(1-\pnav)}\right)^{\sum_{i=1}^{m} \frac{l_{i}-1}{2}}\sum_{\eta} \sum_{a: \eta(a)=\eta} \E\left[ \left(X_{w_1}- \E\left[ X_{w_1} \right]\right)\ldots\left(X_{w_m}- \E\left[ X_{w_m} \right]\right)\right]\\
&=o(1)+\left(\frac{1}{n\pnav(1-\pnav)}\right)^{\sum_{i=1}^{m} \frac{l_{i}-1}{2}}\sum_{\eta} \sum_{a: \eta(a)=\eta} \prod_{i=1}^{\frac{m}{2}} 
\E\Big[ \big(X_{w_{\eta(i,1)}}- \E\big[ X_{w_{\eta(i,1)}} \big]\big)
\big(X_{w_{\eta(i,2)}}- \E\big[ X_{w_{\eta(i,2)}} \big]\big)\Big],
\end{split}
\end{equation}
where $[w_{\eta(i,1)},w_{\eta(i,2)}]$ is a typical CLT word pair. 
From Lemma \ref{lem:expcltpair}, there are two possible cases. 
Firstly, the graph corresponding to the CLT word pair is unicyclic. 
The analysis is of this case is the same as in the odd power case and has been presented in the proof of part (i). 
We only do the analysis of the second case when the graph is a tree. 
Observe that in this case, from Lemma \ref{lem:expcltpair}, both $w_{\eta(i,1)}$ and $w_{\eta(i,2)}$ are Wigner words,
and there is a common edge between the $G_{w_{\eta(i,1)}}$ and $G_{w_{\eta(i,2)}}$. 
Up to a multiplicative factor of $1+o(1)$,
there are $n^{\frac{l_{\eta(i,1)}-1}{2}+1}C_{(l_{\eta(i,1)-1})}$ many Wigner words of length $l_{\eta(i,1)}$. 
Once any such word is fixed there are $\frac{l_{\eta(i,1)}-1}{2}$ many choices for the edge in $w_{\eta(i,1)}$ which is shared by $w_{\eta(i,2)}$. 
Once a word $w_{\eta(i,1)}$ and a choice of this edge is fixed there are exactly two ways this edge can be traversed by $w_{\eta(i,2)}$ depending on which letter appears first. 
Finally, we again have $n^{\frac{l_{\eta(i,2)}-1}{2}+1-2}C_{(l_{\eta(i,2)-1})}$  many choices of $w_{\eta(i,2)}$ after fixing $w_{\eta(i,1)}$, the edge which is shared by $w_{\eta(i,1)}$ and the order of traversal of this edge by $w_{\eta(i,2)}$. So the total number of choices for the CLT word pair of this kind is, up to a multiplicative factor of $1+o(1)$,
\[
2\,n^{\frac{l_{\eta(i,1)}-1}{2}+\frac{l_{\eta(i,2)}-1}{2}}\,\frac{(l_{\eta(i,1)}-1)(l_{\eta(i,2)}-1)}{4} C_{(l_{\eta(i,1)-1})} C_{(l_{\eta(i,2)-1})}. 
\]
Now observe that for any such word pair,
\begin{equation}
\begin{split}
&\E\Big[ \big(X_{w_{\eta(i,1)}}- \E\big[ X_{w_{\eta(i,1)}} \big]\big)
\big(X_{w_{\eta(i,2)}}- \E\big[ X_{w_{\eta(i,2)}} \big]\big)\Big]\\
& = \left(\pnav(1-\pnav)\right)^{\frac{l_{\eta(i,1)}-1}{2}+\frac{l_{\eta(i,2)}-1}{2}-2}\E\left[ (x_{1,2}-\pnav)^{4}\right]\\
& ~~~~~ - \left(\pnav(1-\pnav)\right)^{\frac{l_{\eta(i,1)}-1}{2}+\frac{l_{\eta(i,2)}-1}{2}}\\
&= \left(\pnav(1-\pnav)\right)^{\frac{l_{\eta(i,1)}-1}{2}+\frac{l_{\eta(i,2)}-1}{2}-2}  \Var\left[\left(x_{1,2}-\pnav \right)^{2} \right].
\end{split}
\end{equation}
The rest of the argument is the same as the proof of the odd power case and so we omit the details. 

\subsubsection{Proof of part (iii)} 
Now we give a proof of part (iii) of Theorem \ref{thm:null}. We are supposed to show 
for any $k=o(\log(n \pnav))$,
\begin{equation}\label{eqn:momcycle}
{\Tr(A_{\mathrm{cen1}}^{2k+1})- \sum_{r=3: \text{$r$ odd}}^{2k+1} f(2k+1,r)\frac{2k+1}{r} C_{n,r}(G)} \stackrel{p}{\to} 0.
\end{equation}
We prove this by showing the variance of the L.S. of (\ref{eqn:momcycle}) goes to $0$. Recalling the expression of $C_{n,r}(G)$ from (\ref{def:signed}) that 
\begin{equation}
\begin{split}
C_{n,r}(G) &= \left(\frac{1}{\sqrt{np_{n,\mathrm{av}}(1-p_{n,\mathrm{av}})}}\right)^{r} \sum_{i_0,i_1,\ldots,i_{r-1}} (x_{i_0,i_1}-p_{n,\mathrm{av}})\ldots(x_{i_{r-1}i_0}-p_{n,\mathrm{av}})\\
           &= \left(\frac{1}{\sqrt{np_{n,\mathrm{av}}(1-p_{n,\mathrm{av}})}}\right)^{r} \sum_{w \in \mathrm{Sc}_{r}} X_{w}
\end{split}
\end{equation}
where $i_0,\ldots,i_{r-1}$ are all distinct.
Here $\mathrm{Sc}_{r}$ is the class of closed words such that for any $w \in \mathrm{Sc}_{r}$, $G_{w}$ is a cycle of length $r$. First observe that if $r_{1}\neq r_{2}$, then $\E[X_{w_1}X_{w_2}]=0$ for any $w_{1} \in \mathrm{Sc}_{r_1}$ and $w_{2} \in \mathrm{Sc}_{r_2}$ trivially. As a consequence, $\Cov(C_{n,r_1}(G), C_{n,r_2}(G))=0$ whenever $r_1\neq r_2$. Now we evaluate 
\[
\Cov(\Tr(A_{\mathrm{cen1}}^{2k+1}), C_{n,r}(G))
\]
for any odd number $r\le 2k+1$. 
One can imitate the proof of part (i) to get that 
\begin{equation}\label{nocov}
\begin{split}
E_r:=&\Cov\bigg(\left(\frac{1}{n\pnav(1-\pnav)}\right)^{\frac{2k+1}{2}}\sum_{w \in \mathrm{No}_{k}} X_{w},C_{n,r}(G)\bigg)\\
& \le 
2^{2k+r+3}(C_{1}(2k+r+3))^{2C_2}(2k+r+3)^{6}C_{3}\left(  \frac{2k+r+3}{n\pnav}\right)^{6}.
\end{split}
\end{equation}
Here $C_1$, $C_2$ and $C_3$ are known constants. 
Since $r \le 2k+1$, summing the second expression in (\ref{nocov}) we have 
\begin{equation}
\begin{split}
&\sum_{r=3: r ~ \text{odd}}^{2k+1}E_{r}\le k 2^{4k+4}(C_{1}(4k+4))^{2C_2}(4k+4)^{6}C_{3}\left(  \frac{4k+4}{n\pnav}\right)^{6} \to 0.
\end{split}
\end{equation}
As a consequence, we only need to analyze the covariance between $Y_{n,2k+1}$ and $C_{n,r}(G)$ where $Y_{n,2k+1}$ was defined in (\ref{def:y}). Now 
\begin{equation}\label{covycycle}
\begin{split}
\Cov(Y_{n,2k+1}, C_{n,r}(G))= \left( \frac{1}{n\pnav(1-\pnav)} \right)^{\frac{2k+1+r}{2}}\sum_{w_{1} \notin \mathrm{No}_{k}} \sum_{w_{2} \in \mathrm{Sc}_{r}} \E[X_{w_1} X_{w_2}].
\end{split}
\end{equation}
It is easy to see that $\E[X_{w_1} X_{w_2}]=0$ unless $[w_1,w_2]$ is a weak CLT sentence. 
Observe that if $a=[w_1,w_2]$ is a CLT word pair then then $wt(a)= \frac{2k+1+r}{2}$, this is an integer as $r$ is taken to be odd. 
Again we consider $\mathcal{A}_{2,\zeta}$ the set of weak CLT sentences obtained by concatenating two words $w_{1} \notin \mathrm{No}_{k}$ and $w_{2} \in \mathrm{Sc}_{r}$ such that for any $a \in \mathcal{A}_{2,\zeta}$, $wt(a)= \zeta$. 
Observe that 
\begin{equation}
\begin{split}
\Cov(Y_{n,2k+1}, C_{n,r}(G)) &= \left( \frac{1}{n\pnav(1-\pnav)} \right)^{\frac{2k+1+r}{2}} 
\sum_{\zeta=1}^{\frac{2k+1+r}{2}} \sum_{a \in \mathcal{A}_{2,\zeta}}
\E[X_{w_1} X_{w_2}].
\end{split}
\end{equation}
By applying Lemma \ref{lem:appendix} again we have 
\begin{equation}\label{weakcltcovI}
\begin{split}
 \mathcal{T}_{r}:=&\left( \frac{1}{n\pnav(1-\pnav)} \right)^{\frac{2k+1+r}{2}}\sum_{1 \le \zeta < \frac{2k+1+r}{2}}\sum_{a \in \mathcal{A}_{2,\zeta}} \E[|X_{w_1} X_{w_2}|]\\
 & \le C_52^{4(k+1)}(2C_3(k+1))^{2C_4}\left(\frac{4(k+1)}{n\pnav(1-\pnav)}\right)^{3}.
\end{split}
\end{equation}
Note that 
\[
\sum_{r=3: r ~ \text{odd}}^{2k+1} \mathcal{T}_{r} \le 2^{C_{T}k}P_{T}(k)\left(\frac{1}{n\pnav}\right)^3 \to 0,
\]
where $C_{T}$ is a known constant and $P_{T}$ is a known polynomial in $k$.
The convergence occurs whenever $k=o(\log(n\pnav))$ as $n\pnav\to\infty$.

Now observe that  any $w_2 \in \mathrm{Sc}_{r}$ is a uniword with bracelet length $r$ and there are 
\begin{equation}
n(n-1)\ldots (n-r+1) \ge (n-r+1)^{r}
\end{equation} 
such words. Now by applying Lemma \ref{lem:cltwordpaircnt} we have for each $w_2$, there are at least 
\begin{equation}
\begin{split}
&(n-r)(n-r-1)\ldots \left(n-\frac{2k+1+r}{2}+1 \right) 2(2k+1)f(2k+1,r)\\
& \ge  2(2k+1)f(2k+1,r) \left(n-\frac{2k+1+r}{2}+1 \right)^{\frac{2k+1-r}{2}}
\end{split}
\end{equation} 
many choices of $w_1$ such that $[w_1,w_2]$ is a CLT word pair. As a consequence, 
\begin{equation}\label{cove'mateI}
\begin{split}
 & \left( \frac{1}{n\pnav(1-\pnav)} \right)^{\frac{2k+1+r}{2}}\sum_{a \in \mathcal{A}_{2,\frac{2k+1+r}{2}}}\E[X_{w_1} X_{w_2}]\\
 & \ge \left( \frac{1}{n\pnav(1-\pnav)} \right)^{\frac{2k+1+r}{2}} 2(2k+1)f(2k+1,r)\left(n-\frac{2k+1+r}{2}+1\right)^{\frac{2k+1-r}{2}}\\
 & \hskip 15em \times (n-r+1)^{r}(\pnav(1-\pnav))^{\frac{2k+1+r}{2}}\\
 & \ge 2(2k+1)f(2k+1,r)\left(1-\frac{2k-1+r}{2n}\right)^{\frac{2k+1+r}{2}}.
\end{split}
\end{equation}
On the other hand, arguments similar to step 2 of the proof of part (i) of Theorem \ref{thm:null} gives us
\begin{equation}\label{cove'mateII}
\Var(Y_{n,2k+1}) \le \sum_{r=3 : r ~ \text{odd}}^{2k+1} 2f(2k+1,r)^2\frac{(2k+1)^2}{r} + o(1),
\end{equation} 
and 
\begin{equation}\label{cove'mateIII}
\Var(C_{n,r}) \le 2r.
\end{equation}
Plugging in the estimates in (\ref{weakcltcovI}),(\ref{cove'mateI}), (\ref{cove'mateII}) and (\ref{cove'mateIII}) and recalling the fact $\Cov(C_{n,r_1},C_{n,r_2})=0$ for $r_1 \neq r_2$, we have 
\begin{equation}\label{finalcovex}
\begin{split}
&\Var\left(Y_{n,2k+1}- \sum_{r=3 ~:~ r ~ \text{odd}}^{2k+1} \frac{(2k+1)f(2k+1,r)}{r} C_{n,r}  \right)\\
& = \Var(Y_{n,2k+1})+ \Var\left(\sum_{r} \frac{(2k+1)f(2k+1,r)}{r} C_{n,r}\right)- 2\sum_{r} \Cov\left(Y_{n,2k+1},\frac{(2k+1)f(2k+1,r)}{r}C_{n,r}\right)\\
& \le o(1)+\sum_{r} 4f(2k+1,r)^2\frac{(2k+1)^2}{r}- \sum_{r}4f(2k+1,r)^2\frac{(2k+1)^2}{r}\left(1-\frac{2k-1+r}{2n}\right)^{\frac{2k+1+r}{2}}\\
&= \sum_{r}\left(1-\left(1-\frac{2k-1+r}{2n}\right)^{\frac{2k+1+r}{2}}\right) 4f(2k+1,r)^2\frac{(2k+1)^2}{r} +o(1)\\
&= \sum_{r} O\left(\frac{(2k-1+r)^2}{n}\right)4f(2k+1,r)^2\frac{(2k+1)^2}{r}+ o(1).
\end{split}
\end{equation}
Here the last step follows from the elementary inequality $1-(1-x)^{y}\le \frac{xy}{1-x}$ for any $0<x<1$ and $y>0$. From Proposition \ref{prop:FKcount} we know $f(2k+1,r) \le 2^{2k}$. As a consequence, the first expression in (\ref{finalcovex}) can be further bounded by 
\begin{equation}\label{last}
4(2k+1)^22^{4k}\sum_{r} O\left(\frac{(2k-1+r)^2}{n}\right)+o(1) \to 0
\end{equation}
whenever $k=o(\log(n\pnav))$ as $n\pnav\to\infty$.
Recalling (\ref{nocov}), we get 
\begin{equation}\label{varcen1}
\begin{split}
&\Var\left(\Tr\left(A_{\mathrm{cen1}}^{2k+1}\right)- \sum_{r=3 : r ~ \text{odd}}^{2k+1} \frac{(2k+1)f(2k+1,r)}{r} C_{n,r}  \right)
\to 0.
\end{split}
\end{equation}
This completes the proof of part (iii).
\hfill{$\square$}


\subsection{Proof of parts (i)--(iii) of Theorem \ref{thm:alt}}

We focus on part (i) of Theorem \ref{thm:alt}. The arguments for parts (ii) and (iii) are similar. 
All expectation and variance in this part are taken under $\mathbb{P}_{1,n}$ conditioning on the group assignment $\sigma_i$, $1\leq i\leq n$.

Before going into the proof we introduce some notations that will be useful in the proof. 
We define $\mathcal{E}_{k}= \{(0,1),\ldots,(k-1,k)\}$.
In the proof we often denote $\mathcal{E}_{2k+1}$ by $\mathcal{E}$ for notational convenience.  
We shall deal with two disjoint subsets $\mathcal{E}_{L}$ and $\mathcal{E}_{T}$ of $\mathcal{E}$ such that $\mathcal{E}_{L} \cup \mathcal{E}_{T}= \mathcal{E}.$ Let $w= (i_{0},\ldots,i_{2k+1})$ be any word. Then for any $e=(j,j+1)\in \mathcal{E}$, we define 
\[
e(w)=  (i_{j},i_{j+1}).
\] 
For any word $w$, we consider the graph $G_{w}=(V_{w}, E_{w})$ as defined in Section \ref{subsec:word}. 
Given the word $w$ and  a subset $\mathcal{E}' \subset \mathcal{E}$, we define $E(\mathcal{E}'(w)):=\{ e(w):   e \in \mathcal{E}' \}$. 
Observe that $E(\mathcal{E}'(w))$ is the set of unique (undirected) edges traversed by $e(w),\, e\in \mathcal{E}'$, in the graph $G_{w}$, and it does not take into account the number of passages of any of its elements.

Let $d= \frac{p_n-q_n}{2}$. 
In what follows, we focus on the case where $p_n > q_n$. 
If $p_n<q_n$, we simply need to replace every $t$ with $-t$.
Recall (\ref{traceexp}) to get 
\begin{equation}\label{traceexpalt}
\Tr(\Acenone^{2k+1})= \left(\frac{1}{n\pnav(1-\pnav)}\right)^{\frac{2k+1}{2}}\sum_{w:l(w)= 2k+2 ~ \& ~ w ~\text{closed}} \left[X_{w}\right].
\end{equation}
Here for any word $w$ is an ordered pair $(i_0,\ldots,i_{2k+1})$ where the numbers $i_{j}\in \{  1,2,\ldots, n\}$ for $0\le j \le 2k+1$ and $X_{w}= \prod_{j=0}^{2k}\left(x_{i_j,i_{j+1}}- \pnav\right).$ However, as the data is generated under the alternative, here $\E[x_{i,j}]=p_{i,j}$ where $p_{i,j}=p_n$ if $\sigma_{i}=\sigma_{j}$ and $p_{i,j}=q_n$ if $\sigma_{i}\neq\sigma_{j}$. As a consequence, for any $i\neq j$,
\begin{equation}
x_{i,j}-\pnav= x_{i,j}-p_{i,j}+p_{i,j}-\pnav= x_{i,j}-p_{i,j} + d\sigma_{i} \sigma_{j},
\end{equation}
and so
\begin{equation}\label{Xwdecom}
X_{w} = \prod_{j=0}^{2k}\left(x_{i_j,i_{j+1}}- \pnav\right)
= \prod_{j=0}^{2k} \left( x_{i_j,i_{j+1}}- p_{i_j,i_{j+1}} +d \sigma_{i_j}\sigma_{i_{j+1}} \right).
\end{equation} 
At first we note that 
\begin{equation}\label{eqn:prodassign}
\prod_{j=0}^{2k} {\sigma_{i_j}\sigma_{i_{j+1}}}=1
\end{equation}
irrespective of the values of $\sigma_{i_{j}}$'s. 
This is due to the fact that $\sigma_{i_0} = \sigma_{i_{2k+1}}$ and so each $\sigma_{i_j}$ is multiplied an even number of times in \eqref{eqn:prodassign}.
Note that the foregoing argument depends only on the word being closed, regardless of whether its length is odd or even.
Now we can write (\ref{Xwdecom}) as 
\begin{equation}\label{plug:Xwdecom}
X_{w}= \prod_{j=0}^{2k} \left(  x_{i_j,i_{j+1}}- p_{i_j,i_{j+1}} \right)+ d^{2k+1} + V_{n,w}. 
\end{equation} 
Here $V_{n,w}$ comprises of all the cross terms.
Plugging (\ref{plug:Xwdecom}) in (\ref{traceexpalt}), we get 
\begin{equation}
\begin{split}
&\Tr(\Acenone^{2k+1})\\
&= \left(\frac{1}{n\pnav(1-\pnav)}\right)^{\frac{2k+1}{2}}\sum_{w:l(w)= 2k+2 ~ \& ~ w ~\text{closed}} \left( \prod_{j=0}^{2k} \left(  x_{i_j,i_{j+1}}- p_{i_j,i_{j+1}} \right)+ V_{n,w} \right)\\ 
& ~~~~~~~~~~~~~~~~~~~~~~ + \left(\frac{1}{n\pnav(1-\pnav)}\right)^{\frac{2k+1}{2}} (nd)^{2k+1}.\\
&= \left(\frac{1}{n\pnav(1-\pnav)}\right)^{\frac{2k+1}{2}} \sum_{w}\left( \prod_{j=0}^{2k} \left(  x_{i_j,i_{j+1}}- p_{i_j,i_{j+1}} \right)+ V_{n,w} \right) + t_{n}^{{2k+1}}.
\end{split}
\end{equation}
Here $t_n= \sqrt{\frac{c}{2(1-\pnav)}} \to t$ as $n\to\infty$.

The analysis of 
\[
D_{n,k}:= \left(\frac{1}{n\pnav(1-\pnav)}\right)^{\frac{2k+1}{2}}\sum_{w: l(w)= 2k+2 ~ \& ~ w ~\text{closed}} 
\prod_{j=0}^{2k} \left(  x_{i_j,i_{j+1}}- p_{i_j,i_{j+1}} \right) .
\]
is same as the proof of Theorem \ref{thm:null} part (i). 
We only mention that the covariance structure of $\{D_{n,k_{i}}\}_{i=1}^{l}$ is the same as the covariance structure of $\{\Tr(\Acenone^{2k_i+1}) \}_{i=1}^{l}$ due the fact that whenever $k=o(\log(n\pnav))$ both 
\begin{equation}\label{lim1}
\lim_{n \to \infty}\left(\frac{(\pnav+d)(1-\pnav-d)}{\pnav(1-\pnav)}\right)^{\frac{2k+1}{2}} =1
\end{equation}
and 
\begin{equation}\label{lim2}
\lim_{n \to \infty}\left(\frac{(\pnav-d)(1-\pnav+d)}{\pnav(1-\pnav)}\right)^{\frac{2k+1}{2}} =1.
\end{equation}
It is easy to see that $\mathrm{Cov}( D_{n,k_i}, D_{n,k_j}) /{V_{i,j}}$ is sandwiched by the left sides of (\ref{lim1}) and (\ref{lim2}). 
Here $V_{i,j}$ is defined as in (\ref{varexp}).

In the rest of this subsection, we complete the proof by analyzing the mean and variance of 
\begin{equation}\label{Vnw}
\left(\frac{1}{n\pnav(1-\pnav)}\right)^{\frac{2k+1}{2}}  \sum_{w}V_{n,w}.
\end{equation}

\medskip

\textbf{Analysis of the mean of (\ref{Vnw}).}
At first fix $w$ and consider the graph $G=(V,E)$ corresponding to the word $w$. Now 
\begin{equation}\label{expVnw}
\begin{split}
V_{n,w} &= \sum_{ \emptyset \subsetneq \mathcal{E}_{T}   \subsetneq \mathcal{E}} \prod_{e \in \mathcal{E}_{T}} (\sigma_{e(w)} d) \prod_{e \in \mathcal{E}_{L}} (x_{e(w)}-p_{e(w)})\\
&= \sum_{\emptyset \subsetneq \mathcal{E}_{T}\subsetneq \mathcal{E}}  d^{\# \mathcal{E}_{T} }\prod_{e \in \mathcal{E}_{T}} \sigma_{e(w)}\prod_{e \in \mathcal{E}_{L}} (x_{e(w)}-p_{e(w)}).
\end{split}
\end{equation}
Here for any $e(w)=(i_{j},i_{j+1})$, $\sigma_{e(w)}:= \sigma_{i_{j}}\sigma_{i_{j+1}}$, $x_{e(w)}:= x_{i_{j},i_{j+1}}$ and $p_{e(w)}:= \E[x_{e(w)}]$. 
Observe that 
\begin{equation}\label{leftrandom}
\E\left[ \prod_{e \in \mathcal{E}_{L}} (x_{e(w)}-p_{e(w)}) \right]=0 
\end{equation}
unless all the random variables $x_{e(w)}-p_{e(w)}$, $e\in \mathcal{E}_L$, have been repeated at least twice and in this case 
\[
\E\left[ \prod_{e \in \mathcal{E}_{L}} | x_{e(w)}-p_{e(w)} | \right] \le  (1+o(1)) \left(\pnav(1-\pnav)\right)^{\#E( \mathcal{E}_{L}(w))}.
\]

We now fix a typical set $\emptyset \subsetneq \mathcal{E}_{L}\subsetneq \mathcal{E}$ and an equivalence class $\mathbf{w}$ such that all the random variables on the L.S. of (\ref{leftrandom}) is repeated at least twice. Fixing $\mathbf{w}$ automatically fixes the graph $G_{\mathbf{w}}=(V_{\mathbf{w}},E_{\mathbf{w}})=G=(V,E)$. 
Observe that 
\begin{equation}\label{expectationleftrandom}
\begin{split}
&\left(\frac{1}{n\pnav(1-\pnav)}\right)^{\frac{2k+1}{2}}\sum_{w~: ~ w \in \mathbf{w}}\E\left[\left| d^{\# \mathcal{E}_{T} }\prod_{e \in \mathcal{E}_{T}} \sigma_{e(w)}\prod_{e \in \mathcal{E}_{L}} \left(x_{e(w)}-p_{e(w)}\right)\right|\right]\\
& \le \left(\frac{1}{n\pnav(1-\pnav)}\right)^{\frac{2k+1}{2}} n^{\#V}d^{\# \mathcal{E}_{T} }(1+o(1)) \left(\pnav(1-\pnav)\right)^{\#E( \mathcal{E}_{L}(w))}\\
&= (1+o(1))\left(\frac{1}{n\pnav(1-\pnav)}\right)^{\frac{2k+1}{2}}n^{\#V} \left(\frac{c\pnav}{2n}\right)^{\frac{\# \mathcal{E}_{T}}{2}}\left(\pnav(1-\pnav)\right)^{\#E( \mathcal{E}_{L}(w))}\\
&\le C^{k} \left( \frac{1}{n\pnav} \right)^{\frac{2k+1}{2}}n^{\#V} \left(\frac{\pnav}{n}\right)^{\frac{\# \mathcal{E}_{T}}{2}}\pnav^{\#E( \mathcal{E}_{L}(w))}\\
&= C^{k} \left(\frac{1}{n}\right)^{\frac{2k+1}{2}- \#V + \frac{\# \mathcal{E}_{T}}{2}} \left( \frac{1}{\pnav} \right)^{\frac{2k+1}{2}-\#E( \mathcal{E}_{L}(w))-\frac{\# \mathcal{E}_{T}}{2}}.
\end{split}
\end{equation}
Here $C$ is a deterministic constant depending on $c$ and $(1-p)$ where $p = \lim_{n \to \infty} \pnav \in [0,1).$ Since every edge in $E(\mathcal{E}_L)$ has been traversed at least twice, we have 
\begin{equation}\label{eletbound}
\begin{split}
&2k+1 = \# \mathcal{E}_{L}+ \#\mathcal{E}_{T} \ge 2\# E(\mathcal{E}_{L}(w)) + \#\mathcal{E}_{T}
\,\, \Rightarrow \,\,
\frac{2k+1}{2}-\#E( \mathcal{E}_{L}(w))-\frac{\# \mathcal{E}_{T}}{2} \ge 0.
\end{split}
\end{equation}
Now 
\begin{equation}\label{pownvspowp}
\begin{split}
&\left(\frac{2k+1}{2}- \#V + \frac{\# \mathcal{E}_{T}}{2}\right) - \left( \frac{2k+1}{2}-\#E( \mathcal{E}_{L}(w))-\frac{\# \mathcal{E}_{T}}{2} \right)\\
&
\hskip 10em = \#E( \mathcal{E}_{L}(w)) + \# \mathcal{E}_{T} - \#V \ge \#E -\#V \ge 0.
\end{split}
\end{equation}
Here the last inequality in \eqref{pownvspowp} holds due to Lemma \ref{lem:parity}.
In what follows, we divide the arguments into three different cases, depending on whether the equalities in (\ref{pownvspowp}) and/or (\ref{eletbound}) hold.

\medskip

\underline{Case 1: the equalities in both (\ref{pownvspowp}) and (\ref{eletbound}) hold.}
This case occurs if and only if the following conditions are satisfied:
\begin{enumerate}
\item The graph $G$ is unicyclic (from (\ref{pownvspowp})).
\item Every edge in $E(\mathcal{E}_{T}(w))$ has been traversed exactly once (from (\ref{pownvspowp})).
\item $E(\mathcal{E}_{T}(w)) \cap E( \mathcal{E}_{L}(w)) = \emptyset$ (from (\ref{pownvspowp})).
\item Every edge in $E(\mathcal{E}_{L}(w))$ has been traversed exactly twice (from (\ref{eletbound})).
\end{enumerate} 
Observe that from Proposition \ref{prop:unicyclicgraph} these properties are satisfied if and only if
\[
w \in \mathfrak{W}_{2k+2,r,\zeta} 
\]
for some odd number $r$ and $2\zeta-r=2k+1$. 
Here $\mathfrak{W}_{2k+2,r, \zeta}$ is defined as in Proposition \ref{prop:uniword}. From condition $2$ above, we get $\mathcal{E}_{T}$ corresponds to the walk along the bracelet of the unicyclic graph $G$. 
Hence the collection $\mathcal{E}_{T}$ is actually a closed word. As a consequence, arguing as (\ref{eqn:prodassign}), we get 
$\prod_{e \in \mathcal{E}_{T}} \sigma_{e(w)}=1$.

Using Proposition \ref{prop:uniword} for any $r$, there are $(1+o(1))f(2k+1,r)\frac{2k+1}{r}$ many equivalence classes of such words. Further, for each of these equivalence classes $\mathbf{w}$
\begin{equation}
\begin{split}
&\left(\frac{1}{n\pnav(1-\pnav)}\right)^{\frac{2k+1}{2}}\sum_{w~: ~ w \in \mathbf{w}}\E\left[ d^{\# \mathcal{E}_{T} }\prod_{e \in \mathcal{E}_{T}} \sigma_{e(w)}\prod_{e \in \mathcal{E}_{L}} \left(x_{e(w)}-p_{e(w)}\right)\right]\\
&= (1+o(1))\left(\frac{1}{n\pnav(1-\pnav)}\right)^{\frac{2k+1}{2}}n^{\#V} \left(\frac{c\pnav}{2n}\right)^{\frac{\# \mathcal{E}_{T}}{2}}\left(\pnav(1-\pnav)\right)^{\#E( \mathcal{E}_{L}(w))}\\
&= (1+o(1)) t^r.
\end{split}
\end{equation}  
Here the second step follows from the fact that  $\prod_{e \in \mathcal{E}_{T}} \sigma_{e(w)}=1$ and every edge in $E(\mathcal{E}_{L})$ has been traversed exactly twice. The third step follows from the equality in (\ref{eletbound}) and (\ref{pownvspowp}). So summing over the equivalence classes and the value of $r$, we get the total contribution of these words in the mean of (\ref{Vnw}) is, up to a $1+o(1)$ multiplier,
\begin{equation}
\begin{split}
\sum_{r=3 : r ~ \text{odd}}^{2k-1}f(2k+1,r)\frac{2k+1}{r} t^{r}.
\end{split}
\end{equation}
We sum up to $2k-1$ due to the fact that $\mathcal{E}_{T}\neq \mathcal{E}$.
 
\medskip

\underline{Case 2:
the equality (\ref{pownvspowp}) is satisfied but that in  (\ref{eletbound}) is violated.}
In this case, the graph is unicyclic. 
Let $Z$ and $F$ be the bracelet and the forest corresponding to $G$ respectively. 
Using the parity principle (Lemma \ref{lem:parity}) we get that every edge in the forest $F$ has been traversed an even number of times. 
So the edges traversed exactly once are a subset of the edges in the bracelet $Z$. 
Let $r$ be the circuit length. Then $\#E^{1} \le r$. Let $a=[w'_{i}]_{i=1}^{m}$ be the FK parsing of the word $w$. Then from Lemma \ref{lem:wordbyFK} and Lemma \ref{lem:mbdd} we have the number of equivalence classes corresponding to a given $m$ is bounded by 
\begin{equation}\label{gammabound}
\begin{split}
\# \Gamma(\zeta,2k+2,m) \le 2^{2k+1-m}\binom{2k+1}{m-1} \zeta^{2(m-1)} \le 2^{2k+1} (2k+1)^{3(m-1)}.
\end{split}
\end{equation}
 and 
 \begin{equation}\label{mboundddd}
 m = \#E^{1}_{a} - 2wt(a) +2 + (2k+2) \le \#E^{1} -2 \zeta +2 + (2k+2).
 \end{equation}
 Here $\zeta = wt(a)= wt(w).$
  
As the equality in (\ref{pownvspowp}) is satisfied, we have $\#E(\mathcal{E}_{T}(w))=\#E^{1}= \# \mathcal{E}_{T}$. From the definition of $V_{n,w}$ we also have $\#E^{1}>0$. 
However, the word $w$ is closed which makes $\#E^{1}=r$. Observe that 
\begin{equation*}
\zeta = \#E(\mathcal{E}_{L}(w))+r \quad \mbox{and} \quad 
\#\mathcal{E}_{L}+r = 2k+1.
\end{equation*}
Let $$m' :=  \#\mathcal{E}_{L}-2\#E(\mathcal{E}_{L}(w))$$ where $m' \ge 1$ as the inequality in (\ref{eletbound}) is strict. 
 Plugging in these values in (\ref{mboundddd}) we have 
 \[
 m \le r-2(\#E(\mathcal{E}_{L}(w))+r)+2 + (2k+2) = -r - \#\mathcal{E}_{L} + m' +2 + (2k+2)= m'+3.
 \]
On the other hand, 
 \[
  \frac{2k+1}{2}-\# E(\mathcal{E}_{L}(w))-\frac{\# \mathcal{E}_{T}}{2} = \frac{2k+1}{2}- \frac{\#\mathcal{E}_L-m'}{2} - \frac{r}{2}= \frac{m'}{2}.
 \]
 Plugging in these estimates in (\ref{expectationleftrandom}) and summing over all equivalence classes $\mathbf{w}$ of current concern and summing over all such choices of $\mathcal{E}_{T}$ ($\le (2^{2k+1}-1)$), we have the contribution of these words in the expectation of (\ref{Vnw}) is bounded by
 \[
 (2^{2k+1}-1)2^{2k+1} \sum_{m'=1}^{2k+1} (2k+1)^{6}\left(\frac{(2k+1)^{6}}{n\pnav}\right)^{\frac{m'}{2}} \to 0.
 \]

\medskip

\underline{Case 3:
the equality in (\ref{pownvspowp}) is not satisfied.}
In this case the graph is not unicyclic. As a consequence, $\#E -\# V \ge 1$. So for any equivalence class $\mathbf{w}$ of this type, we have from the rightmost side of \eqref{expectationleftrandom} that
 \begin{equation}\label{thirdexpsum}
 \begin{split}
 &\left(\frac{1}{n\pnav(1-\pnav)}\right)^{\frac{2k+1}{2}}\sum_{w: w \in \mathbf{w}}\E \left| d^{\# \mathcal{E}_{T} }\prod_{e \in \mathcal{E}_{T}} \sigma_{e}\prod_{e \in \mathcal{E}_{L}} x_{e(w)}-p_{e(w)}\right|
\le  \frac{C^{k}}{n}.
 \end{split}
 \end{equation}
Consider any $w \in \mathbf{w}$ and let $a=[w'_{i}]_{i=1}^{m}$ with $wt(a)=\zeta$ be the FK parsing of the word $w$. Then using Lemma \ref{lem:wordbyFK} and Lemma \ref{lem:mbdd} again we have the number of equivalence classes corresponding to a given $m$ is bounded by 
\begin{equation}\label{gammaboundII}
\begin{split}
\# \Gamma(\zeta,2k+2,m) 
\le 2^{2k+1-m}\binom{2k+1}{m-1} \zeta^{2(m-1)} \le 2^{2k+1} (2k+1)^{3(m-1)}.
\end{split}
\end{equation}
 and 
 \begin{equation}\label{mbounddddII}
 m = \#E^{1}_{a} - 2wt(a) +2 + (2k+2) \le \#E^{1} -2 \zeta +2 + (2k+2)\le 4k+3.
 \end{equation}
Here the last step follows from $\#E^{1} \le 2k+1$ and $\zeta \ge 1$.  So (\ref{gammaboundII}) can further be upper bounded by 
\[
2^{2k+1} (2k+1)^{3(4k+3)}.
\] 
Now taking the sum over $\mathbf{w}$ in the first expression of (\ref{thirdexpsum}) and summing over all such choices of $\mathcal{E}_{T}$ ($\le (2^{2k+1}-1)$), we get the contribution of theses words in (\ref{Vnw}) is bounded by
\begin{equation}\label{firstana}
C^{k} (2^{2k+1}-1)2^{2k+1} (2k+1)^{3(4k+3)}\frac{1}{n},
\end{equation}
which converges to zero as $n\to\infty$ for all
$k= o(\sqrt{\log n})$.

\medskip

Combining all these results, we have 
\[
\left(\frac{1}{n\pnav(1-\pnav)}\right)^{\frac{2k+1}{2}}  \E\left[\sum_{w}V_{n,w}\right]- \sum_{r=3 : r ~ \text{odd}}^{2k-1}f(2k+1,r)\frac{2k+1}{r} t^{r} \to 0. 
\]

\medskip

\textbf{Analysis of the variance of (\ref{Vnw}).}
Now we prove the variance of the random variable defined in (\ref{Vnw}) goes to $0$. For any given word $w$ and $\mathcal{E}_{T} \subset \mathcal{E}$ let us define 
\begin{equation}\label{def:Vet}
V_{n,w,\mathcal{E}_{T}}= \prod_{e \in \mathcal{E}_{T}} (\sigma_{e(w)} d) \prod_{e \in \mathcal{E}_{L}(w)} (x_{e(w)}-p_{e(w)}).
\end{equation}
So 
\[
V_{n,w}- \E\left[V_{n,w}\right]= \sum_{\emptyset \subsetneq \mathcal{E}_{T}\subsetneq \mathcal{E}} V_{n,w,\mathcal{E}_{T}}- \E\left[ V_{n,w,\mathcal{E}_{T}} \right].
\]
As a consequence, 
\begin{equation}\label{varvnw}
\begin{split}
&\left(\frac{1}{n\pnav(1-\pnav)}\right)^{2k+1}  \Var\left[\sum_{w}V_{n,w}\right]\\
&=\left(\frac{1}{n\pnav(1-\pnav)}\right)^{2k+1} \sum_{w}\sum_{x}\sum_{\emptyset \subsetneq \mathcal{E}_{T_1}\subsetneq \mathcal{E}}\sum_{\emptyset \subsetneq \mathcal{E}_{T_2}\subsetneq \mathcal{E}} \Cov\big(V_{n,w,\mathcal{E}_{T_1}},V_{n,x,\mathcal{E}_{T_2}} \big).
\end{split}
\end{equation}
Observe that if the graphs $G_{w}$ and $G_{x}$ corresponding to words $w$ and $x$ do not share any edge, the random variables $V_{n,w,\mathcal{E}_{T_1}}$ and $V_{n,x,\mathcal{E}_{T_2}}$ are independent no matter what $\mathcal{E}_{T_1}$ and $\mathcal{E}_{T_2}$ are. As a consequence, 
\[
\Cov\big(V_{n,w,\mathcal{E}_{T_1}},V_{n,x,\mathcal{E}_{T_2}} \big) =0
\]
for these word pairs. So we consider the case when $G_{w}$ and $G_{x}$ share at least one edge. 

As the first step, we bound the number of such word pairs
by applying the embedding algorithm stated in the proof of Lemma \ref{lem:appendix}. Here $m=2$ and the partition $\eta =\{ 1,2\}$. 
As a consequence, applying the embedding algorithm to any such pair $(w,x)$ leads to a closed word $\mathfrak{w}$ of length $4k+3$ where at least one edge in the graph $G_{\mathfrak{w}}$ has been repeated at least twice. We call the function corresponding to the embedding algorithm $f$ (i.e. $f(w,x)=\mathfrak{w}$). One can check that in this case 
\[
\#f^{-1}(\mathfrak{w}) \le (4k+3)^{3}
\]
for any closed word $\mathfrak{w}$. Now given any closed word $\mathfrak{w}$ with $wt(\mathfrak{w})= \zeta$, 
we consider its FK parsing $a=[\mathfrak{w}_{i}]_{i=1}^{m}$ where $wt(a)=\zeta$. 
We again use Lemma \ref{lem:wordbyFK} to get the number of equivalence classes of $a$ with a fixed $m$ and $\zeta$ is bounded by 
\begin{equation}\label{gammabddIII}
\# \Gamma(\zeta,4k+3,m)\le 2^{4k+3} (4k+3)^{3(m-1)}
\end{equation}
On the other hand, using Lemma \ref{lem:mbdd} we get for any $a$
\begin{equation}\label{mbdddIII}
m = \#E^{1}_{a} - 2wt(a) +2 + (4k+3) \le \#E^{1}_{\mathfrak{w}}-2\zeta + 2+ (4k+3)\le 8k+5.
\end{equation}
Here the last step follows from $\#E^{1}_{\mathfrak{w}} \le 4k+2$ and $\zeta \ge 1$. Plugging in the upper bound of $m$ in the R.S. of (\ref{mbdddIII}) in (\ref{gammabddIII}) we get for any $m$,
\begin{equation}\label{gammabddIV}
\# \Gamma(\zeta,4k+3,m)\le 2^{4k+3} (4k+3)^{3(8k+4)}.
\end{equation}

Now observe that for any $\mathcal{E}_{T_1}$ and $\mathcal{E}_{T_2}$, $V_{n,w,\mathcal{E}_{T_1}}$ and $V_{n,x,\mathcal{E}_{T_2}}$ are product of independent Bernoulli random variables multiplied with some deterministic constants. So $\Cov(V_{n,w,\mathcal{E}_{T_1}},V_{n,x,\mathcal{E}_{T_2}})=0$ unless all the random variables in the product $V_{n,w,\mathcal{E}_{T_1}}V_{n,x,\mathcal{E}_{T_2}}$ are repeated at least twice.  On the other hand, for all $(\mathcal{E}_{T_1},\mathcal{E}_{T_2})$ where all the random variables in the product $V_{n,w,\mathcal{E}_{T_1}}V_{n,x,\mathcal{E}_{T_2}}$ are repeated at least twice, by Jensen's inequality
$\E|V_{n,w,\mathcal{E}_{T_1}}  V_{n,x,\mathcal{E}_{T_2}} | \ge  \E|  V_{n,w,\mathcal{E}_{T_1}} |
\E |  V_{n,x,\mathcal{E}_{T_2}} |$.
 As a consequence,
\[
\big|\Cov\big(V_{n,w,\mathcal{E}_{T_1}},V_{n,x,\mathcal{E}_{T_2}} \big)\big| \le 2 \E \big|V_{n,w,\mathcal{E}_{T_1}}  V_{n,x,\mathcal{E}_{T_2}}  \big| .
\]
Recall that $\mathcal{E}_{L_i}= \mathcal{E}\backslash \mathcal{E}_{T_i}$ for $i=1,2$.
Now, (\ref{varvnw}) can be upper bounded by 
\begin{equation}\label{embedded}
\begin{split}
&\left(\frac{1}{n\pnav(1-\pnav)}\right)^{2k+1}  \Var\left[\sum_{w}V_{n,w}\right]\\
&= \left(\frac{1}{n\pnav(1-\pnav)}\right)^{2k+1} \sum_{\mathfrak{w}} \sum_{(w,x)\in f^{-1}(\mathfrak{w})} \sum_{\mathcal{E}_{T_1},\mathcal{E}_{T_2}} \Cov\left(V_{n,w,\mathcal{E}_{T_1}},V_{n,x,\mathcal{E}_{T_2}} \right)\\
&\le \left(\frac{1}{n\pnav(1-\pnav)}\right)^{2k+1}\sum_{\mathfrak{w}}\sum_{(w,x)\in f^{-1}(\mathfrak{w})}\sum_{\mathcal{E}_{T_1},\mathcal{E}_{T_2}} 2 
\E \left|V_{n,w,\mathcal{E}_{T_1}}  V_{n,x,\mathcal{E}_{T_2}}  \right|  \\
&= 2\left(\frac{1}{n\pnav(1-\pnav)}\right)^{2k+1}
\\ 
&~~~ \times \sum_{\mathfrak{w}}\sum_{(w,x)\in f^{-1}(\mathfrak{w})}\sum_{\mathcal{E}_{T_1},\mathcal{E}_{T_2}} d^{\#\mathcal{E}_{T_1}+\#\mathcal{E}_{T_2}} 
\E \bigg|\prod_{e \in \mathcal{E}_{L_1}}(x_{e(w)}-p_{e(w)}) \times \prod_{e \in \mathcal{E}_{L_2}}(x_{e(x)}-p_{e(x)}) \bigg| \\
&\le 2\left(\frac{1}{n\pnav(1-\pnav)}\right)^{2k+1}\sum_{\mathfrak{w}}\sum_{(w,x)\in f^{-1}(\mathfrak{w})} \sum_{\emptyset \subsetneq \mathcal{E}_{T} \subsetneq \mathcal{E}_{4k+2}} d^{\# \mathcal{E}_{T}} 
\E  \bigg|\prod_{e \in \mathcal{E}_{L}} (x_{e(\mathfrak{w})}-p_{e(\mathfrak{w})}) \bigg| .
\end{split}
\end{equation}
Here $\mathcal{E}_{4k+2}:= \{ (0,1),\ldots,(4k+1,4k+2) \}$ and $\mathcal{E}_{T}$ and $\mathcal{E}_L$ give a disjoint partition of $ \mathcal{E}_{4k+2}$.
Let $w=[\alpha_{i}]_{i=0}^{2k+1}$, $x=[\beta_{i}]_{i=0}^{2k+1}$ and $\mathfrak{w}=[\gamma_{i}]_{i=0}^{4k+2}$. 
Also let for any $e=(e_1,e_2)\in \mathcal{E}$, $\alpha_{e}=(\alpha_{e_1},\alpha_{e_2})$ similarly define $\beta_{e}$ and $\gamma_{e}$.
The last expression follows from the construction of $\mathfrak{w}$ from $(w,x)$. 
The most important observation here is, after fixing the words $w$ and $x$, the function $f$ is defined in such a way that for any $e^{(1)} \in \mathcal{E}_{L_1}$ and $e^{(2)} \in \mathcal{E}_{L_2}$ there are unique $e'^{(1)},e'^{(2)} \in \mathcal{E}_{4k+2}$ such that $\alpha_{e^{(1)}}= \gamma_{e'^{(1)}}$ and $\beta_{e^{(2)}}=\gamma_{e'^{(2)}}$. Further, $e'^{(1)} \neq e'^{(2)}$. Hence the last expression of (\ref{embedded}) is justified.
  
It is easy to observe that 
$\E|\prod_{e \in \mathcal{E}_{L}}(x_{e(\mathfrak{w})}-p_{e(\mathfrak{w})})  |   \le (1+o(1))(\pnav(1-\pnav))^{ \# E ( \mathcal{E}_{L} )    }$
and plugging in the estimate $\#f^{-1}(\mathfrak{w})\le (4k+3)^3$ we find the last expression in (\ref{embedded}) is further bounded by 
\begin{equation}\label{furtherembed}
\begin{split}
&(1+o(1))2\left(\frac{1}{n\pnav(1-\pnav)}\right)^{2k+1}(4k+3)^3\sum_{\mathfrak{w}}\sum_{\mathcal{E}_{T}\neq \emptyset} d^{\#\mathcal{E}_{T}}(\pnav(1-\pnav))^{ \# E ( \mathcal{E}_{L} )    }\\
&\le 4 \left(\frac{1}{n\pnav(1-\pnav)}\right)^{2k+1}(4k+3)^3 \sum_{\mathbf{w}} \sum_{\mathcal{E}_{T} \neq \emptyset}\sum_{\mathfrak{w} \in \mathbf{w}} d^{\#\mathcal{E}_{T}}(\pnav(1-\pnav))^{ \# E ( \mathcal{E}_{L} )    }\\
&= 4 \left(\frac{1}{n\pnav(1-\pnav)}\right)^{2k+1}(4k+3)^3\sum_{\mathbf{w}} \sum_{\mathcal{E}_{T} \neq \emptyset} n^{\#V_{\mathbf{w}}}d^{\#\mathcal{E}_{T}}(\pnav(1-\pnav))^{ \# E ( \mathcal{E}_{L} )    }\\
&\le 4 (4k+3)^3C'^{k} \sum_{\mathbf{w}} \sum_{\mathcal{E}_{T} \neq \emptyset} \left(\frac{1}{n}\right)^{2k+1- \#V + \frac{\# \mathcal{E}_{T}}{2}} \left( \frac{1}{\pnav} \right)^{2k+1-\#E( \mathcal{E}_{L})-\frac{\# \mathcal{E}_{T}}{2}}.
\end{split}
\end{equation} 
Here $\mathbf{w}$ is the equivalence class corresponding to $\mathfrak{w}$ and $C'$ is a numeric constant. 
Arguing as (\ref{eletbound}) and (\ref{pownvspowp}) we again get 
\begin{equation}\label{eletboundII}
2k+1-\#E( \mathcal{E}_{L})-\frac{\# \mathcal{E}_{T}}{2} \ge 0
\end{equation}
 and 
 \begin{equation}\label{pownvspowpII}
 \#E( \mathcal{E}_{L}) + \# \mathcal{E}_{T} - \#V_{\mathbf{w}} \ge \#E_{\mathbf{w}} -\#V_{\mathbf{w}} \ge 0.
 \end{equation}
To see the last inequality in \eqref{pownvspowpII}, notice that $\#E_w \geq \#V_w$ and $\#E_x \geq \#V_x$ due to the parity principle,
and the embedding algorithm ensures that $\#E_w + \#E_x - \#E_\mathbf{w} \leq \#V_w + \#V_x - \#V_\mathbf{w}$.
Indeed we can further show that the inequality in (\ref{pownvspowpII}) is always strict. Recall that if the equality in (\ref{pownvspowpII}) holds, then $G_{\mathbf{w}}$ is a unicyclic graph and every edge in $E(\mathcal{E}_{T})$ has been traversed exactly once. Let $\mathfrak{w}=f(w,x) \in \mathbf{w}$ be any word and $Z_{\mathfrak{w}}$ and $F_{\mathfrak{w}}$ be the bracelet and the forest in $G_{\mathfrak{w}}$ where $r$  is the circuit length. As $\#E(\mathcal{E}_{T}) >0$, $\#E^{1}_{\mathfrak{w}}>0$. We have argued earlier that in this case $\#E^{1}_{\mathfrak{w}} =1$ and all the edges in the bracelet $Z_{\mathfrak{w}}$ have traversed exactly once. On the other hand $f$ is defined in such a way that $V_{\mathfrak{w}}=V_{w} \cup V_{x}$ and $E_{\mathfrak{w}}= E_{w} \cup E_{x}$. As $l(w)=2k+2$ and $l(x)=2k+2$, $\#E_{w} \ge \#V_{w}$ and $\#E_{x} \ge \#V_{x}$ by the parity principle. This forces both $G_{w}$ and $G_{x}$ to be unicyclic. This means $Z_{w}=Z_{x}=Z_{\mathfrak{w}}$ since $E_{w} \cap E_{x} \neq \emptyset$. This is a contradiction to the fact that all the edges in $Z_{\mathfrak{w}}$ have been traversed exactly once by the word $\mathfrak{w}$. As a consequence the term inside the summand of the last expression in (\ref{furtherembed}) is bounded by $\frac{1}{n}$ for any $\mathbf{w}$ and any $\mathcal{E}_{T}$. Plugging in this estimate and recalling that there are at most $2^{4k+3} (4k+3)^{3(8k+4)}$ many $\mathbf{w}$'s and at most $2^{4k+2}$ many $\mathcal{E}_{T}$, we come to the following final upper bound to the last expression in (\ref{furtherembed}):
\begin{equation}\label{finalembed}
4 (4k+3)^3C'^{k}2^{4k+2}2^{4k+3} (4k+3)^{3(8k+4)}\frac{1}{n}.
\end{equation}
Analysis similar to (\ref{firstana}) will prove that (\ref{finalembed}) goes to $0$. This completes the proof. \hfill{$\square$}

\subsection{Proof of part (iv) of Theorem \ref{thm:null} and part (iv) of Theorem \ref{thm:alt}}

Here we focus on the proof of part (iv) of Theorem \ref{thm:null}. 
The proof of part (iv) of Theorem \ref{thm:alt} is similar.
We first state two important Lemmas which will play important roles in the proof. 
\begin{lemma}(Bernstein inequality)\label{lem:bern}
Let $\{X_{i}\}_{i=1}^{m}$ be independent mean $0$ random variables such that $|X_i|\le M$ for some fixed $M$. Then for any $s>0$, 
\begin{equation}\label{bernstein}
\mathbb{P}\left[ \sum_{i=1}^{m} X_{i} \ge s \right] \le \exp \left( - \frac{\frac{1}{2}s^2}{\sum_{i}\E[X_{i}^2]+\frac{1}{3}Ms} \right).
\end{equation}
In particular, if $X_{i}$'s are i.i.d. centered Bernoulli $p$ random variables, then 
\begin{equation}\label{usebern}
\mathbb{P}\left[ \left|\sum_{i=1}^{m} X_{i} \right|\ge s \right] \le 2 \times \left\{
\begin{array}{ll}
\exp \left( - \frac{3 s}{4}\right) & \text{if} ~ 3mp(1-p) \le s, \\
\exp \left( - \frac{s^2}{4mp(1-p)} \right) & \text{if} ~ 3mp(1-p) > s.
\end{array}
\right.
\end{equation} 
The above inequality directly follows from plugging in $M=1$ and using inequality (\ref{bernstein}) on $X_{i}$ and $-X_{i}$ and taking the union bound.
\end{lemma} 
\noindent
This is a well known inequality in probability theory hence its proof will be omitted.  

For any event $E$, let $\mathbb{I}_E$ stand for its indicator function.
\begin{lemma}\label{lem:condexpec}
\begin{enumerate}
\item Suppose $A$ and $B$ are any two random variables. Then 
\begin{equation}
\E\left[\left|AB\mathbb{I}_{|B|\le s}\right|\right]\le s\E[|A|].
\end{equation}
\item Let $E$ be an event with $\mathbb{P}(E) \ge (1-c)$ and $A$ be any random variable with $|A|\le 1$. Then 
\begin{equation}
\left|\E\left[A\mathbb{I}_E\right]- \E[A]\right| \le c.
\end{equation} 
\end{enumerate}
\end{lemma}
\noindent 
The proofs follow from direct application of the definition of expectation. We omit the details.

\bigskip

With Lemma \ref{lem:bern} and Lemma \ref{lem:condexpec} in hand, we now turn to the proof of part (iv) of Theorem \ref{thm:null}.
In the rest of this subsection, all expectation and variance are taken with respect to $\mathbb{P}_{0,n}$.
The fundamental idea behind this proof is the following. As we estimate $\hpnav= \frac{1}{n(n-1)}\sum X_{ij}$,
$\Var(\hpnav)= \frac{\sqrt{2\pnav(1-\pnav)}}{\sqrt{n(n-1)}} \approx \frac{\sqrt{2\pnav}}{n}$. 
As a consequence, one expect that in a typical realization  $\frac{\sqrt{\pnav}}{n}\ll |\hpnav-\pnav| \ll \frac{\sqrt{\pnav}}{\sqrt{n}}$.
Then one could imitate the proof of part (i) of Theorem \ref{thm:alt}. We now formalize these ideas.

At first we fix some $\delta \in (\frac{1}{2},1)$. Let
\begin{equation}\label{def:Ev}
\mathrm{Ev}:= \mathbb{I}_{\left|\hpnav-\pnav\right|\le \frac{\sqrt{\pnav}}{n^{\delta}}} .
\end{equation}
Recall from (\ref{cen2}), that 
\begin{equation}\label{traceexpcen2}
\Tr(\Acentwo^{2k+1})= \left(\frac{1}{n\hpnav(1-\hpnav)}\right)^{\frac{2k+1}{2}}\sum_{w: l(w)= 2k+2 ~ \& ~ w ~\text{closed}} \left[\hat{X}_{w}\right].
\end{equation}
Here for any word $w$, we define $\hat{X}_{w}:= \prod_{j=0}^{2k}\left(x_{i_j,i_{j+1}}- \hpnav\right).$ We now write the R.S. of (\ref{traceexpcen2}) in the following way:
\begin{equation}\label{tracecomcen2}
\begin{split}
&\left(\frac{\pnav(1-\pnav)}{\hpnav(1-\hpnav)}\right)^{\frac{2k+1}{2}}\left(\frac{1}{n\pnav(1-\pnav)}\right)^{\frac{2k+1}{2}}\sum_{w~:~ l(w)= 2k+2 ~ \& ~ w ~\text{closed}} \left[\hat{X}_{w}\right]\\
&= \mathrm{Ev}\left(\frac{\pnav(1-\pnav)}{\hpnav(1-\hpnav)}\right)^{\frac{2k+1}{2}}\left(\frac{1}{n\pnav(1-\pnav)}\right)^{\frac{2k+1}{2}}\sum_{w~:~ l(w)= 2k+2 ~ \& ~ w ~\text{closed}} \left[\hat{X}_{w}\right]\\
& ~~~~~~~~~~~~~~~~~~~+(1-\mathrm{Ev}) \left(\frac{1}{n\hpnav(1-\hpnav)}\right)^{\frac{2k+1}{2}}\sum_{w~:~ l(w)= 2k+2 ~ \& ~ w ~\text{closed}} \left[\hat{X}_{w}\right].
\end{split}
\end{equation}
Now we apply (\ref{usebern}) with $m= \frac{n(n-1)}{2}$ and $s= \frac{m\sqrt{\pnav}}{n^{\delta}}$  to get that 
\begin{equation}\label{probbound:Ev}
\mathbb{P}\left[ \mathrm{Ev} =0\right]\le 2 \times
\left\{  
\begin{array}{ll}
\exp\left(-\frac{3s}{4}\right) & \text{if} ~ 3\sqrt{\pnav}(1-\pnav) \le \frac{1}{n^{\delta}} \\
\exp\left( - \frac{s^2}{4m\pnav(1-\pnav)} \right) & \text{if} ~ 3\sqrt{\pnav}(1-\pnav) \ge \frac{1}{n^{\delta}}.
\end{array}
\right. 
\end{equation}
Since $m=O(n^2)$, 
\[
s = \frac{m\sqrt{\pnav}}{n^{\delta}} =O(n^{2-\delta}\sqrt{\pnav})= O(n^{\frac{3}{2}-\delta}\sqrt{n\pnav}) \gg \sqrt{n}.
\]
On the other hand,
\[
\frac{s^2}{m\pnav(1-\pnav)}= O\left( \frac{m^2\pnav}{n^{2\delta}m\pnav(1-\pnav)} \right)= O\left( \frac{m}{n^{2\delta}} \right)=O(n^{2-2\delta}).
\]
As a consequence, in either case there exists some $\eta >0$ such that $$\mathbb{P}\left[ \mathrm{Ev} =0\right] \le \exp(- n^{\eta}) \to 0.$$
So we can ignore the second term in the last expression of (\ref{tracecomcen2}). 

\medskip

Now we analyze the first term of (\ref{tracecomcen2}).
Observe that when $\mathrm{Ev}=1$, 
\[
\frac{\hpnav}{\pnav}= 1+ \frac{\hpnav- \pnav}{\pnav}= 1+ O\left(\frac{\sqrt{\pnav}}{n^{\delta}\pnav}\right)= 1+ O\left( \frac{1}{n^{\delta-\frac{1}{2}}\sqrt{n\pnav}} \right)
\]
and 
\[
\frac{1-\hpnav}{1-\pnav}= 1- \frac{\hpnav- \pnav}{1-\pnav}.
\]
Now 
\[
\left|\left( 1+ \frac{\hpnav- \pnav}{\pnav} \right)^{\frac{2k+1}{2}} -1 \right|= O\left( \frac{2k+1}{2n^{\delta-\frac{1}{2}}\sqrt{n\pnav}}\right) \to 0.
\]
Hence 
$\left(\frac{\pnav}{\hpnav}\right)^{\frac{2k+1}{2}} \to 1$.
A similar argument proves that 
$\left(\frac{1-\pnav}{1-\hpnav}\right)^{\frac{2k+1}{2}} \to 1$.
Now
\begin{equation}
\begin{split}
&\mathrm{Ev}\left(\frac{1}{n\pnav(1-\pnav)}\right)^{\frac{2k+1}{2}}\sum_{w~:~ l(w)= 2k+2 ~ \& ~ w ~\text{closed}} \left[\hat{X}_{w}\right]\\
&=\mathrm{Ev}\left(\frac{1}{n\pnav(1-\pnav)}\right)^{\frac{2k+1}{2}}\left(\sum_{w~:~ l(w)= 2k+2 ~ \& ~ w ~\text{closed}} \left[{X}_{w}\right] + \sum_{w} E_{n,w}\right)
\end{split}
\end{equation} 
where 
\begin{equation}\label{def:enw}
E_{n,w}= \sum_{\mathcal{E}_{T}} (\pnav- \hpnav)^{\# \mathcal{E}_{T}} \prod_{e \in \mathcal{E}_{L}} (x_{e(w)}-\pnav).
\end{equation}
Here $\mathcal{E}_{T}$ and $\mathcal{E}_{L}$ are as defined in the proof of part (i) of Theorem \ref{thm:alt}. 
As a consequence, in order to prove part (iv) of Theorem \ref{thm:null}, it is enough to prove 
\[
\E\left[ \left(\mathrm{Ev} \left(\frac{1}{n\pnav(1-\pnav)}\right)^{\frac{2k+1}{2}}\sum_{w} E_{n,w}\right)^2 \right] \to 0.
\]
To this end, first note
\begin{equation}\label{pnavexp}
\begin{split}
&\E\left[ \left(\mathrm{Ev} \left(\frac{1}{n\pnav(1-\pnav)}\right)^{\frac{2k+1}{2}}\sum_{w} E_{n,w}\right)^2 \right]\\
&= \E \left[ \mathrm{Ev} \left(\frac{1}{n\pnav(1-\pnav)}\right)^{2k+1} \sum_{w}\sum_{x} E_{n,w}E_{n,x}  \right]\\
&=  \left(\frac{1}{n\pnav(1-\pnav)}\right)^{2k+1} \times \\
&~~~~~ \E\left[ \mathrm{Ev} \sum_{w,x,\mathcal{E}_{T_1},\mathcal{E}_{T_2}}\left( \pnav-\hpnav\right)^{\# \mathcal{E}_{T_1}+\#\mathcal{E}_{T_2}}\prod_{e \in \mathcal{E}_{L_1}} (x_{e(w)}-\pnav) \prod_{e \in \mathcal{E}_{L_2}} (x_{e(x)}-\pnav) \right].
\end{split}
\end{equation}
We divide the remaining arguments into two different cases, depending on whether $E(\mathcal{E}_{L_1}(w))\cup E(\mathcal{E}_{L_2}(x))$ has any edge that has been traversed only once.

\medskip

\noindent
\underline{Case 1: At least one edge in $E(\mathcal{E}_{L_1}(w))\cup E(\mathcal{E}_{L_2}(x))$ has been traversed exactly once.}
Let $\mathrm{def}(w,x)$ $(= \mathrm{def}_{L_1,L_2}(w,x))\ge 1$ be the total number of edges in $E(\mathcal{E}_{L_1}(w))\cup E(\mathcal{E}_{L_2}(x))$ which have been traversed exactly once. In this case,
$$\E\left[\prod_{e \in \mathcal{E}_{L_1}} (x_{e(w)}-\pnav)\prod_{e \in \mathcal{E}_{L_2}} (x_{e(x)}-\pnav)\right]=0.$$ It is easy to check that 
\[
\left|\prod_{e \in \mathcal{E}_{L_1}} (x_{e(w)}-\pnav)\prod_{e \in \mathcal{E}_{L_2}} (x_{e(x)}-\pnav)\right|<1
\]
since each $|x_{e}-\pnav|< 1$. 
We now expand the last expression of (\ref{pnavexp}) in this case.
\begin{equation}\label{pnavexpII}
\begin{split}
& \E \left[ \mathrm{Ev} \left( \pnav-\hpnav\right)^{\# \mathcal{E}_{T_1}+\#\mathcal{E}_{T_2}}\prod_{e \in \mathcal{E}_{L_1}} (x_{e(w)}-\pnav) \prod_{e \in \mathcal{E}_{L_2}} (x_{e(x)}-\pnav) \right] \\
&= \E \left[\mathrm{Ev} \left(\frac{2}{n(n-1)}\right)^{\# \mathcal{E}_{T_1}+\#\mathcal{E}_{T_2}} \times 
\sum_{I_1,\ldots, I_{\# \mathcal{E}_{T_1}+\#\mathcal{E}_{T_2}}}
P_{w,x, I_j}(\mathcal{E}_{T_1},\mathcal{E}_{L_1},\mathcal{E}_{T_2},\mathcal{E}_{L_2})
 \right]
\end{split}
\end{equation}
where for any $I_{j} \in \{ (u,v)~|~ 1\le u<v \le n\}$
\begin{equation}
	\label{eq:subcase}
\begin{split}
& P_{w,x, I_j}(\mathcal{E}_{T_1},\mathcal{E}_{L_1},\mathcal{E}_{T_2},\mathcal{E}_{L_2})\\
& ~~~=\prod_{j=1}^{\# \mathcal{E}_{T_1}+\#\mathcal{E}_{T_2}}(x_{I_j}-\pnav)\prod_{e \in \mathcal{E}_{L_1}} (x_{e(w)}-\pnav) \prod_{e \in \mathcal{E}_{L_2}} (x_{e(x)}-\pnav).
\end{split}
\end{equation}

\medskip

\underline{Subcase (a):
Every random variable in \eqref{eq:subcase}
has been repeated at least twice.} 
Clearly in this case $\# \mathcal{E}_{T_1}+\#\mathcal{E}_{T_2} \ge \mathrm{def}(w,x)$ since otherwise there are simply not enough random variables in the first product to match those that appear only once in the second and the third products combined. 
Let $l\ge \mathrm{def}(w,x)$ be the number of random variables common between
\begin{equation}
	\label{eq:subcase-prod1}
	\prod_{j=1}^{\# \mathcal{E}_{T_1}+\#\mathcal{E}_{T_2}}(x_{I_j}-\pnav)
\end{equation}
and 
\begin{equation}
	\label{eq:subcase-prod2}
	\prod_{e \in \mathcal{E}_{L_1}} (x_{e(w)}-\pnav) \prod_{e \in \mathcal{E}_{L_2}} (x_{e(x)}-\pnav).
\end{equation}
Note that there are at most $\binom{ \#\left(E(\mathcal{E}_{L_1}(w)) \cup E(\mathcal{E}_{L_2}(x))\right)}{l}$ ways these common random variables can be chosen from the product in \eqref{eq:subcase-prod2}.
Once any such collection of random variables are fixed, we look at the positions occupied by these $l$ common random variables in the product in \eqref{eq:subcase-prod1}.
Let $\theta$ be the number of positions occupied by these $l$ random variables. Clearly, $\theta \ge l$. There are $\binom{\# \mathcal{E}_{T_1}+ \#\mathcal{E}_{T_2}}{\theta}$ many choices of the positions. Once these positions are fixed, the chosen $l$ random variables induces a partition of these $\theta$ positions into $l$ blocks. There are at most $l^{\theta}$ many partitions of $\theta$ objects into $l$ blocks. Finally one can permute the $l$ random variables once such a partition is fixed. This further induces an additional $l!$ factor. Once all these are fixed, one is free to choose the rest of the positions in the product \eqref{eq:subcase-prod1},
for those random variables which have not appeared in the product \eqref{eq:subcase-prod2}.
Now 
\begin{equation}
\begin{split}
&\sum_{j: j ~ \text{is not fixed}} \prod_{j=1}^{\# \mathcal{E}_{T_1}+\#\mathcal{E}_{T_2}}(x_{I_j}-\pnav)\\
& =  \prod_{j: j ~ \text{is fixed}} (x_{I_j}-\pnav)\left( \sum_{j~:~ j ~ \text{is not fixed}} (x_{I_j}-\pnav) \right)^{\# \mathcal{E}_{T_1}+ \mathcal{E}_{T_2}-\theta}\\
&=  \prod_{j: j ~ \text{is fixed}} (x_{I_j}-\pnav) \left( \pnav- \hpnav + O\left(\frac{k}{n^2}\right) \right)^{\# \mathcal{E}_{T_1}+ \mathcal{E}_{T_2}-\theta} \left( \frac{n(n-1)}{2}\right)^{\# \mathcal{E}_{T_1}+ \mathcal{E}_{T_2}-\theta}
\end{split}
\end{equation}
Observe that each of the quantities, $l!\binom{\# \mathcal{E}_{T_1}+\#\mathcal{E}_{T_2}}{l} \le (2k+1)^{(2k+1)}$, $\binom{ \#\left(E(\mathcal{E}_{L_1}(w)) \cup E(\mathcal{E}_{L_2}(x))\right)}{l}$ and $\theta^{l}$ are uniformly bounded by $(2k+1)^{2k+1}$. 
These estimates allow us to write R.S. of (\ref{pnavexpII}) in the following way 
 \begin{equation}\label{pnavexpfirstreduction}
 \begin{split}
 & \E\left|(2k+1)^{6k+3}\mathrm{Ev} \sum_{l\ge \mathrm{def}(w,x)} \sum_{\theta \ge l}\left(\frac{2}{n(n-1)}  \right)^{\theta} \left( \pnav- \hpnav + O\left(\frac{k}{n^2}\right) \right)^{\# \mathcal{E}_{T_1}+ \mathcal{E}_{T_2}-\theta} R_{l,\theta} \right|.
 \end{split}
 \end{equation}
 Here $R_{l,\theta}$ is a monomial of $\#\left(E(\mathcal{E}_{L_1}(w)) \cup E(\mathcal{E}_{L_2}(x))\right)$ many independent Bernoulli random variables such that each of the random variables appear more than once.
Now  
\[
\E\left| R_{l,\theta} \right| \le \left( \pnav(1-\pnav) \right)^{ \#E_{a}(L_1, L_2)}.
\]
Here $E_{a}(L_1,L_2):=\left(E(\mathcal{E}_{L_1}(w)) \cup E(\mathcal{E}_{L_2}(x))\right)$ is a subset of edges in the graph $G_{a}$ for the sentence $a=[w,x].$ 
Now applying Lemma \ref{lem:condexpec} and using the fact when $\mathrm{Ev}=1$, $|\pnav-\hpnav|\le \frac{\sqrt{\pnav}}{n^{\delta}}$ and $\frac{\sqrt{\pnav}}{n^{\delta}} \gg \frac{k}{n^2}$, we can bound (\ref{pnavexpfirstreduction}) by  
\begin{equation}\label{expectationEV}
\begin{split}
&(2k+1)^{6k+3} \sum_{l\ge \mathrm{def}(w,x)} \sum_{\theta \ge l}\left(\frac{2}{n(n-1)}  \right)^{\theta} \left( \frac{\sqrt{\pnav}}{n^{\delta}} \right)^{\# \mathcal{E}_{T_1}+ \# \mathcal{E}_{T_2}-\theta} \left[\left( \pnav(1-\pnav)\right)^{\#E_{a}(L_1,L_2)} 
 + \exp(-n^{\eta}) \right]\\
&\le (2k+1)^{6k+3}\left( 2\pnav(1-\pnav)\right)^{\#E_{a}(L_1,L_2)} \left( \frac{\sqrt{\pnav}}{n^{\delta}} \right)^{\# \mathcal{E}_{T_1}+ \# \mathcal{E}_{T_2}} \sum_{l\ge \mathrm{def}(w,x)} \sum_{\theta \ge l}\left(\frac{2n^{\delta}}{n(n-1)\sqrt{\pnav}}\right)^{\theta}\\
& \le C (2k+1)^{6k+3}\left( 2\pnav(1-\pnav)\right)^{\#E_{a}(L_1,L_2)} \left(\frac{2}{n(n-1)}\right)^{\mathrm{def}(w,x)} \left(\frac{\sqrt{\pnav}}{n^{\delta}}\right) ^{\# \mathcal{E}_{T_1}+\#\mathcal{E}_{T_2}-\mathrm{def}(w,x)}
\end{split}
\end{equation}
form some numeric constant $C$.
Here we have used the facts that $\frac{2n^{\delta}}{n(n-1)\sqrt{\pnav}} \to 0$ and that 
$ (\pnav(1-\pnav))^{Ck}  \gg \exp(-n^{\eta})$ for any positive numeric constants $C$ and $\eta$ and ${\#E_{a}(L_1,L_2)} = O(k)$.

Now we look at the equivalence classes corresponding to the sentence $a=[w,x]$. Fixing any equivalence class $\mathbf{a}$ , let $\#V_{\mathbf{a}}$ be the number of vertices in the graph $G_{\mathbf{a}}$. Summing the R.S. of the last expression of (\ref{expectationEV})  over all $a\in \mathbf{a}$ and dividing the sum by $\left( \frac{1}{n\pnav(1-\pnav)} \right)^{2k+1}$, we have it is less than or equal to
\begin{equation}\label{simplifyeqn}
\begin{split}
&D^{2k+1} (2k+1)^{6k+3}\left( \frac{1}{n\pnav} \right)^{2k+1}n^{\#V_{\mathbf{a}}}\left( \pnav \right)^{\#E_{\mathbf{a}}(L_1,L_2)}\left(\frac{1}{n^2}\right)^{\mathrm{def}(\mathbf{a})} \left(\frac{\sqrt{\pnav}}{n^{\delta}}\right) ^{\# \mathcal{E}_{T_1}+\#\mathcal{E}_{T_2}-\mathrm{def}(\mathbf{a})}\\
\end{split}
\end{equation}
Here $\#V_{\mathbf{a}}$, $\#E_{\mathbf{a}}(L_1,L_2)$ and $\mathrm{def}(\mathbf{a})$ are the common value of $\#V_{a}$, $\#E_{a}(L_1,L_2)$ and $\mathrm{def}(a)$ for any $a \in \mathbf{a}$. Also note that we have ignored the terms containing  $(1-\pnav)$ since $\lim_{n}(1-\pnav)>(1-p)$ and $p \in [0,1)$.
Simplifying, we have the powers of $\frac{1}{n}$ and $\frac{1}{\pnav}$ in (\ref{simplifyeqn}) are given by 
\begin{equation}\label{powernsimplifyeqn}
2k+1 - \#V_{\mathbf{a}}+ 2\mathrm{def}(\mathbf{a})+ \delta(\# \mathcal{E}_{T_1}+\#\mathcal{E}_{T_2}-\mathrm{def}(\mathbf{a}))
\end{equation}
and 
\begin{equation}\label{powerpsimplifyeqn}
2k+1 - \#E_{\mathbf{a}}(L_1,L_2)- \frac{1}{2}(\# \mathcal{E}_{T_1}+\#\mathcal{E}_{T_2}-\mathrm{def}(\mathbf{a})).
\end{equation}
Observe that 
\begin{equation*}
\begin{split}
&2(\#E_{\mathbf{a}}(L_1,L_2)-\mathrm{def}(\mathbf{a}))+ \mathrm{def}(\mathbf{a}) \le \#\mathcal{E}_{L_1}+ \#\mathcal{E}_{L_2}
\\
& 
\Leftrightarrow \#E_{\mathbf{a}}(L_1,L_2) \le \frac{1}{2}(\#\mathcal{E}_{L_1}+ \#\mathcal{E}_{L_2}+\mathrm{def}(\mathbf{a})).
\end{split}
\end{equation*}
Plugging this estimate in (\ref{powerpsimplifyeqn}) and using the fact $\#\mathcal{E}_{L_1}+ \#\mathcal{E}_{L_2}+ \#\mathcal{E}_{T_1}+ \#\mathcal{E}_{T_2}= 4k+2,$, we have 
(\ref{powerpsimplifyeqn}) is always greater than or equal to $0$. Now we prove the difference between (\ref{powernsimplifyeqn}) and (\ref{powerpsimplifyeqn}) is always greater than or equal to $\frac{1}{2}$:
\begin{equation}\label{finalpluggine'mate}
\begin{split}
& - \#V_{\mathbf{a}}+ 2\mathrm{def}(\mathbf{a})+ \delta(\# \mathcal{E}_{T_1}+\#\mathcal{E}_{T_2}-\mathrm{def}(\mathbf{a}))+ \#E_{\mathbf{a}}(L_1,L_2)+\frac{1}{2}(\# \mathcal{E}_{T_1}+\#\mathcal{E}_{T_2}-\mathrm{def}(\mathbf{a}))\\
& \ge - \#V_{\mathbf{a}} + 2 \mathrm{def}(\mathbf{a})- (\delta+\frac{1}{2})\mathrm{def}(\mathbf{a})+ \#E_{\mathbf{a}}(L_1,L_2) + \# \mathcal{E}_{T_1}+\#\mathcal{E}_{T_2}\\
& \ge \#E_{\mathbf{a}}- \#V_{\mathbf{a}} + \frac{1}{2} \mathrm{def}(\mathbf{a})\ge \frac{1}{2} \mathrm{def}(\mathbf{a})\ge \frac{1}{2}.
\end{split}
\end{equation}
Here we have used the fact $\frac{1}{2}< \delta <1$. Recall (\ref{gammabddIII}) and (\ref{mbounddddII}) to get that there are at most $(C'k)^{D'k}$ many equivalence classes $\mathbf{a}$ where $C'$ and $D'$ are some known numbers. So summing (\ref{simplifyeqn}) over all the equivalence classes $\mathbf{a}$, we get the contribution of all terms in the current subcase in (\ref{pnavexp}) is bounded by 
\begin{equation}\label{con:subcaseI}
D^{2k+1} (2k+1)^{4k+2}(C'k)^{D'k} \frac{1}{\sqrt{n}} \to 0.
\end{equation}

\medskip

\underline{Subcase (b):
At least one random variable in the product \eqref{eq:subcase}
appears only once.} 
Here we apply Lemma \ref{lem:condexpec} to get 
\begin{equation*}
\begin{split}
&\sum_{I_1,\ldots, I_{\# \mathcal{E}_{T_1}+\#\mathcal{E}_{T_2}}}\left| \E\left[\mathrm{Ev} \left(\frac{2}{n(n-1)}\right)^{\# \mathcal{E}_{T_1}+\#\mathcal{E}_{T_2}} \times \right.\right.\\
& ~~~~~~~ \left.\left.\prod_{j=1}^{\# \mathcal{E}_{T_1}+\#\mathcal{E}_{T_2}}(x_{I_j}-\pnav)\prod_{e \in \mathcal{E}_{L_1}} (x_{e(w)}-\pnav) \prod_{e \in \mathcal{E}_{L_2}} (x_{e(x)}-\pnav) \right]\right|\\
& \le n^{2(\# \mathcal{E}_{T_1}+\#\mathcal{E}_{T_2})} \exp(-n^{\eta}).
\end{split}
\end{equation*}
One can again sum over all the equivalence classes $\mathbf{a}$ to get the contribution of the current subcase in (\ref{pnavexp}) is bounded by
\begin{equation}\label{con:subcaseII}
C_{1}^{2k+1}(C_{2}k)^{C_{3}k}n^{C_{4}k}\exp(-n^{\eta}) \le n^{C_{5}k}\exp(-n^{\eta})
\end{equation}
Here $C_{1},\ldots,C_{5}$ are some known constants. Since $k=o(\min(\sqrt{\log n},\log(n\pnav)))$, one gets $n^{C_{5}k}\exp(-n^{\eta}) \to 0$ for any $\eta>0$. As a consequence, the contribution of the current subcase in (\ref{pnavexp}) goes to $0$.

\medskip

\noindent
\underline{Case 2: All the edges in  $E(\mathcal{E}_{L_1}(w))\cup E(\mathcal{E}_{L_2}(x))$ have been traversed at least twice.} 
This case can be done by imitating the analysis of mean of (\ref{Vnw}) in the proof of part (i) of Theorem \ref{thm:alt}. In particular using arguments analogous to (\ref{eletbound}) and (\ref{pownvspowp}) for the sentence $a=[w,x]$ one can prove that the contribution of the present case in (\ref{pnavexp}) is bounded by 
\begin{equation}\label{con:caseII}
\frac{(C_{\mathrm{r}}k)^{D_{\mathrm{r}}k}}{n^{\delta-\frac{1}{2}}} \to 0
\end{equation}
for some known $C_{\mathrm{r}}$ and $D_{\mathrm{r}}$. 

\paragraph{Summary}
Combining (\ref{con:subcaseI}), (\ref{con:subcaseII}) and (\ref{con:caseII}) one gets 
\begin{equation}\label{pluginvar}
\begin{split}
&\E\left[ \left(\mathrm{Ev} \left(\frac{1}{n\pnav(1-\pnav)}\right)^{\frac{2k+1}{2}}\sum_{w} E_{n,w}\right)^2 \right]\\
& \hskip 3em \le D^{2k+1} (2k+1)^{4k+2}(C'k)^{D'k} \frac{1}{\sqrt{n}}+ n^{C_{5}k}\exp(-n^{\eta})+\frac{(C_{\mathrm{r}}k)^{D_{\mathrm{r}}k}}{n^{\delta-\frac{1}{2}}} \to 0.
\end{split}
\end{equation}
This completes the proof.
\hfill{$\square$}

\subsection{Proof of Theorem \ref{thm:even}}

In this proof we focus on proving results under both null and alternative for $A_{\mathrm{cen1}}$. 
The proof of being able to use $A_{\mathrm{cen2}}$ instead of $A_{\mathrm{cen1}}$ is similar to the proof of part (iv) of Theorem \ref{thm:null}, and hence is omitted. 

\subsubsection{Proof under the null}
Throughout this part, all expectation and variance are taken with respect to $\mathbb{P}_{0,n}$.
Observe that
\begin{equation}
\begin{split}
& \Tr\left( A_{\mathrm{cen1}}^{2k}\right)=  \left(\frac{1}{n\pnav(1-\pnav)}\right)^{k}\sum_{w} X_{w}\\
& ~~~~ =   \left(\frac{1}{n\pnav(1-\pnav)}\right)^{k} \left[\sum_{w\in \mathcal{W}_1} X_{w} + \sum_{w\in \mathcal{W}_2} X_{w} + \sum_{w \in \mathcal{W}_3 } X_{w}+ \sum_{w \in \mathcal{W}_4} X_{w} \right]. 
\end{split}
\end{equation}
Here $\mathcal{W}_1$ corresponds to the set of Wigner words, $\mathcal{W}_2$ stands for the set of all weak Wigner words (Definition \ref{def:weakwigner}), $\mathcal{W}_3=\cup_{r} \mathfrak{W}_{2k+1,r, k+r/2}$ (Proposition \ref{prop:uniword}) collects all words corresponding to unicyclic graphs with a bracelet of length at least 4, and $\mathcal{W}_{4}$ is the complement of $\mathcal{W}_1\cup \mathcal{W}_{2}\cup \mathcal{W}_{3}$. 

A direct and straightforward evaluation shows that as $n\pnav \to\infty$,
\[
\Var\left[\left(\frac{1}{n\pnav(1-\pnav)}\right)^{k}\sum_{w\in \mathcal{W}_1} X_{w} - 
k\psi_{2k} \Tr(\Acenone^2)
\right] \to 0.
\]
There are $\psi_{2k}n(n-1)\ldots(n-k)$ many words in $\mathcal{W}_1$ and each of them have expectation $\left(\pnav(1-\pnav)\right)^{k}$. As a consequence,
\begin{align*}
\E\left[ \left(\frac{1}{n\pnav(1-\pnav)}\right)^{k}\sum_{w\in \mathcal{W}_1} X_{w} \right] & = n\psi_{2k}- \psi_{2k}\sum_{j=1}^{k} j +O\left(\frac{1}{n}\right)\\
& = n\psi_{2k}- \psi_{2k}\binom{k+1}{2}+O\left( \frac{1}{n}\right).
\end{align*}
Next, arguments similar to the proof of Theorem \ref{thm:null} part (i) and (iii) lead to
\[
\left(\frac{1}{n\pnav(1-\pnav)}\right)^{k} \sum_{w \in \mathcal{W}_3 } X_{w} - \sum_{r=4:r ~ \text{even}}^{2k} \frac{f(2k,r)2k}{r} C_{n,r}(G) \stackrel{p}{\to} 0.
\]
Furthermore, by definition all the words in $\mathcal{W}_{4}$ have zero expectation and we have seen in Lemma \ref{lem:expcltpair} that these words do not contribute in the asymptotic variance of $\Tr\left(A_{\mathrm{cen}_1}^{2k}\right)$, either. 
So we can simply ignore the words in $\mathcal{W}_{4}$.
At this point it is clear that 
\[
T_{2k}-  \left(\frac{1}{n\pnav(1-\pnav)}\right)^{k}\sum_{w\in \mathcal{W}_2} X_{w} \stackrel{p}{\to} 0.
\] 
So we turn to inspecting the words $w \in \mathcal{W}_{2}$. 

The words in $\mathcal{W}_{2}$ can further be divided into the following two classes:
\begin{enumerate}
\item $w$ is a critical weak Wigner word (Definition \ref{def:weakwigner}): $\#V_{w}= k$ for $G_{w}= (V_{w},E_{w})$. 
\item $w$ is not a critical weak Wigner word. In this case we call $w$ to be a sub-critical weak Wigner words.
\end{enumerate}
The computation related to critical weak Wigner words on dense graphs can be found in \cite[pp.320-322]{AZ05}. 
However, when the graph is sparse (i.e., when $\pnav\to 0$) we have to be especially careful with the trees: 
Since the number of edges in a tree is one less than the number of vertices, one gets additional powers of $\pnav$ in the denominator. 

From Proposition \ref{prop:weakwigrep}, we know that whenever $w$ is critical weak Wigner word, $G_{w}$ is either a unicyclic graph or a tree. 
When $G_{w}$ is a tree, there is one exceptional edge which is traversed four times and all the other edges are traversed twice. Also $\#V_{w}=k$ and $\#E_{w}=k-1$. 
So for any such $w$,
\[
\E[X_{w}]= \left(\pnav(1-\pnav)\right)^{k-2} \E\left[ (x_{1,2}-\pnav)^{4} \right].
\] 
On the other hand, given any equivalence class there are $(1+o(1))n^{k}$ many such words. Let $\alpha_{2,2k}$ be the number of equivalence classes of such critical weak Wigner trees. (An exact enumeration can be found below.) 
So the total contribution of these critical weak Wigner trees to the expectation of $T_{2k}$ is given by
\[
\frac{\alpha_{2,2k}\E\left[ (x_{1,2}-\pnav)^{4}\right]}{\left(\pnav(1-\pnav)\right)^2} = (1+o(1))\frac{\alpha_{2,2k}}{\pnav}
\] 
for vanishing $\pnav$.
For the variance calculation, we need to consider the word pairs $[w_1,w_2]$ where both $w_1$ and $w_2$ are critical weak Wigner tree and the sentence $a=[w_1,w_2]$ is a weak CLT sentence. 
The leading term here comes from the case when $a$ is a tree and $w_1$, $w_2$ share exactly one edge. 
There can be three possible sub-cases here. 
Firstly, one edge in $a$ is repeated exactly eight times and all the other edges are repeated exactly twice. 
Secondly, three edges in $a$ are repeated exactly four times and all the other edges are repeated exactly twice. 
Thirdly, one edge is repeated exactly six times, one edge is repeated exactly four times and all the other edges are repeated exactly twice. 
Now $\#V_{a}= 2k-2$ since $w_1$ and $w_2$ share exactly one edge this corresponds to two common vertices. Whenever $\pnav \to 0$, in all the aforesaid cases $\E[X_{a}]=(1+o(1))\left(\pnav(1-\pnav)\right)^{2k-1}$. 
Let $v_{2k}n^{2k-2}$ be the number of such sentences. 
The contribution of these sentences in the variance of $T_{2k}$ is 
\[
(1+o(1))v_{2k}\frac{n^{2k-2}\left(\pnav(1-\pnav)\right)^{2k-1}}{n^{2k}\left(\pnav(1-\pnav)\right)^{2k}}= (1+o(1))\frac{v_{2k}}{n^2\pnav^{3}}.
\]
It can be proven that the variances of all the other random variables in $T_{2k}$ are negligible with respect to $\frac{1}{n^2\pnav^3}$. 

When $w$ is a critical weak Wigner word and $G_{w}$ is unicyclic, $\E[X_{w}]= \left(\pnav(1-\pnav)\right)^{k}$. As a consequence, the total contribution of these critical weak Wigner unicyclic words in the expectation of $T_{2k}$ is given by
\[
\alpha_{1,2k}(1+o(1)).
\] 
Here $\alpha_{1,2k}$ is the number of equivalence classes of such words. 
We will also compute the exact value of $\alpha_{1,2k}$ later in the proof.

\smallskip

Now we consider the contributions of sub-critical weak Wigner words. 
In this case, we are only concerned about the trees. 
All the words $w$ such that $G_{w}$ is a tree and $\#V_{w}= k-1$ contribute jointly a term of $\frac{\alpha_{3,2k}}{n\pnav^2}$ in the expectation of $\E[T_{2k}]$. Here $\alpha_{3,2k}$ is the number of equivalence classes of such words. 
The variance of these words is negligible compared to $\frac{1}{n^2\pnav^3}$. Unlike the previous two cases we do not know an easy way to calculate $\alpha_{3,2k}$ explicitly. 
There are also two cases arising here. Firstly, one edge in $a$ is repeated exactly six times and all the other edges are repeated exactly twice. Secondly, two edges in $a$ are repeated exactly four times and all the other edges are repeated exactly twice. 
Here one needs to consider multiple bracelets and the argument becomes tedious.
 
\medskip
 
We conclude this part by deriving the expressions of $\alpha_{1,2k}$ and $\alpha_{2,2k}$.
A generic recipe for evaluating $\alpha_{1,2k}$ and $\alpha_{2,2k}$ is given in equation (46) of \cite{AZ05},
which was done for more general matrices. 
Simplifying all the results in (46) of \cite{AZ05} for Wigner matrices one gets 
\begin{equation}\label{alpha1}
\begin{split}
\alpha_{1,2k} &= \sum_{r=3}^{k} f(2k,2r) \frac{k(r+1)}{r} = \sum_{r=3}^{k} (r+1) \binom{2k}{k+r}\\
              &= \sum_{r=3}^{k} (r+1) \binom{2k}{k+r} - \binom{k+1}{2} \psi_{2k}+ 3\binom{2k}{k+2} + \binom{k+1}{2} \psi_{2k} - 3\binom{2k}{k+2} \\
              &= \sum_{r=1}^{k}  \binom{2k}{k+r} - 2k\psi_{2k}  + \binom{k+1}{2} \psi_{2k} - 3\binom{2k}{k+2}\\
              &= 2^{2k-1} - \binom{2k}{k}\frac{5k+1}{2(k+1)} + \binom{k+1}{2} \psi_{2k} - 3\binom{2k}{k+2},
\end{split}
\end{equation}
and 
\begin{equation}\label{alpha2}
\alpha_{2,2k}= f(2k,4)\frac{k}{2}= \binom{2k}{k+2}.
\end{equation}
Note that in order to derive the final expression in (\ref{alpha1}), we have used the following identity 
\[
\sum_{r=1}^{k} \binom{2k}{k+r}- 2k \psi_{2k}= \sum_{r=3}^{k} (r+1) \binom{2k}{k+r} - \binom{k+1}{2} \psi_{2k}+ 3\binom{2k}{k+2}.
\]
This can be verified by elementary calculation. Hence the proof is skipped.
Using Lemma \ref{lem:bounded} and Lemma \ref{lem:appendix} one gets 
\begin{equation}\label{bound:alpha3}
\alpha_{3,2k} \le 2^{2k} (2k)^{12}
\end{equation}
and 
\begin{equation}\label{bound:v2k}
v_{2k} \le 2^{4k} (C_{1}k)^{C_2}
\end{equation}
for some positive numeric constants $C_{1}$ and $C_{2}$.
This completes the proof under the null.

\subsubsection{Proof under the alternative}
The proof under the alternative is very much similar in spirit to the proof of part (i) of Theorem \ref{thm:alt}. 
The major difference is here we also need to deal with those trees in which the number of edges is one less than the number of vertices. 
In what follows, we shall only focus on any word $w$ such that $G_{w}$ is a tree and shall only give the analysis of the mean. 
All the other cases follow from the arguments similar to the proof of part (i) of Theorem \ref{thm:alt} with suitable modifications for the trees described here.
We will also use notations defined at the beginning of the proof of part (i) of Theorem \ref{thm:alt}.
Throughout this part, all expectation and variance are taken with respect to $\mathbb{P}_{1,n}$ conditioning on the group assignment $\sigma_i$ for $1\leq i\leq n$.

We at first fix any word $w \in \mathcal{W}_{2}$ such that $G_{w}$ is a tree. 
Recall that for a word $w=(i_{0},\ldots, i_{2k})$ we have
\begin{equation*}
\begin{split}
&\left( \frac{1}{n\pnav(1-\pnav)} \right)^{k} \sum_{w \in \mathcal{W}_{2}: G_{w} ~ \text{is a tree}}X_{w}\\
&= \left( \frac{1}{n\pnav(1-\pnav)} \right)^{k} \sum_{w \in \mathcal{W}_{2}: G_{w} ~ \text{is a tree}}\prod_{j=0}^{2k-1}\left(x_{i_{j},i_{j+1}}-\pnav\right)\\
     &= \left( \frac{1}{n\pnav(1-\pnav)} \right)^{k} \sum_{w \in \mathcal{W}_{2}: G_{w} ~ \text{is a tree}}\left[\prod_{j=0}^{2k-1}\left(x_{i_{j},i_{j+1}}-p_{i_{j},i_{j+1}}\right) + d^{2k} + V_{n,w}\right] 
\end{split}
\end{equation*}
where 
\begin{equation*}
\begin{split}
V_{n,w} &= \sum_{ \emptyset \subsetneq \mathcal{E}_{T}   \subsetneq \mathcal{E}_{2k}} \prod_{e \in \mathcal{E}_{T}} (\sigma_{e(w)} d) \prod_{e \in \mathcal{E}_{L}} (x_{e(w)}-p_{e(w)})\\
&= \sum_{\emptyset \subsetneq \mathcal{E}_{T}\subsetneq \mathcal{E}_{2k}}  d^{\# \mathcal{E}_{T} }\prod_{e \in \mathcal{E}_{T}} \sigma_{e(w)}\prod_{e \in \mathcal{E}_{L}} (x_{e(w)}-p_{e(w)}).
\end{split}
\end{equation*}

Since the graph corresponding to any word $w \in \mathcal{W}_{2}$ has the number of vertices less than or equal to $k$, it is easy to see that 
\[
\left( \frac{1}{n\pnav(1-\pnav)} \right)^{k} \sum_{w \in \mathcal{W}_{2}: G_{w} ~ \text{is a tree}}d^{2k} \to 0.
\]

\noindent 
Arguing as before, we have $\E\left[ V_{n,w}\right] \neq 0$ only if each edge in $E(\mathcal{E}_{L}(w))$ has been repeated at least twice by the exploration $e \in \mathcal{E}_{L}.$ We shall only focus on this case.
Now fix $\emptyset \subsetneq \mathcal{E}_{T}\subsetneq \mathcal{E}_{2k} $ and an equivalence class $\mathbf{w} \subset \mathcal{W}_{2}$ corresponding to graph $G=(V,E)$ such that for any $w \in \mathbf{w}$, $G_{w}$ is a tree. Arguing as (\ref{leftrandom})-(\ref{expectationleftrandom}) one arrives at the following upper bound:  
\begin{equation}
\begin{split}
&\left(\frac{1}{n\pnav(1-\pnav)}\right)^{k}\sum_{w : w \in \mathbf{w}}\E\left[\left| d^{\# \mathcal{E}_{T} }\prod_{e \in \mathcal{E}_{T}} \sigma_{e(w)}\prod_{e \in \mathcal{E}_{L}} \left(x_{e(w)}-p_{e(w)}\right)\right|\right]\\
&\le C^{k} \left(\frac{1}{n}\right)^{k- \#V + \frac{\# \mathcal{E}_{T}}{2}} \left( \frac{1}{\pnav} \right)^{k-\#E( \mathcal{E}_{L}(w))-\frac{\# \mathcal{E}_{T}}{2}}.
\end{split}
\label{eq:thm33temp}
\end{equation}
Here $w \in \mathbf{w}$ is any word.
Observe that in this case $G_{w}$ is a tree. So we require a slight modification of (\ref{eletbound}) and (\ref{pownvspowp}) in the current scenario. 
 Firstly 
\begin{equation}\label{newpownvspowp}
\begin{split}
&\left(k- \#V + \frac{\# \mathcal{E}_{T}}{2}\right) - \left( k-\#E( \mathcal{E}_{L}(w))-\frac{\# \mathcal{E}_{T}}{2} \right)\\
& = \#E( \mathcal{E}_{L}(w)) + \# \mathcal{E}_{T} - \#V
\end{split}
\end{equation}
 as before. Now 
 \begin{equation}\label{explhs}
 \begin{split}
 \#E( \mathcal{E}_{L}(w)) + \# \mathcal{E}_{T} = \sum_{\gamma \in E_{w}} \left[ \mathbb{I}_{\gamma \in E( \mathcal{E}_{L}(w))} + \sum_{e \in \mathcal{E}_{T}} \mathbb{I}_{\gamma = e(w)}\right].
 \end{split}
 \end{equation}
Since $\emptyset \subsetneq \mathcal{E}_{T}$ and $\mathcal{E}_{L} \cup \mathcal{E}_{T}=\mathcal{E}_{2k}$, we have for all $\gamma$
\begin{equation}\label{sumgammabound}
\mathbb{I}_{\gamma \in E( \mathcal{E}_{L}(w))} + \sum_{e \in \mathcal{E}_{T}} \mathbb{I}_{\gamma = e(w)} \ge 1
\end{equation}
and there exists at least one $\gamma$ such that (\ref{sumgammabound}) is greater than equal to $2$. As a consequence, (\ref{explhs}) is greater than or equal to $\#E_{w}+1$. So the final expression of (\ref{newpownvspowp}) is always greater than equal to $0$. Observe that the equality happens only if  $\mathcal{E}_{T}$ is either exactly equal to both the traversal of an edge $\gamma$ traversed exactly twice or $\#\mathcal{E}_{T}=1$ and the corresponding edge has been traversed at least four times.

Now 
\begin{equation*}
\begin{split}
2k= \# \mathcal{E}_{L} + \# \mathcal{E}_{T}= \sum_{\gamma \in E_{w}}\left[  \sum_{e \in \mathcal{E}_{L}} \mathbb{I}_{\gamma = e(w)} + \sum_{e \in \mathcal{E}_{T}} \mathbb{I}_{\gamma = e(w)}\right]
\end{split}
\end{equation*}
Arguing similarly as (\ref{eletbound}) we always have 
\begin{equation}\label{neweletbound}
k- \#E(\mathcal{E}_{L}(w)) -\frac{\# \mathcal{E}_{T}}{2} \ge 0.
\end{equation}
However here we prove that (\ref{newpownvspowp}) and (\ref{neweletbound}) can not be $0$ simultaneously.  
Observe that as $\mathbf{w} \subset \mathcal{W}_{2}$, all the edges in $G$ is traversed at least twice and at least one edge is traversed at least four times. As a consequence, we have for all $\gamma$ we have 
\[
\sum_{e \in \mathcal{E}_{L}} \mathbb{I}_{\gamma = e(w)} + \sum_{e \in \mathcal{E}_{T}} \mathbb{I}_{\gamma = e(w)} \ge 2
\]
and there exists at least one $\gamma$ such that the above sum is greater than equal to $4$. Now consider the cases when (\ref{newpownvspowp}) is $0$. Let us fix a $\gamma' \in E_{w}$ such that $\gamma'$ is traversed at least $4$ times. In both cases, we have $\mathbb{I}_{\gamma' \in E(\mathcal{E}_{L}(w))}=1$ and 
\[
\sum_{e \in \mathcal{E}_{L}} \mathbb{I}_{\gamma'= e(w)}=4 
\] 
in the first case and 
\[
\sum_{e \in \mathcal{E}_{L}} \mathbb{I}_{\gamma'= e(w)}=3 
\]
in the second case. As a consequence, in both cases 
\[
\sum_{\gamma \in E_{w}} \sum_{e \in \mathcal{E}_{L}} \mathbb{I}_{\gamma = e(w)} \ge 2 \#E(\mathcal{E}_{L}(w))+1 
\,\,\Rightarrow\,\, 
k- \#E(\mathcal{E}_{L}(w)) -\frac{\# \mathcal{E}_{T}}{2} \ge \frac{1}{2}.
\]
Thus, we always have the right side of \eqref{eq:thm33temp} converging to $0$ as $n\pnav\to \infty$.
 \hfill{$\square$}

\subsection{Proof of Proposition \ref{prop:correction}}
We actually determine the equivalence classes in this two cases. When $2k=4$, it is easy to observe that $\alpha_{1,2k}=\alpha_{3,2k}=0$ and $\alpha_{1,2k}=1$ and each word in this class is equivalent to the word $12121$. Summing over all the words in this equivalent class one gets the random variable corresponding to this equivalence class is 
\[
 \left( \frac{1}{n\pnav(1-\pnav)} \right)^2\sum_{i,j} (x_{i,j}-\pnav)^{4}.
\]

For $k=3$, observe that $\alpha_{3,2k}=1$ and each word in this class is equivalent to the word $1212121$. Summing over all the words in this equivalent class one gets the random variable corresponding to this equivalence class is 
\[
 \left( \frac{1}{n\pnav(1-\pnav)} \right)^3\sum_{i,j} (x_{i,j}-\pnav)^{6}.
\]

Note that $\alpha_{1,2k}=4$ hence there are four distinct equivalence classes. Any word of this kind is equivalent to one of the following four words $1231231$, $1231321$, $1213231$, $1232131$. The random variables corresponding to each word have mean $\pnav^{3}(1-\pnav)^3$. It is also easy to see that the variance of the random variable corresponding to each of the equivalence class vanishes as $n\pnav\to \infty$.

Finally $\alpha_{2,2k}=6$ hence there are six distinct equivalence classes. Any word of this kind is equivalent to one of the following six words $1212321$, $1232121$, $1232321$, $1212131$, $1213121$ and $1213131$. Observe that the sum over all the words corresponding to any of this equivalence classes give rise to the following random variable
\[
 \left( \frac{1}{n\pnav(1-\pnav)} \right)^3 \sum_{i_1,i_2,i_3} (x_{i_1,i_2}-\pnav)^{4}(x_{i_2,i_3}-\pnav)^2.
\]
So the six equivalence classes will give a total contribution of 
\[
6  \left( \frac{1}{n\pnav(1-\pnav)} \right)^3 \sum_{i_1,i_2,i_3} (x_{i_1,i_2}-\pnav)^{4}(x_{i_2,i_3}-\pnav)^2.
\]
This completes the proof. \hfill{$\square$}

\subsection{Proof of Theorem \ref{thm:cyclestotrace}}

\subsubsection{Proof of part (i)}
Recall that
\begin{equation*}
f(m,r)\frac{m}{r}= \left\{
\begin{array}{ll}
{m\choose \frac{m+r}{2}} & \text{whenever}~ m-r ~\text{even}\\
0 & \text{otherwise.}
\end{array}
\right.
\end{equation*}
Observe that $\binom{m}{\frac{m+r}{2}} $ is the coefficient of $\left(z^{r}+1/z^{r}\right)$ in the expansion of $\left(z+ {1}/{z}\right)^{m}$. 
Recall the definition of $\mathbb{D}_{2k+1}$ in (\ref{def:D}). We further define 
\[
u_{2k+1}:= \left(\binom{3}{2},\binom{5}{3},\ldots, \binom{2k+1}{k+1} \right)'.
\]
Then
\begin{equation*}
\begin{split}
\left(
\begin{array}{ll}
1& 0\\
u_{2k+1}&\mathbb{D}_{2k+1}
\end{array}
\right)
\left(z+\frac{1}{z},z^3+\frac{1}{z^3},\ldots, z^{2k+1}+\frac{1}{z^{2k+1}}\right)'~~~~~~~~~&\\
= \left( \left(z+\frac{1}{z}\right), \left(z+\frac{1}{z} \right)^3,\ldots, \left(z+\frac{1}{z} \right)^{2k+1} \right)'.&
\end{split}
\end{equation*}
Hence
\begin{equation*}
\begin{split}
& \left(z+\frac{1}{z},z^3+\frac{1}{z^3},\ldots, z^{2k+1}+\frac{1}{z^{2k+1}}\right)'\\
& = \left(
\begin{array}{ll}
1& 0\\
- \mathbb{D}_{2k+1}^{-1}u_{2k+1}&\mathbb{D}_{2k+1}^{-1}
\end{array}
\right) \left( \left(z+\frac{1}{z} \right), \left(z+\frac{1}{z}\right)^3,\ldots, \left(z+\frac{1}{z}\right)^{2k+1} \right)'.
\end{split}
\end{equation*}
On the other hand, we have defined $P_{2k+1}(\cdot)$ to be such that 
\[
P_{2k+1}\left( z+ \frac{1}{z} \right)= z^{2k+1} + \frac{1}{z^{2k+1}}.
\]
As a consequence, we have 
\[
\mathbb{D}_{2k+1}^{-1}(k,j)= P_{2k+1}[2j+1].
\]
The rest of the proof of part (i) is exactly similar to the proofs of part (iii) of Theorem \ref{thm:null} and part (iii) of Theorem \ref{thm:alt}. 
We thus omit the details.

\subsubsection{Proof of part (ii)} 
Overall the proof here is similar to the proof of part (i). 
However, here we further prove $\Var\left( \sum_{r=0: r~ \text{even}}^{2k} P_{2k}(r) T_{r}\right) \to 0$ when $n^2\pnav^3 \to 0$ under both null and local alternative.
Firstly, under $\mathbb{P}_{0,n}$,
\begin{equation}
\begin{split}
\Var\left( \sum_{r=0: r~ \text{even}}^{2k} P_{2k}(r) T_{r} \right)
&\le (2k)^2 \sup_{r\le k}P_{2k}(2r)^{2} \sup_{r \le k} v_{2r} \frac{1}{n^{2}\pnav^3}
\\
&
\le (C_{3}'k)^{C_{4}'} (C_1')^{C_{2}'k}\frac{1}{n^{2}\pnav^3}
\end{split}
\end{equation}
for some universal constants $C_{1}', C_{2}',C_{3}'$ and $C_{4}'$. Here we have used the well known fact that $\sup_{r\le k}P_{2k}(2r) \le 4^{2k}$.

Under $\mathbb{P}_{1,n}$ conditioning on the group assignments, recalling the discussion of the outline of the proof of Theorem \ref{thm:even} we have,
 for any $i=1,2$  
\[
W_{2,i} \stackrel{d}{=} \tau_{i} + \Xi_{i}
\]
where $\tau_{i}$ has same asymptotic distribution as $
W_{2,i}\,|\,\mathbb{P}_{0,n}$ and  
\[
\E [\Xi_{i}] = O\left(  \frac{t^{2k}}{n}\right) \quad \mbox{and} \quad
\Var [\Xi_{i}] = O\left( \frac{\left(C_1k\right)^{C_2k}}{n} \right) \quad \mbox{for}\quad i= 1,2, 
\]
for some universal positive constants $C_{1}$ and $C_{2}$. As a consequence,
\begin{equation*}
\begin{split}
&
\Var\left( \sum_{r=0~:~ r~ \text{even}}^{2k} P_{2k}(r) T_{r} \right)
\\
&~~~~
\le (C_{3}'k)^{C_{4}'} (C_1')^{C_{2}'k}\frac{1}{n^{2}\pnav^3} + (2k)^2 \sup_{r\le k}P_{2k}(2r)^{2} \frac{\left(C_1k\right)^{C_2k}}{n} \to 0.
\end{split}
\end{equation*}
This completes the proof. 
\hfill{$\square$}

\subsection{Proof of Proposition \ref{prop:multi}} 
We now provide a result which justify the consistency of the tests proposed in Proposition \ref{prop:sing}. Following the discussion before the statement of Proposition \ref{prop:multi} it is clear that the results in (\ref{formula1}), Theorem \ref{thm:even} and Theorem \ref{thm:cyclestotrace} remains valid in the multiple block case also. In particular, one can still write for any $k=o\left(\min\left(\log(np_{n,\mathrm{av}}), \sqrt{\log{n}}\right)\right)$
\begin{equation*}
C_{n,k}(G)=\left\{ 
\begin{array}{ll}
\Tr\left(P_{k}(\Acentwo)\right)+o_{p}(1) & \text{for} ~ k ~ \text{odd}.\\
\Tr\left(P_{k}(\Acentwo) \right)- \sum_{r=0:r~ \text{even}}^{k} P_{k}[r] \left[T_{r}- \binom{\frac{r}{2}+1}{2}\psi_{r}\right]+o_{p}(1) & \text{for} ~k ~\text{even}. 
\end{array}
\right.
\end{equation*}
where $T_{r}$'s have same asymptotic distribution under both null and alternative. 
Further, one can also prove 
\[
\frac{C_{n,k}(G)- \mu_{\kappa,k}}{\sqrt{2k}}\,\bigg|\,\mathbb{P}_{1,n} \stackrel{d}{\to} N(0,1).
\]
However, in the multiple block case $\mu_{\kappa,k}$ is no longer $t^{k}$ for some known $t$. 
In order to prove the consistency of our tests, it is enough to show that when $p_n > q_n$, above the Kesten-Stigum threshold, $\mu_{\kappa,k} \ge t_{\kappa}^{k-1}$ where $t_{\kappa} >1$. 
When $p_n < q_n$, we replace $\mu_{\kappa,k} $ with $(-1)^k \mu_{\kappa,k}$ while all the others remain the same.
In what follows, all expectation and variance are taken under $H_{1,\kappa}$ as defined in \eqref{eq:alt-kappa}.

\medskip

Observe that in this case  
\[
\pnav= \frac{p_{n}+ (\kappa-1) q_{n}}{\kappa}.
\]
In what follows, we focus on the assortative case, i.e., $p_n > q_n$, as before, and indicate the key difference for the disassortative case whenever needed.
Now 
\begin{equation}
\begin{split}
&C_{n,k}(G)=\left(\frac{1}{n\pnav(1-\pnav)}\right)^{\frac{k}{2}} \sum_{i_0,i_1,\ldots,i_{k-1}} (x_{i_0,i_1}-\pnav)\ldots(x_{i_{k-1}i_0}-\pnav)\\
&= \left(\frac{1}{n\pnav(1-\pnav)}\right)^{\frac{k}{2}}\sum_{i_0,i_1,\ldots,i_{k-1}} (x_{i_0,i_1}-p_{i_0,i_1}+p_{i_0,i_1}-\pnav)\ldots(x_{i_{k-1}i_0}-p_{i_{k-1},i_{k}}+p_{i_{k-1},i_{k}}-\pnav)\\
&=\left(\frac{1}{n\pnav(1-\pnav)}\right)^{\frac{k}{2}}\sum_{i_0,i_1,\ldots,i_{k-1}}(x_{i_0,i_1}-p_{i_0,i_1}+L(\sigma_{i_0},\sigma_{i_1}))\ldots(x_{i_{k-1}i_0}-p_{i_{k-1},i_{k}}+L(\sigma_{i_{k-1}},\sigma_{i_k}))\\
&=\left(\frac{1}{n\pnav(1-\pnav)}\right)^{\frac{k}{2}}\left[\sum_{i_0,i_1,\ldots,i_{k-1}}(x_{i_0,i_1}-p_{i_0,i_1})\ldots (x_{i_{k-1}i_0}-p_{i_{k-1},i_{k}}) + \prod_{j=0}^{k-1} L(\sigma_{i_j},\sigma_{i_{j+1}})\right]+ V_{n,k,\kappa}.
\end{split}
\end{equation}
Here $V_{n,k,\kappa}$ is the random variable corresponding to all the cross terms,
\[
L(\sigma_{1},\sigma_{2})= \left\{
\begin{array}{ll}
\frac{\kappa-1}{\kappa}(p_n-q_n) & \text{if} ~ \sigma_{1}= \sigma_{2}\\
- \frac{1}{\kappa}(p_n-q_n) & \text{otherwise},
\end{array}
\right.
\]
 $p_{i,j}=p_n$ if $\sigma_{i}=\sigma_{j}$ and $q_n$ otherwise.
Let
\[
D_{n,k,\kappa}:=\left(\frac{1}{n\pnav(1-\pnav)}\right)^{\frac{k}{2}}\sum_{i_0,i_1,\ldots,i_{k-1}}(x_{i_0,i_1}-p_{i_0,i_1})\ldots (x_{i_{k-1}i_0}-p_{i_{k-1},i_{k}}).
\]
It can be shown that the random variable $\frac{D_{n,k,\kappa}}{\sqrt{2k}}$ have asymptotically $N(0,1)$ distribution by arguments similar to the proof of Proposition 4.1 part (i) in \cite{Ban16}. 
In addition, arguments similar to the proof of Theorem \ref{thm:alt} part (i) will show that $\E\left[V_{n,k,\kappa}\right]=0$ and $\Var\left[ V_{n,k,\kappa} \right] \to 0$ as $n \to \infty$.
Arguments similar to the proof of (\ref{eqn:prodassign}) further show 
\[ 
\prod_{j=0}^{k-1} L(\sigma_{i_j},\sigma_{i_{j+1}})= \left\{
\begin{array}{ll}
\left| \prod_{j=0}^{k-1} L(\sigma_{i_j},\sigma_{i_{j+1}}) \right| & \text{if} ~ p_{n}>q_{n}\\
(-1)^{k} \left| \prod_{j=0}^{k-1} L(\sigma_{i_j},\sigma_{i_{j+1}}) \right| & \text{if} ~ p_{n} < q_{n}.
\end{array}
\right.
\] 
 irrespective of the value of $\sigma_{i_{j}}$'s.
 So the proof will be complete if we can show, 
 \begin{equation}\label{boundmulti}
 \frac{ \sum_{i_{0},\ldots,i_{k-1}}\left| \prod_{j=0}^{k-1} L(\sigma_{i_j},\sigma_{i_{j+1}}) \right|}{\left({n\pnav(1-\pnav)}\right)^{\frac{k}{2}}}\ge \frac{ \sum_{i_{0},\ldots,i_{k-1}}\left| \prod_{j=0}^{k-1} L(\sigma_{i_j},\sigma_{i_{j+1}}) \right|}{\left({n\pnav}\right)^{\frac{k}{2}}} \ge t_{\kappa}^{k-1}
 \end{equation}
 with high probability for some $t_{\kappa} >1$.

To prove \eqref{boundmulti},
let for any $l \in \{1,\ldots,\kappa \}$, $N_{l}$ be the number of nodes having label $l$. Obviously, $\sum_{l=1}^{\kappa} N_{l}= n$.
 As the labels $\sigma_{u}$ have been assigned uniformly and independently, for any $l$
\[
N_{l} \sim \mathrm{Bin}\left(n,\frac{1}{\kappa} \right).
\] 
By CLT, we have for any $\epsilon >0$ and for any $l$ 
\[
\mathbb{P}\left[ \sqrt{\frac{\kappa-1}{\kappa^2}}\left| N_{l}- \frac{n}{\kappa} \right| \le  n^{\frac{1}{2}+\epsilon}  \right] = 1- o(1).
\]
As a consequence, 
\[
\mathbb{P} \left[ \bigcap_{i=1}^{\kappa}\left\{\sqrt{\frac{\kappa-1}{\kappa^2}}\left| N_{l}- \frac{n}{\kappa} \right| \le  n^{\frac{1}{2}+\epsilon} \right\}\right]= 1- o(1).
\]
Let $\mathrm{Ev}$ be the indicator function of the event
\[
\bigcap_{i=1}^{\kappa} \left\{\sqrt{\frac{\kappa-1}{\kappa^2}}\left| N_{l}- \frac{n}{\kappa} \right| \le  n^{\frac{1}{2}+\epsilon} \right\}.
\]
From now on, we shall only consider the case when $\mathrm{Ev}=1$.
It is easy to observe that when $\mathrm{Ev}=1$, for any $0<\epsilon<1$ and $k= o\left(\min\left(\log(np_{n,\mathrm{av}}), \sqrt{\log{n}}\right)\right)$,
\begin{equation}\label{multilow}
\prod_{j=1}^{k} \left( \min_{l}N_{l}- j +1 \right) = (1+o(1))\left(\frac{1}{\kappa}\right)^{k}n^{k}
\end{equation}
and 
\begin{equation}\label{multiup}
 \left( \max_{l}N_{l} \right)^{k} = 
 (1+o(1))\left(\frac{1}{\kappa}\right)^{k}n^{k}.
\end{equation}
We now proceed by induction. 
Without loss of generality we can take $p_{n}>q_{n}$. The analysis for $p_{n}<q_{n}$ is exactly similar.
First we drop the index $i_{k}$ and evaluate  
\[
\sum_{i_{0},\ldots,i_{k-1}}\left| \prod_{j=0}^{k-2} L(\sigma_{i_j},\sigma_{i_{j+1}}) \right|.
\]
Here $i_{0},\ldots, i_{k-1}$ are all distinct. At first observe that for any $r\le k-1$ fixing the values $i_0,\ldots,i_{r}$, there are at least $N_{\sigma_{i_r}}-r$ many nodes which have the label same as $\sigma_{i_r}$ and at least $\sum_{l\neq \sigma_{i_r}} N_{l}-r$ many nodes which have the label different from $\sigma_{i_{r}}$.
As a consequence, 
\begin{equation}\label{multiinductionI}
\begin{split}
&\left(\prod_{j=0}^{r-1} \left|L(\sigma_{i_j},\sigma_{i_{j+1}})\right|\right)
\left[
\frac{\kappa-1}{\kappa}\left(N_{\sigma_{i_r}}-r\right)(p_n-q_n)+ \frac{1}{\kappa}\left(\sum_{l\neq \sigma_{i_r}} N_{l}-r\right)(p_n-q_n) 
\right]\\
&\le \sum_{j \notin \{ i_{0},\ldots,i_{r} \}}\prod_{j=0}^{r-1} \left|L(\sigma_{i_j},\sigma_{i_{j+1}})\right|\left|L(\sigma_{i_r},\sigma_{j}) \right|\\
&\le \left(\prod_{j=0}^{r-1} \left|L(\sigma_{i_j},\sigma_{i_{j+1}})\right|\right)
\left[ 
\frac{\kappa-1}{\kappa}\left(N_{\sigma_{i_r}}\right)(p_n-q_n)+ \frac{1}{\kappa}\left(\sum_{l\neq \sigma_{i_r}} N_{l}\right)(p_n-q_n) 
\right]
\end{split}
\end{equation} 
When $\mathrm{Ev}=1$, using (\ref{multilow}) and (\ref{multiup}) we can write (\ref{multiinductionI}) as 
\begin{equation}
\begin{split}
&\sum_{j \notin \{ i_{0},\ldots,i_{r} \}}\prod_{j=0}^{r-1} \left|L(\sigma_{i_j},\sigma_{i_{j+1}})\right|\left|L(\sigma_{i_r},\sigma_{j}) \right|\\
& = (1+o(1)) n \left(\prod_{j=0}^{r-1} \left|L(\sigma_{i_j},\sigma_{i_{j+1}})\right|\right)\left( \frac{p_{n}-q_{n}}{\kappa}\right)\left( \frac{2(\kappa-1)}{\kappa} \right)\\
& > (1+o(1)) \sqrt{n\pnav} \left( \frac{2(\kappa-1)}{\kappa} \right)\left(\prod_{j=0}^{r-1} \left|L(\sigma_{i_j},\sigma_{i_{j+1}})\right|\right).
\end{split}
\end{equation}  
As a consequence,
\[
\sum_{i_{0},\ldots,i_{k-1}}\left| \prod_{j=0}^{k-2} L(\sigma_{i_j},\sigma_{i_{j+1}}) \right| > (1+o(1)) \left(n\pnav\right)^{\frac{k-1}{2}} n \left( \frac{2(\kappa-1)}{\kappa} \right)^{k-1}.
\]
Here the additional $n$ factor comes due to the summation over $i_0$.
Let 
\[
L_{l_{1},l_{2}}:= \sum_{i_{0},\ldots,i_{k-1}: \sigma_{i_0}=l_{1},\sigma_{i_{k-1}}=l_2}\left| \prod_{j=0}^{k-2} L(\sigma_{i_j},\sigma_{i_{j+1}}) \right|.
\]
By symmetry it is easy to check that 
 \[
 L_{l_{1},l_{2}}=
 \left\{
 \begin{array}{ll}
 (1+o(1)) L_{k,1} \left(n\pnav\right)^{\frac{k-1}{2}} n, & \text{if} ~ l_{1}=l_{2},\\
 (1+o(1)) L_{k,2}  \left(n\pnav\right)^{\frac{k-1}{2}} n, & \text{otherwise}.
 \end{array}
 \right.
 \]
 Here $L_{k,1}$ and $L_{k,2}$ are some known number depending only on $k$.
 As a consequence,
 \[
 (1+o(1))\left[ \kappa L_{k,1} + \kappa(\kappa-1) L_{k,2} \right]> 
 \left[ \frac{2(\kappa-1)}{\kappa} \right]^{k-1}.
 \] 
 However 
 \begin{equation}
 \begin{split}
 &\sum_{i_{0},\ldots,i_{k-1}}\left| \prod_{j=0}^{k-2} L(\sigma_{i_j},\sigma_{i_{j+1}}) \right|\\
 &= (1+o(1)) \left(n\pnav\right)^{\frac{k-1}{2}} n \left( \kappa L_{k,1}\frac{\kappa - 1 }{\kappa}(p_n-q_n) +  \kappa(\kappa-1) L_{k,2} \frac{p_n-q_n}{\kappa}\right)\\
 &= (1+o(1)) \left(n\pnav\right)^{\frac{k-1}{2}}n\frac{(p_n-q_n)}{\kappa} \left[ L_{k,1}\kappa(\kappa-1) + L_{k,2} \kappa(\kappa-1) \right]\\
 & \ge (1+o(1))\left(n\pnav\right)^{\frac{k-1}{2}}n\frac{(p_n-q_n)}{\kappa} \left[ \kappa L_{k,1} + \kappa(\kappa-1) L_{k,2} \right]\\
 & > (1+o(1))\left(n\pnav\right)^{\frac{k-1}{2}} \sqrt{n\pnav} \left[ \frac{2(\kappa-1)}{\kappa} \right]^{k-1}\\
 & > (1+o(1))\left(n\pnav\right)^{\frac{k}{2}} \left[ \frac{2(\kappa-1)}{\kappa} \right]^{k-1}.
 \end{split}
 \end{equation}
We complete the proof by noting that $\frac{2(\kappa-1)}{\kappa}>1$ for every $\kappa\ge 3$.
 \hfill{$\square$}

\subsection{Proof of \eqref{eq:cheby-facts}}
Consider the first identity first. Note that by definition for any integer $r\geq 0$,
\begin{equation*}
\psi_{2r} = \int_{-2}^2x^{2r}\frac{\sqrt{4-x^2}}{2\pi}\,\mathrm{d}x.
\end{equation*}
Hence for any $k\geq 2$
\begin{align*}
\sum_{r=0}^k P_{2k}[2r]\psi_{2r}
& = \int_{-2}^2 \sum_{r=0}^k P_{2k}[2r]x^{2r} \frac{\sqrt{4-x^2}}{2\pi}\mathrm{d}x 
= \int_{-2}^2 P_{2k}(x) \frac{\sqrt{4-x^2}}{2\pi}\mathrm{d}x\\
& = \int_{-2}^2 P_{2k}(x)[P_0(x) - P_2(x)] \frac{1}{2\pi\sqrt{4-x^2}}\mathrm{d}x\\
& = 0.
\end{align*}
Here, the second equality holds since the odd power terms in $P_{2k}$ are all zeros.
The third equality holds since $P_0(x)=2$ and $P_2(x) = x^2-2$.
The last equality holds since the $P_j$'s are mutually orthogonal with respect to the weight function $1/\sqrt{4-x^2}$ on $[-2,2]$ by their definitions and the orthogonality of Chebyshev polynomials with respect to $1/\sqrt{1-x^2}$ on $[-1,1]$.

Turn to the second identity. we have
\begin{align*}
\sum_{r=1}^k P_{2k}[2r] (2r) \psi_{2r}
& = \int_{-2}^2 \left[\frac{\mathrm{d}}{\mathrm{d}x} P_{2k}(x) \right] x \frac{\sqrt{4-x^2}}{2\pi}\mathrm{d}x \\
& = P_{2k}(x) x \frac{\sqrt{4-x^2}}{2\pi}\bigg|_{-2}^{2}
- \int_{-2}^2  P_{2k}(x) \frac{\mathrm{d}}{\mathrm{d}x}\left[x \frac{\sqrt{4-x^2}}{2\pi} \right]\mathrm{d}x\\
& = 0 + \int_{-2}^2 P_{2k}(x) P_2(x) \frac{1}{\pi\sqrt{4-x^2}}\mathrm{d}x = 0.
\end{align*}
Here we have used the fact that $P_2(x) = x^2-2$. This completes the proof.

\bibliographystyle{abbrvnat}
\bibliography{HYP_SBM}

\end{document}